\documentclass[11pt]{amsart}
\usepackage{amsmath,amssymb}
\usepackage{graphicx}
\usepackage{hyperref}

\usepackage{amsfonts,amscd,latexsym,epsfig, psfrag}

\usepackage{hyperref}
\usepackage[all]{xy}

\newcommand{\er}{{\Diamond}}

\newcommand{\orr}{{\mathfrak o}}
\newcommand{\Fix}{{\rm Fix}}
\newcommand{\cl}{{\rm cl}}
\newcommand{\uEE}{{\und{\E}}}
\newcommand{\E}{{\mathbb E}}
\newcommand{\s}{{\mathfrak s}}

\newcommand{\Tx}{{\Tilde x}}

\newcommand{\TGa}{{\Tilde \Ga}}
\newcommand{\Tde}{{\Tilde \de}}

            \newcommand{\TLa}{{\Tilde\La}}

   \newcommand{\ubE}{{\und \bE}}
  \newcommand{\ubB}{{\und \bB}}
 
  \newcommand{\uVv}{{\und \Vv}}

  \newcommand{\us}{{\und{\mathfrak s}}}

 \newcommand{\ux}{{\underline{x}}}
  
    \newcommand{\urho}{{\underline{\rho}}}
 \newcommand{\uy}{{\underline{y}}}
  
    \newcommand{\uz}{{\underline{z}}}

  \newcommand{\uA}{{\underline{A}}}
    \newcommand{\uC}{{\underline{C}}}
  
  \newcommand{\uU}{{\underline{U}}}
  
      \newcommand{\uN}{{\und N}}
    \newcommand{\ubK}{{\underline{\bK}}}
   \newcommand{\uV}{{\underline{V}}}
  \newcommand{\uW}{{\underline{W}}}
   \newcommand{\uphi}{{\underline{\phi}}}
      \newcommand{\upsi}{{\underline{\psi}}}
\newcommand{\uKk}{{\underline{\Kk}}}

\newcommand{\Gg}{{\mathcal G}}
\newcommand{\Zz}{{\mathcal Z}}

\newcommand{\Aa}{{\mathcal A}}

\newcommand{\Mor}{{\rm Mor}}
\newcommand{\Id}{{\rm Id}}
\newcommand{\Obj}{{\rm Obj}}

\newcommand{\Trho}{{\Tilde \rho}}
\newcommand{\Tphi}{{\Tilde \phi}}

\newcommand{\und}{\underline}
\newcommand{\ul}{\underline}
\renewcommand{\Hat}{\widehat}

\newcommand{\Cc}{{\mathcal C}}
\newcommand{\Kk}{{\mathcal K}}

\newcommand{\im}{{\rm im\,}}
\newcommand{\less}{{\smallsetminus}}

\newcommand{\TU}{{\Tilde U}}
\newcommand{\TW}{{\Tilde W}}
\newcommand{\p}{{\partial}}
\newcommand{\al}{{\alpha}}

\newcommand{\be}{{\beta}}

\newcommand{\eps}{{\varepsilon}}
\newcommand{\de}{{\delta}}

\newcommand{\ga}{{\gamma}}
\newcommand{\Ga}{{\Gamma}}
\newcommand{\io}{{\iota}}

\newcommand{\La}{{\Lambda}}
\newcommand{\si}{{\sigma}}

\newcommand{\Br}{{\rm Br}}

\newcommand{\Bb}{{\mathcal B}}
\newcommand{\Ww}{{\mathcal W}}

\newcommand{\ov}{\overline}

\newcommand{\id}{{\rm id}}
\newcommand{\rd}{{\rm d}}
\newcommand{\rT}{{\rm T}}

\renewcommand{\Tilde}{\widetilde}

\newcommand{\TV}{{\Tilde V}}

\newcommand{\coker}{{\rm coker\,}}

\newcommand{\Ii}{{\mathcal I}}

\newcommand{\Vv}{{\mathcal V}}

\newcommand{\N}{{\mathbb N}}

\newcommand{\Q}{{\mathbb Q}}
\newcommand{\R}{{\mathbb R}}
\newcommand{\C}{{\mathbb C}}

\newcommand{\Z}{{\mathbb Z}}

\newcommand{\Hom}{{\rm Hom}}

\newcommand{\Hh}{{\mathcal H}}

\newcommand{\SSS}{{\smallskip}}

\newcommand{\bA}{{\bf A}}
\newcommand{\bB}{{\bf B}}

\newcommand{\bG}{{\bf G}}

\newcommand{\bE}{{\bf E}}
\newcommand{\bK}{{\bf K}}

\newcommand{\bX}{{\bf X}}

\newcommand{\bz}{{\bf z}}
\newcommand{\bZ}{{\bf Z}}

\newtheorem{theorem}{Theorem}[subsection]
\newtheorem{thm}[theorem]{Theorem}

\newtheorem{lemma}[theorem]{Lemma}

\newtheorem{prop}[theorem]{Proposition}
\newtheorem{definition}[theorem]{Definition}
\newtheorem{defn}[theorem]{Definition}
\newtheorem{example}[theorem]{Example}

\newtheorem{remark}[theorem]{Remark}
\newtheorem{rmk}[theorem]{Remark}

\numberwithin{figure}{subsection}
\numberwithin{equation}{subsection}

\newcommand{\MS}{{\medskip}}

\newcommand{\NI}{{\noindent}}

\newcommand{\ti}{\tilde}
\newcommand{\Ti}{\widetilde}

\newcommand{\pr}{{\rm pr}}

\newcommand{\lm}{\Lambda^{\rm max}\,}

   \newcounter{qcounter}

\newenvironment{enumilist}
   { \begin{list} {\rm (\roman{qcounter})\;}{\usecounter{qcounter}
     \setlength{\itemsep}{.5ex} \setlength{\leftmargin}{4.2ex} } }
   { \end{list} }

\newenvironment{itemlist}
   { \begin{list} {$\bullet$}
         {  \setlength{\itemsep}{.5ex} \setlength{\leftmargin}{2.5ex} } }
   { \end{list} }

\newcommand*{\longhookleftarrow}{\ensuremath{\leftarrow\joinrel\relbar\joinrel\rhook}}
\newcommand*{\longhookrightarrow}{\ensuremath{\lhook\joinrel\relbar\joinrel\rightarrow}}

\newcommand{\leftsub}[2]{{\vphantom{#2}}_{#1}{#2}}

\newcommand\quotient[2]{
        \mathchoice
            {
                \text{\raise1ex\hbox{$#1$}\Big/\lower1ex\hbox{$#2$}}%
            }
            {
                #1\,/\,#2
            }
            {
                #1\,/\,#2
            }
            {
                #1\,/\,#2
            }
    }

\newcommand\quot[2]{
                \text{\raise1ex\hbox{$#1$}/\lower1ex\hbox{$\scriptstyle#2$}}
  }

\newcommand\quo[2]{
                \text{\raise1ex\hbox{$#1\!\!$}/\lower1ex\hbox{$\!\scriptstyle#2$}}
  }

\newcommand\qu[2]{
                \text{\raise.8ex\hbox{$\scriptstyle#1\!$}/\lower.8ex\hbox{$\!\scriptstyle#2$}}
  }

\newcommand\qq[2]{
                \text{\raise.8ex\hbox{$#1\!$}/\lower.8ex\hbox{$#2$}}
}

 \title{Smooth Kuranishi atlases with isotropy}
 \author{Dusa McDuff}
  \thanks{partially supported by NSF grants DMS1308669, DMS 0844188, and DMS 0905191 }
\address{Department of Mathematics,
 Barnard College, Columbia University}
\email{dusa@math.columbia.edu}
\author{Katrin Wehrheim}
\address{Department of Mathematics, UC Berkeley}
\email{katrin@math.berkeley.edu}
\keywords{virtual fundamental cycle, virtual fundamental class, pseudoholomorphic curve, Kuranishi   structure, Gromov--Witten invariant, transversality, weighted branched manifold}
\subjclass[2010]{53D35,53D45,54B15,57R17,57R95}

\begin{document}
\maketitle
\begin{abstract}
Kuranishi structures were introduced in the 1990s by Fukaya and Ono for the purpose of assigning a virtual cycle to moduli spaces of pseudoholomorphic curves that cannot be regularized by geometric methods. 
Their core idea was to build such a cycle by patching local finite dimensional reductions, given by smooth sections that are equivariant under a finite isotropy group.

Building on our notions of topological Kuranishi atlases and perturbation constructions in the case of trivial isotropy, we develop a theory of Kuranishi atlases and cobordisms that transparently resolves the challenges posed by nontrivial isotropy. 
We assign to a cobordism class of weak Kuranishi atlases both a virtual moduli cycle (VMC -- a cobordism class of weighted branched manifolds) and a virtual fundamental class (VFC -- a Cech homology class). 
\end{abstract}

\tableofcontents
\section{Introduction}

\subsection{Overview}\label{ss:over}
\hspace{1mm}\\ \vspace{-3mm}

This is the third in a series \cite{MW1,MW2} of papers that construct a fundamental class for compact spaces $X$ that are modeled locally by the zero set of smooth sections $s_i:U_i\to E_i$ in finite rank bundles over finite dimensional manifolds.
While these obstruction bundles have fixed index, they may have varying rank, and thus an ambient space $\bigcup U_i / \!\!\sim$ naively constructed from the ambient manifolds of the local zero sets $s_i^{-1}(0)$ modulo transition data is lacking all topological controls (Hausdorffness, local compactness, in fact existence) that are needed for a perturbative construction ``$[X]:=\bigcup (s_i+\nu_i)^{-1}(0) / \!\!\sim \,$'' of the fundamental class.
Moreover, most interesting cases involve nontrivial isotropy groups that are captured in the local charts as finite symmetry groups $\Gamma_i$ of the sections $s_i$, so that $X$ is locally modeled by the quotients $s_i^{-1}(0)/\Gamma_i$.

Pioneered by Fukaya et al \cite{FO,FOOO}, this problem has been considered by symplectic topologists since the 1990s as a tool for ``counting curves'', i.e.\ assigning homological information to moduli spaces of pseudoholomorphic curves, such as the Gromov-Witten moduli spaces (in which isotropy arises from multiply covered curves).
In the case of trivial isotropy, a comprehensive solution was developed in \cite{MW1,MW2} by introducing notions of Kuranishi atlases which on the one hand can in practice be constructed from moduli spaces, and on the other hand have sufficient compatibility between the local models for the construction of a virtual fundamental class. 
This paper extends these techniques to the case of nontrivial isotropy, proving the following result.

\MS
 \NI {\bf Theorem A.}\,\,{\it
Let $\Kk$ be 
an oriented, $d$-dimensional, additive, smooth weak 
Kuranishi atlas on a compact metrizable space $X$. Then $\Kk$ determines 
\begin{itemlist}
\item
a {\bf virtual moduli cycle (VMC)} as cobordism class of weighted branched manifolds,
\item
a {\bf virtual fundamental class (VFC)} $[X]^{vir}_\Kk\in \check{H}_d(X;\Q)$ in \v{C}ech homology,
\end{itemlist}
both of which depend only on the cobordism class of $\Kk$.}
\MS

A more precise statement that also applies when $\Kk$ is a cobordism from an atlas $\Kk^0$ on $X^0$ to an atlas $\Kk^1$ on $X^1$  is given  in Theorem~\ref{thm:VMCF}.
The guiding idea of a Kuranishi atlas $\Kk$ is to start with a family of basic charts $(\bK_i)_{i=1,\dots,N}$, where 
each basic chart
$$
\bK_i  = \bigl(U_i, E_i,\Ga_i, s_i,\psi_i\bigl)
$$
is a tuple consisting of a domain $U_i$, obstruction space $E_i$, group $\Ga_i$, section $s_i:U_i\to E_i$, and footprint map 
$\psi_i:s_i^{-1}(0)\to  X$ inducing a homeomorphism from  $s_i^{-1}(0)/\Ga_i$ onto the ``footprint", an open subset $F_i\subset X$ such that $(F_i)_{i=1,\dots,N}$ covers $X$.  
The compatibility of these charts then involves transition charts 
$\bK_I = \bigl(U_I, E_I,\Ga_I, s_I,\psi_I\bigl)$ of the same type as the basic charts, but with $I\subset\{1,\dots,N\}$
such that  $F_I := \cap_{i\in I}F_i\ne \emptyset$.
Finally, the basic and transition charts are related by coordinate changes from $\bK_I$ to $ \bK_J$ whenever $I\subset J$. 
This gives rise to an ``\'etale-like" category $\bB_\Kk$ whose space of objects is
 $\bigsqcup_I U_I$, and whose morphisms are determined by the local group actions and the coordinate changes.    
The  category $\bB_\Kk$ is not a groupoid since some morphisms (those relating the different charts) are not invertible.  On the other hand, its spaces of objects and morphisms are very closely controlled, which enables us to carry out various geometric constructions, in particular the construction of perturbations, very explicitly. 
The realization $|\Kk|$ of $\bB_{\Kk}$  (the space of objects modulo the equivalence relation generated by the morphisms) is much larger than $X$, though it does contain a homeomorphic image of $X$ formed from the zero sets of the local sections $s_I$ that are induced by the Cauchy--Riemann operator.  As in \cite{MW2}, the class $[X]^{vir}_\Kk$ is constructed from the zero sets of suitable 
perturbations $\s_\Kk+\nu$ of the basic section $\s_\Kk = (s_I)$ of $\Kk$.

Even if $X$ is an orbifold so that no obstruction 
spaces are needed, our formulations are new.\footnote
{Our construction was outlined in \cite{Mcl}.  In \cite{Pardon}, Pardon independently takes a similar approach to handling the isotropy groups.}
Rather than being given by inclusions $U_I\supset U_{IJ}\hookrightarrow U_J$ as in the case with
trivial isotropy, our notion of coordinate changes in the presence of isotropy involves 
 equivariant covering maps $\Trho_{IJ}: (\TU_{IJ}, \Ga_J)\to (U_{IJ}, \Ga_I) \subset (U_I, \Ga_I)$ where $\TU_{IJ}$ is a suitable subset of the domain $U_J$ and $\TU_{IJ}\to U_{IJ}$ is a principal $\Ga_J/\Ga_I$-bundle. 
As the following result from  \cite{Mcorb} shows,  every orbifold has a structure of this kind.
\MS

\NI {\bf Proposition.}\,\, {\it   Every compact orbifold $Y$ has an orbifold atlas $\Kk$ with trivial obstruction spaces whose associated groupoid 
$\bG_{\Kk}$ is an orbifold structure on $Y$.   Moreover, there is a bijective correspondence between commensurability classes of such Kuranishi atlases and Morita equivalence classes of ep groupoids.}
\MS

To apply the above theory to moduli spaces  $X$ that arise in  geometric examples, one needs to develop methods for
constructing Kuranishi atlases on such $X$.  Some parts of this construction were detailed in the 2012 preprint  \cite{MW0}, and now appear in \cite{MW2}.  They will be extended in \cite{MW:gw} to include multiply covered curves (and hence nontrivial isotropy) as well as nodal curves.
Both McDuff~\cite{Mcn} and Pardon~\cite{Pardon} outline the needed construction for moduli spaces of closed stable maps, 
though neither approach is sufficient to give the smooth charts whose existence is assumed in the current paper.
In \cite{MW:gw} we will combine the same setup with an implicit function theorem from polyfold theory \cite{HWZ-abstract} to obtain compatible choices of smooth structures near nodal curves.
An alternative approach is to extend the VMC/VFC construction to less smooth sections.
In fact, Castellano \cite{Cast} proves a gluing theorem for Gromov--Witten moduli spaces that allows the construction of stratified smooth Kuranishi atlases
with $\Cc^1$-differentiability across strata, to which our construction applies with minor modifications.
He moreover shows that the resulting genus zero Gromov--Witten invariants satisfy the standard axioms.

\subsection
{Outline of the  construction}
\hspace{1mm}\\ \vspace{-3mm}

 This paper contains all relevant definitions and  a fair amount of review so 
 that it can be read independently of the previous papers in this series.
This outline will also be rather brief since the earlier papers 
give extensive explanation and justification for our approach:
\begin{itemlist}
\item[-] 
The first part of \cite{MW2} is a general discussion of different approaches to 
regularizing moduli spaces -- e.g.\ as VMC/VFC -- 
and explains important analytic background.
\item[-]  The paper  \cite{MW1} starts with an overview of the 
topological challenges that need to be addressed in constructing  
a VMC/VFC, and then
proves the basic  topological results needed to show that a filtered weak Kuranishi atlas
determines a tame Kuranishi atlas $\Kk$, well defined up to cobordism, whose realization $|\Kk|$ is Hausdorff and metrizable and contains
a homeomorphic copy of  the moduli space $X$.
\item[-]   The 
second part of \cite{MW2} carries out the full construction of the 
VMC as the zero set of a suitable perturbation of the canonical section $\s_\Kk$ in the case of trivial isotropy.
\end{itemlist}

We now discuss the  main steps in the construction below in more detail, highlighting the new features needed to deal with nontrivial isotropy.

\begin{itemlist}\item
 In order to simplify the abstract discussion, we decided to give a rather narrow definition of a Kuranishi atlas $\Kk$.
 Thus the domains of both the basic and transition charts are  
 group quotients $(U_I,\Ga_I)$, and the coordinate changes are determined by rather special
 equivariant covering maps $(\TU_{IJ},\Ga_J)\to (U_{IJ}, \Ga_I)$.  The basic theory  is set up in \S\ref{ss:grquot}; see 
 in particular  Definitions~\ref{def:cover} and Lemma~\ref{le:vep}.
If there were a need, one could  no doubt
replace these group quotients by more general \'etale groupoids and use more general covering maps and obstruction bundles, 
at the expense of revisiting the construction of perturbations.

\item Smooth atlases and coordinate changes are defined in \S\ref{ss:chart} and \S\ref{ss:Ks}.  Though in general the definitions are  similar  to those in the case with trivial isotropy, there is an important difference in the notion of coordinate change:  When $I\subset J$  this is 
 now given by a covering map from an appropriate  submanifold $\TU_{IJ}$ of the domain of the higher dimensional 
 domain $U_J$ onto an open subset $U_{IJ}$ of the lower dimensional domain $U_I$.  If the isotropy groups  are trivial, this map is a diffeomorphism with inverse equal to 
 the coordinate changes $\phi_{IJ}:U_{IJ}\to U_J$ considered in \cite{MW1,MW2}.
 Another small difference is that we build in the notion of additivity since at least some version of this is needed for the taming construction discussed below. (In some situations, for example when considering products, this formulation is too rigid; for appropriate generalizations see \cite{Mcn}.)

 \item 
An important feature of our definitions is that the quotients $\uU_I: = \qq{U_I}{\Ga_I}$ fit together to form an {\bf  intermediate atlas}, 
which Lemma~\ref{le:Ku3} shows to be 
a filtered topological atlas in the sense of ~\cite{MW1}.
In particular it has an associated category $\bB_{\uKk}$ with space of objects the orbifold $\Obj_{\bB_{\uKk} }: = \bigsqcup_I \uU_I$, 
and identical realization $|\uKk|=|\Kk|$.

\item
One difficulty in constructing a VFC for a given moduli space $X$ is that
in practice one cannot usually construct an atlas on $X$.  Instead one constructs a weak atlas, which is like an atlas except that one has less control of the domains of the charts and coordinate changes;
c.f.\ the various {\bf cocycle conditions}  discussed in Definition~\ref{def:cocycle} and Lemma~\ref {le:compos0}.
But a weak atlas does not even define a a category, let alone one 
 whose realization $|\bB_\Kk|=: |\Kk|$ has good topological properties. For example, we would like $|\Kk|$ to be Hausdorff and (in order to make local constructions possible) for the projection $\pi_\Kk: U_I\to |\Kk| $ to be a homeomorphism to its image.  
 Theorem~\ref{thm:K}  summarizes the main topological facts about $\Kk$ that are needed for subsequent constructions.  
 We achieve these via {\bf shrinking} and {\bf taming}.  
Our definitions were designed so that all
 the topological constructions of \cite{MW1}, such as  the taming, cobordism and reduction constructions, 
 apply to the  intermediate atlas $\uKk$ and then lift to  
 $\Kk$  because the quotient maps 
$U_I\to \uU_I$  are  proper.   However, we do need to take some care with the proof
of the linearity properties of the projection $\pr:|\bE_\Kk|\to |\bB_\Kk|$.

\item
Another important part of Theorem~\ref{thm:K}  is the claim that any two tame shrinkings of a weak atlas $\Kk$ are {\bf concordant}, i.e.\  cobordant over $[0,1]\times X$, 
which is required to show independence of the VMC/VFC from the choice of shrinking.
In  \S\ref{ss:cob} we give the precise definition of a cobordism atlas.  This is an immediate generalization of the notion of cobordism in \cite{MW1,MW2}, and the relevant proofs generalize easily.

 \item 
Given a weak atlas, the taming procedure gives us two categories $\bB_{\Kk}$ and $\bE_{\Kk}$ with a projection functor $\pr: \bE_{\Kk}\to \bB_{\Kk}$ and section functor $\s_\Kk: \bB_{\Kk}\to \bE_{\Kk}$.  However, even when the isotropy is trivial,
 the category has too many morphisms (i.e.\ the 
 chart domains 
 overlap too much) for us to be able to construct a 
 perturbation $\nu: \bB_{\Kk}\to \bE_{\Kk}$  that is transverse to $0$ (written $\s_\Kk+\nu\pitchfork 0$).  
We therefore  pass to a full subcategory 
$\bB_{\Kk}|_{\Vv}$ of $\bB_{\Kk}$ with objects $\Vv: = \bigsqcup V_I$ that does support suitable perturbations $\nu: 
\bB_{\Kk}|_{\Vv}\to \bE_{\Kk}|_{\Vv}$.  This subcategory $\bB_{\Kk}|_{\Vv}$ is called a {\bf reduction} of $\Kk$; c.f.\ Definition~\ref{def:vicin}.  Constructing it is akin to passing from the covering of a triangulated space by the stars of its vertices to the covering by the stars of its first barycentric subdivision.  
Again this construction can be done at the level of the intermediate category, so that the methods of \cite{MW1} immediately give us the required reductions.

\item
In the presence of nontrivial isotropy, we may still not be able to construct a transverse perturbation $\nu:\bB_{\Kk}|_{\Vv} \to  \bE_{\Kk}|_{\Vv}$ as a functor, since local perturbations $\nu_I$ are required to be $\Gamma_I$-equivariant.
In general, this can be resolved by using multivalued perturbations. 
Our setup allows for a simplified approach: We define perturbations $\nu=(\nu_I)_{I\in \Ii_\Kk}$  to be families of maps that are compatible with the covering maps $\rho_{IJ}$ but need not be $\Gamma_I$-equivariant.
We show in \S\ref{ss:red} that this construction inherits enough equivariance to yield an \'etale category that represents the zero set of the perturbed section $\s_\Kk|_\Vv + \nu$, assuming that this is transverse to $0$. The remaining morphisms are then added back in at the expense of weighting functions, which give the perturbed zero set the structure of a {\bf weighted branched manifold}.
More precisely, we construct the perturbed zero set in Theorem~\ref{thm:zero}
as the Hausdorff realization $|\bZ^\nu|_\Hh$ of an \'etale (but non proper) category $\bZ^\nu$ whose space of objects has one component $Z_{I} = (s_I|_{V_I} + \nu)^{-1}(0)$ for each $I\in \Ii_\Kk$, and whose branching locus and weighting function are explicitly determined by the reduction $\Vv$ and the isotropy groups.  

\item
For the convenience of the reader we prove the needed results about weighted branched manifolds and cobordisms in Appendix \ref{ss:br}. 
Moreover, the short paper \cite{Mcorb} explains the construction of $\bZ^\nu$ in the orbifold case.  This is much simpler, since the obstruction spaces, and hence also the section $\s_\Kk, \nu$  are zero.

\item
Moreover, we must ensure that the perturbed zero sets are compact and unique up to cobordism.
As we show in Proposition~\ref{prop:ext} the rather intricate construction in \cite{MW2} carries through in the current situation without essential change.

\item 
In \S\ref{ss:orient}  we extend the notion of orientation  to atlases with nontrivial isotropy.
As in \cite{MW2}, we define the orientation line bundle of $\Kk$ in two equivalent ways,
showing in
Proposition~\ref{prop:orient} that the  bundle $\det \s_\Kk$ (with local bundles  $(\det s_I)_{I\in \Ii_\Kk}$)  is isomorphic to
$\La_\Kk$ (with local bundles  $(\lm U_I\otimes (\lm E_I)^*)_{I\in \Ii_\Kk}$).  Most of the needed proofs can again be quoted directly from \cite{MW2}.     Lemma~\ref{le:locorient} explains how these bundles are used to orient local zero sets of sections.
 
\item
The final step is  to build the  homology class $[X]_\Kk^{vir}\in  \check H_d(X;\Q)$ from the zero set
$(\s_\Kk|_\Vv+\nu)^{-1}(0)$.   Many of the details here are again the same as in \cite{MW2}.  In particular, we build a geometric representative $|\bZ^\nu|_\Hh$ for 
this class that maps to the precompact \lq\lq neighbourhood"\footnote
{
In fact, $\io_\Kk(X)$ does {\it not} have a compact neighbourhood in $|\Kk|$; 
we should think of  $|\Vv|$ 
as the closest we can come to such  a  neighbourhood.}   
$|\Vv| = \bigcup_I \pi_\Kk(V_I)\subset |\Kk|$ of $\io_\Kk(X) = |\s_\Kk^{-1}(0)|$, and then define   $[X]_\Kk^{vir}$ by taking an appropriate inverse limit
in rational \v Cech homology.
 \end{itemlist}

\medskip
\noindent
{\bf Acknowledgements:}
We would like to thank
Mohammed Abouzaid,
Kenji Fukaya,
Tom Mrowka,
Kaoru Ono,
Yongbin Ruan,
Dietmar Salamon,
Bernd Siebert,
Cliff Taubes,
Gang Tian,
and
Aleksey Zinger
for encouragement and enlightening discussions about this project.
We moreover thank MSRI, IAS, BIRS and SCGP for hospitality.

\section{Smooth Kuranishi atlases with isotropy}\label{s:chart}

In this section we extend the notions of smooth Kuranishi charts and 
transition data introduced in \cite{MW2} to nontrivial isotropy and then discuss cobordisms and taming.
The main result is Theorem~\ref{thm:K}.

Throughout this section we fix $X$ to be a compact metrizable space.
The main change from \cite{MW2} is that the domains of the charts are no longer smooth manifolds, 
but rather group quotients. We begin by setting up notation for the latter.
As in \cite[Remark~5.1.2]{MW2} we assume all manifolds are smooth and second countable.

\subsection{Group quotients}\label{ss:grquot}\hspace{1mm}\\ \vspace{-3mm}

\begin{defn}\label{def:gq}
A  {\bf group quotient} is a pair $(U,\Ga)$ consisting of a smooth manifold $U$ and a 
finite group $\Ga$ together with a smooth action $\Ga\times U \to U$.
We will denote the quotient space by
$$
\uU: = \qq{U}{\Ga},
$$
giving it the quotient topology, and write $\pi: U\to \uU$ for the associated projection.
Moreover, we denote the {\bf stabilizer} of each $x\in U$ by
$$
\Ga^x := \{\ga\in \Ga \,|\, \ga x = x \} \subset \Ga .
$$
\end{defn}

We could consider a group quotient as a topological category with space of objects $U$ and morphisms $U\times \Ga$, but in the interest
of simplicity will often avoid doing this.

Both the basic and transition charts of Kuranishi atlases will be group quotients, related by coordinate changes that are composites of the following kinds of maps.

\begin{defn}\label{def:inject}
Let $(U,\Ga), (U',\Ga')$  be group quotients.
A {\bf group  embedding} 
$$
(\phi, \phi^\Ga): (U,\Ga)\to (U',\Ga')
$$
is a smooth embedding $\phi: U\to U'$ that is 
equivariant with respect to an injective group homomorphism
$\phi^\Ga:\Ga\to \Ga'$
and 
induces an injection $\uphi:\uU\to \uU'$ on the quotient  spaces.
We call a group embedding {\bf equidimensional} if $\dim U=\dim U'$.
\end{defn}
 
In a Kuranishi atlas we often consider embeddings $(\phi, \phi^\Ga): (U,\Ga)\to (U',\Ga)$ where $\dim U <\dim U'$ and $\phi^\Ga:\Ga\to \Ga':=\Ga$ is the identity map. 
On the other hand, group quotients of the same dimension are usually related
either by restriction or by coverings as follows.

\begin{defn}\label{def:grprestr}
Let $(U,\Ga)$ be a group quotient and $\uV\subset \uU$ an open subset. 
 Then the {\bf restriction of $(U,\Ga)$ to $\uV$} is the group quotient $(\pi^{-1}(\uV), \Ga)$.
\end{defn}

Note that the inclusion $\pi^{-1}(\uV)\to U$ induces an equidimensional group embedding $(\pi^{-1}(\uV), \Ga)\to (U,\Ga)$ that covers the inclusion $\uV\to\uU$.
The third kind of map that occurs in a coordinate change is a group covering.  This notion
is less routine; notice in particular the requirement in (ii) that $\ker \rho^\Ga$ act freely.
Further, 
the two domains $\TU, U$ will necessarily have the same dimension since they are related by a regular covering $\rho$.

\begin{defn}\label{def:cover}
Let $(U,\Ga)$  be a group quotient. A {\bf group covering} of $(U,\Ga)$ is a tuple 
$(\TU,\TGa,\rho,\rho^\Ga)$
 consisting of
\begin{enumerate}
\item a surjective group homomorphism $\rho^\Ga: \TGa\to \Ga$,
\item a group quotient $(\TU,\TGa)$ where $\ker \rho^\Ga$ 
 acts freely,
\item
a regular covering $\rho: \TU \to U$ that is the quotient map $\TU\to\qu{\TU}{\ker \rho^\Ga}$ composed with a diffeomorphism $\qu{\TU}{\ker \rho^\Ga}\cong U$ that is equivariant with respect to the induced $\Ga = \im (\rho^\Ga)$ action on both spaces.
\end{enumerate}
Thus $\rho:\TU\to U$ is equivariant with respect to $\rho^\Ga: \TGa\to \Ga$ 
and $\rho^\Ga$ acts transitively on the fibers of $\rho$.
We denote by $\urho: \ul{\TU}\to \uU$ the induced map on quotients.
\end{defn}

Next, we establish some basic properties of group quotients, in particular the fact that coverings  induce homeomorphisms between the quotients.  
Here and subsequently we denote a precompact inclusion by $V\sqsubset U$.

 \begin{lemma}\label{le:vep}
Let $(U,\Ga)$ be a group quotient.
\begin{enumerate}
\item
The projection $\pi: U\to \uU$ is open, closed, and proper.
In particular, any precompact set $P\sqsubset\uU$ has precompact preimage $\pi^{-1}(P)\sqsubset U$,
Moreover, $\uU$ is a separable, locally compact metric space.
\item
Every point $x\in U$ has a neighbourhood $U_x$ that is invariant under $\Ga^x$ 
and is such that inclusion $U_x\hookrightarrow U$ induces a homeomorphism from $\qu{U_x}{\Ga^x}$ to 
 $\pi(U_x)$.
In particular, the inclusion $(U_x,\Ga^x)\to \bigl(\pi^{-1}(\pi(U_x)), \Ga\bigr)$ is a group embedding.
\item
If $(\TU,\TGa,\rho,
\rho^\Ga)$ is a group covering of $(U,\Ga)$, then $\urho:   \ul{\TU}\to \uU$ is a homeomorphism
and $\rho^\Ga$ induces isomorphisms between the stabilizers $\TGa^y\to \Ga^{\rho(y)}$ for all $y\in \TU$.

\end{enumerate}
\end{lemma}
\begin{proof}
Let $W\subset U$ be open.   
Then $\pi^{-1}(\pi(W)) = \bigcup_{\ga\in \Ga} \ga W$ is open since each 
$\ga W$ is the preimage, under the continuous action of $\ga^{-1}$,
of 
the open set $W$.
Hence, by definition of the quotient topology, $\pi(W)$ is open. This shows that $\pi$ is open. 
The same argument applied to the complement of a closed set shows that $\pi$ is closed.

To see that $ \pi$ is proper, consider a compact set $\und V\subset \uU$.  Given any open cover $(U_\al)_{\al\in A}$  of $\pi^{-1}(\und V)$, choose for each $x\in \pi^{-1}(\und V)$ an element $\al_x\in A$  such that $x\in U_{\al_x}$. Then for each 
$\ux\in \und V$ define
$$ 
\textstyle
\uW_{\ux}: = \bigcap_{x\in \pi^{-1}(\und x)} \pi(U_{\al_x}) \;\subset\; \uU.
$$
These are open sets since $\pi^{-1}(\ux)$ is finite and the map $\pi$ is open, and they cover the compact set $\und V$. So we may choose a finite subcover $(\uW_{\und x_i})_{i=1,\dots, n}$ of $\und V$. Then  
$(U_{\al_x})_{x\in \pi^{-1}\{\und x_1, \ldots,\und x_n\}}$ 
is a finite subcover of $\pi^{-1}(\und V)$. 
This shows that preimages of compact sets are compact, i.e.\ $\pi$ is proper.

To see that preimages of precompact sets $P\sqsubset \uU$ are precompact, it suffices to note that the continuity of $\pi$ gives $\ov{\pi^{-1}(P)}\subset \pi^{-1}(\ov P)$, so that $\ov{\pi^{-1}(P)}$ is compact because it is a closed subset of $ \pi^{-1}(\ov P)$, which is compact as preimage of the compact set $\ov{P}\subset \uU$.

To finish the proof of (i) we must show that $\uU$ is a separable, locally compact metric space. 
But $\uU$ inherits these properties from $U$ by  \cite[Ex.~31.7]{Mun} which applies to closed continuous surjective maps $\pi:X\to Y$ such that $\pi ^{-1}(y)$ is compact for all $y\in Y$.  

To prove (ii), first choose any open neighbourhood $V_x\subset U$ of $x$ that is disjoint from its images under the elements of $\Ga\less \Ga^x$, and then set
$$ \textstyle
U_x: = \bigcap_{\ga\in \Ga^x} \ga V_x.
$$
Then $U_x$ is open since $\Ga^x$ is finite and each $\ga V_x$ is open. Moreover, $U_x$ is
invariant under $\Ga^x$, and has the property that
its intersection with each $\Ga$-orbit is either empty or is a $\Ga^x$-orbit.  Thus the restriction of $\pi$ to $U_x$ is simply the quotient by the $\Ga^x$ action, so that $\qu{U_x}{\Ga^x}\to\pi(U_x)$ is the identity.

To prove the first claim in (iii), note that $\Ga$ acts on the partial quotient $\qu{\TU}{\ker \rho^\Ga}$ 
via its identification with $\im \rho^\Ga =  \qu{\TGa}{\ker \rho^\Ga}$ 
to induce a homeomorphism $\qu{\TU}{\TGa}\cong\qu{\bigl(\qu{\TU}{\ker \rho^\Ga}\bigr)}{\Ga}$.
Now $\urho$ is this identification composed with the homeomorphism $\qu{\bigl(\qu{\TU}{\ker \rho^\Ga}\bigr)}{\Ga}\to \qu{U}{\Ga}$ induced by the $\Ga$-equivariant diffeomorphism $\qu{\TU}{\ker \rho^\Ga}\cong U$.

As for the statement about stabilizers, notice that we have $\TGa^y\cap ( \ker \rho^\Ga) =\id$,
because $\ker\rho^\Ga$ acts freely. Thus $\rho^\Ga|_{\TGa^y}$ is injective. It takes values in $ \Ga^x$ for $x: = \rho(y)$ by the equivariance of $\rho$ with respect to $\rho^\Ga$.
To see that $\rho^\Ga|_{\TGa^y}:\TGa^y\to \Ga^x$ is surjective, fix an element $\de\in \Ga^x$. By surjectivity of $\rho^\Ga:\TGa\to \Ga$ we can choose a lift $\Tde\in (\rho^\Ga)^{-1}(\de)$. Since $\rho(\Tde y)=\rho^\Ga(\Tde)\rho(y)=\de x =\rho(y)$ and the fibers of $\rho$ are $\ker\rho^\Ga$ orbits, there is a unique $\ga\in \ker \rho^\Ga$ such that $\ga \Tde y = y$, and hence $\ga\Tde\in \TGa^y$.  Since $\rho^\Ga(\ga\Tde) = \rho^\Ga(\Tde)= \de$, this shows that the induced map on stabilizers $\TGa^y\to \Ga^x$ is surjective and hence an isomorphism. 
\end{proof}

\begin{rmk} \label{rmk:cover2}\rm 
In order to make our presentation more accessible we have chosen to require that the domains of our Kuranishi charts are explicit group quotients $(U,\Ga)$.  Instead we could have worked with \'etale proper groupoids $\Gg$ with the additional property that the realization map $\Obj_{\Gg}\to \Obj_{\Gg}/\!\!\sim$, that identifies two objects iff there is a morphism between them, is proper.
This extra properness assumption is proved for group quotients in Lemma~\ref{le:vep}~(i). 
We will see below that this properness allows us to deduce results about a Kuranishi atlas $\Kk$ from results of \cite{MW1} applied to the intermediate atlas $\uKk$ in which the charts have domains $\und U=\qu{U}{\Ga}$.
$\hfill\er$
\end{rmk}

\subsection{Kuranishi charts and coordinate changes}\label{ss:chart}\hspace{1mm}\\ \vspace{-3mm}

We begin by generalizing the notion of smooth Kuranishi chart (with trivial isotropy) from \cite{MW2} to the case of nontrivial finite isotropy.

\begin{rmk} \label{rmk:conv} \rm
To simplify language, we will not add the specifications ``smooth'' ,  ``nontrivial isotropy'' or ``additive" to Kuranishi charts, coordinate changes, and atlases in this paper. Hence a Kuranishi atlas in this paper is a generalization (allowing nontrivial isotropy) of the notion of smooth additive Kuranishi atlas in \cite{MW2}. 
We will see that it induces a filtered topological Kuranishi atlas in the sense of \cite{MW1}, given by the ``intermediate charts and coordinate changes'' introduced in Definitions~\ref{def:quotlev} and Remark~\ref{rmk:coord} below. So in this paper we will take ``intermediate'' to include the specification ``topological''. 
$\hfill\er$
\end{rmk}

\begin{defn}\label{def:chart}
A Kuranishi chart for $X$ is a tuple $\bK = (U,  E, \Ga,s,\psi)$ consisting of
 \begin{itemize}
 \item
the {\bf domain} $U$ which is a smooth finite dimensional manifold;
\item  a finite dimensional vector space $E$ called the {\bf obstruction space};
\item 
a finite {\bf isotropy group} $\Ga$ with a smooth action on $U$ and a linear action on~$E$;
\item a smooth $\Ga$-equivariant function $s:U\to E$,
called the {\bf section};
 \item
a continuous map $\psi : s^{-1}(0) \to X$ that induces a homeomorphism
$$
\und{\psi}:\und{s^{-1}(0)}: = \qq{s^{-1}(0)}{\Ga} \;\to\; F
$$
with open image $F\subset X$, called the {\bf footprint} of the chart.
\end{itemize}
The {\bf dimension} of $\bK$ is $\dim(\bK): =\dim U-\dim E$.
\end{defn}

In order to extend topological constructions from \cite{MW1} to the case of nontrivial isotropy, we will also consider the following notion of intermediate Kuranishi charts which have trivial isotropy but less smooth structure.

\begin{defn}\label{def:quotlev}
We associate to each Kuranishi chart $\bK = (U, E,\Ga, s, \psi)$ the {\bf intermediate chart} $\und{\bK}: = (\uU, \uEE, \us, \und{\psi})$ consisting of
\begin{itemize}
\item
the {\bf intermediate domain} $\uU:=\qu{U}{\Ga}$;
\item
the {\bf intermediate obstruction ``bundle''}, whose total space $\uEE:=\und{U\times E}$ is the quotient by the diagonal action of $\Ga$, with the projection $\pr:\uEE\to \uU, \Ga(u,e)\mapsto\Ga u$ and zero section $0: \uU\to \uEE, \Ga u \mapsto \Ga(u,0)$;
\item
the {\bf intermediate section} $\us:\uU\to\uEE$ induced by $\s=\id_U\times s :U \to U\times E$;
\item
the {\bf intermediate footprint map} $\upsi:\us^{-1}(\im 0)\to X$ induced by ${\psi:s^{-1}(0)\to X}$.
\end{itemize}
We write $\pi: U\to \uU$ for the projection from the Kuranishi domain. Moreover if a chart $\bK_I= 
(U_I, E_I,\Ga_I,s_I, \psi_I)$ has the label $I$, then $\ubK_I= (\uU_I, \uEE_I,\us_I, \upsi_I)$ 
and $\pi_I: U_I\to \uU_I$ denote the corresponding intermediate chart and projection.
\end{defn}

The intermediate charts and coordinate changes of a Kuranishi atlas (with isotropy) will form a topological Kuranishi atlas (without isotropy). 
For the charts, the following is a direct consequence of Lemma~\ref{le:vep}.

\begin{lemma}\label{le:topinterm}  
The intermediate chart $\und{\bK}$ is a topological chart in the sense of \cite[Definition~2.1.3]{MW1}. In other words,
\begin{itemize}
\item
the intermediate domain $\uU$ is a separable, locally compact metric space;
\item
the  intermediate  obstruction ``bundle'' $\pr :\uEE \to \uU$ is a continuous map between 
separable, locally compact metric spaces, so that 
the zero section $0: \uU \to \uEE$ is a continuous map with $\pr \circ 0 = \id_{\uU}$;
\item
the intermediate section $\us: \uU\to \uEE$ is a continuous map with $\pr \circ \us = \id_{\uU}$;
\item
the intermediate footprint map $\upsi : \us^{-1}(0) \to X$ is a homeomorphism onto 
the footprint $\psi(\s^{-1}(0))=F$, which is an open subset of $X$.
\end{itemize}
\end{lemma}

\begin{rmk}\rm  
(i)
The intermediate bundle $\pr:\uEE\to \uU$ is an orbibundle and hence has more structure 
than a general topological chart. In particular, it has a natural zero section $0: \uU\to \uEE$.  
Hence, when working with labeled charts $\ubK_I$, we will usually simply denote the projection and zero section by $\pr$ and $0$ rather than $\pr_I, 0_I$.
\MS

\NI (ii)  
We will find that many results from \cite{MW1}, in particular the taming constructions, carry over to nontrivial isotropy via the intermediate charts, since precompact subsets of $\und U$ lift to precompact subsets of $U$ by Lemma~\ref{le:vep}~(i).
An important exception is the construction of perturbations which must be done on the smooth spaces $U$.
$\hfill\er$
\end{rmk}

Next, as in \cite{MW1,MW2}, compatibility of Kuranishi charts will require restrictions and embeddings to common transition charts.

\begin{defn} \label{def:restr}
Let $\bK = (U,E, \Ga, s,\psi)$ be a Kuranishi chart and $F'\subset F$
an open subset of its footprint.
A {\bf restriction of $\bK$ to $\mathbf{\emph F\,'}$} is a Kuranishi chart of the form
$$
\bK' = \bK|_{\uU'} := \bigl(U', E, \Ga, s'=s|_{U'} \,,\, \psi'=\psi|_{s'^{-1}(0)}\, \bigr) 
\qquad\text{with}\qquad
U':=\pi^{-1}(\uU')
$$
given by  a choice of open subset $\uU'\subset \uU$ such that $\uU'\cap \upsi^{-1}(F) = \upsi^{-1}(F')$.

We call $\uU'$ the {\bf domain} of the restriction.
\end{defn}

Note that the restriction $\bK'$ in the above definition has footprint $\psi'(s'^{-1}(0))=F'$, and its domain group quotient $(U',\Ga)$ is the restriction of $(U,\Ga)$ to $\uU'$ in the sense of Definition~\ref{def:grprestr}.
Moreover, because the restriction of a chart is determined by a subset of the intermediate domain $\und{U}$, 
we can in the following use the existence result in \cite{MW1} for restrictions of topological charts to obtain restrictions of charts with isotropy.
Here we use the notation $\sqsubset$ to denote a precompact inclusion and we write ${\rm cl}_V(V')$ for the closure of a subset $V'\subset V$ in the relative topology of $V$.

\begin{lemma}\label{le:restr0}
Let $\bK$ be a Kuranishi chart. Then for any open subset $F'\subset F$ there  is a restriction $\bK'$ to $F'$ with domain $\uU'$ such that $U':=\pi^{-1}(\uU')$ satisfies ${\rm cl}_U(U')\cap s^{-1}(0) = \psi^{-1}({\rm cl}_X(F'))$.
Moreover, if $F'\sqsubset F$ is precompact, then $\uU'\sqsubset \uU$ can be chosen precompact so that $U'\sqsubset U$.
\end{lemma}
\begin{proof} 
By \cite[Lemma~2.1.6]{MW1} applied to the intermediate chart $\ubK$,
there is a subset $\und{U'}\subset \und{U}$ 
that defines a restriction of this topological chart, in particular satisfies $\und{U'}\cap \us^{-1}(0) = \upsi^{-1}(F')$, with the additional property
${\rm cl}_{\und{U}}(\und{U'})\cap \us^{-1}(0) = \upsi^{-1}({\rm cl}_X(F'))$. 
Further, we may assume that $\und{U'}$ is precompact in $\und U$ if $F'\sqsubset F$.
Then $U'=\pi^{-1}(\und{U'})$ is the required domain.  It inherits precompactness by Lemma~\ref{le:vep}~(i).
Further, the same lemma shows that $\pi^{-1}({\rm cl}_{\uU}(\und{U'})) = {\rm cl}_U(U')$.  
Hence applying $\pi^{-1}$ to the identity ${\rm cl}_{\und{U}}(\und{U'})\cap \us^{-1}(0) = \upsi^{-1}({\rm cl}_X(F'))$ implies that
${\rm cl}_U(U')\cap s^{-1}(0) = \psi^{-1}({\rm cl}_X(F'))$.
\end{proof}

Most definitions in \cite{MW2} extend, as the previous ones, with only minor changes to the case of nontrivial isotropy. However, the notion of smooth coordinate change \cite[Def.5.2.2]{MW2} needs to be generalized significantly to include a covering map. 
For simplicity we will formulate the definition in the situation that is relevant to 
additive Kuranishi atlases.\footnote{
While additivity was introduced as separate property in \cite{MW0},
it is both so crucial and natural that below in \S\ref{ss:Ks}, we will define the notion of Kuranishi atlas to be automatically additive.
}
That is, we suppose that a finite set of basic Kuranishi charts $(\bK_i)_{i\in \{1,\dots, N\}}$ is given such that for each  $I\subset \{1,\dots, N\}$ with $F_I: = \bigcap_{i\in I} F_i \ne \emptyset$ we have another Kuranishi chart $\bK_I$ with 
\begin{itemize}
\item[-] 
isotropy group $\Ga_I := \prod_{i\in I} \Ga_i$,
\item[-]  
obstruction space $E_I := \prod_{i\in I }E_i$ on which $\Ga_I$ acts with the product action,
 \item[-]  
footprint $F_I: = \bigcap_{i\in I} F_i$. 
\end{itemize}
Then for $I\subset J$ we have the natural splitting $\Ga_J\cong\Ga_I\times \Ga_{J\less I}$ with induced inclusion $\Ga_I \hookrightarrow \Ga_I\times \{\id\} \subset \Ga_J$ and projection $\rho_{IJ}^\Ga: \Ga_J\to \Ga_I$ with kernel $\Ga_{J\less I}$.
(Here,we include the case $I=J$, interpreting $\Ga_{\emptyset}: = \{\id\}$.)
Moreover, we have the natural inclusion $\Hat\phi_{IJ}:  E_I\to E_J$ which is equivariant with respect to the inclusion $\Ga_I\hookrightarrow\Ga_J$ and so that the complement of this inclusion $\Ga_{J\less I}$ acts trivially on the image $\Hat\phi_{IJ}(E_I)\subset E_J$.

\begin{defn} \label{def:change}  
Given  $I\subset J\subset \{1,\dots, N\}$ let $\bK_I$ and $\bK_J$ be Kuranishi charts as above with $F_I\supset F_J$.
A {\bf smooth coordinate change} $\Hat\Phi_{IJ}$ from $\bK_I$ to $\bK_J$ consists of
\begin{itemize}
\item 
a choice of 
domain $\uU_{IJ}\subset \uU_I$ such that  $\bK_I|_{\uU_{IJ}}$ is a restriction of $\bK_I$ to $F_{J}$,
\item 
the splitting $\Ga_J\cong\Ga_I\times \Ga_{J\less I}$ as above, and the induced inclusion $\Ga_I\hookrightarrow \Ga_J$ and projection $\rho_{IJ}^\Ga:\Ga_J\to \Ga_I$,
\item 
a $\Ga_J$-invariant submanifold $ \TU_{IJ}\subset U_J$ on which $\Ga_{J\less I}$ acts freely,
and the induced $\Ga_J$-equivariant inclusion $\Tphi_{IJ}:\TU_{IJ} \hookrightarrow U_J$,
\item
a group covering $(\TU_{IJ},\Ga_J,\rho_{IJ}, \rho_{IJ}^\Ga)$ 
of $(U_{IJ}, \Ga_I)$, where $U_{IJ}: = \pi_I^{-1}(\uU_{IJ})\subset U_I$,
\item
the linear equivariant injection $\Hat\phi_{IJ}:E_I\to E_J$ as above,
 \end{itemize}
such that the 
inclusions $\Tphi_{IJ},\Hat\phi_{IJ}$, and covering $\rho_{IJ}$ 
intertwine the sections and footprint maps,
\begin{align}\label{eq:change2}
s_{J}\circ\Tphi_{IJ}  & = \Hat\phi_{IJ}\circ s_I\circ \rho_{IJ} \quad \mbox{ on }  \TU_{IJ}, \\
\psi_{J}\circ\Tphi_{IJ}  & = \psi_I\circ \rho_{IJ} \qquad\quad\; \mbox{ on } 
s_J^{-1}(0)\cap \TU_{IJ} = 
\rho_{IJ}^{-1}(s_I^{-1}(0)).
\nonumber
\end{align}
Moreover, we denote $s_{IJ}:=s_I\circ\rho_{IJ}:\Ti U_{IJ}\to E_I$ and require the {\bf index condition}:
\begin{enumerate}
\item
The embedding $\Tphi_{IJ}:\TU_{IJ}\hookrightarrow  U_J$ identifies the kernels,
$$
\rd_u\Tphi_{IJ} \bigl(\ker\rd_u s_{IJ} \bigr) =  \ker\rd_{\Tphi_{IJ}(u)} s_J    \qquad \forall u\in \TU_{IJ};
$$
\item
the linear embedding $\Hat\phi_{IJ}:E_{I}\to E_J$ identifies the cokernels,
$$
\forall u\in \TU_{IJ} : \qquad
E_{I} = \im(\rd_u s_{IJ}) \oplus C_{u,I}  \quad \Longrightarrow \quad E_J = \im (\rd_{\Tphi_{IJ}(u)} s_J) \oplus \Hat\phi_{IJ}(C_{u,I}).
$$
\end{enumerate}
The subset $\uU_{IJ}\subset \uU_I$ is called the  {\bf domain} of the coordinate change,
while $\TU_{IJ}\subset U_J$ is its  {\bf lifted domain}.
\end{defn}

Recall that we have $\dim\TU_{IJ}= \dim U_I$ since $\rho_{IJ}: \TU_{IJ}\to U_{IJ}$ is a regular covering.
Moreover, $\rho_{IJ}$ identifies the kernels and images of $\rd s_{IJ}$ and $\rd s_I$, 
\begin{equation}\label{eq:indexcond}
\rd_u\rho_{IJ} (\ker\rd_u s_{IJ})=\ker\rd_{\rho_{IJ}(u)} s_I,
\qquad \im(\rd_u s_{IJ})= \im(\rd_{\rho_{IJ}(u)}) s_I \subset E_I ,
\end{equation}
and hence the index condition is equivalent to kernels and cokernels of $\rd_{\rho_{IJ}(u)} s_I$ and $\rd_u s_J$ being identified by the coordinate change. As in \cite[Lemma~5.2.5]{MW2} it is also equivalent to the {\bf tangent bundle condition}
\begin{equation}\label{tbc}
\rd_{\Tphi_{IJ} (u)} s_J : \;\quotient{\rT_{\Tphi_{IJ} (u)} U_J}{\rd_u\Tphi_{IJ}  (\rT_u \TU_{IJ})} \;\stackrel{\cong}\longrightarrow \; \quotient{E_J}{\Hat\phi_{IJ}(E_I)}\qquad\forall u\in \TU_{IJ}.
\end{equation}
This 
also shows that any two charts that are related by a coordinate change have the same dimension.
To keep our language similar to that in \cite{MW2}, we denote a coordinate change by  
$\Hat\Phi_{IJ} = (\Tphi_{IJ},\Hat\phi_{IJ},\rho_{IJ}): \bK_{I}|_{\uU_{IJ}}\to \bK_J$.
However, since the linear map $\Hat\phi_{IJ}$ is fixed by our conventions, the coordinate change $\Hat\Phi_{IJ}$ is in fact determined by a group covering $(\TU_{IJ},\Ga_J,\rho_{IJ}, \rho_{IJ}^\Ga)$ of $(\pi_I^{-1}(\uU_{IJ}), \Ga_I)$, where $\uU_{IJ}\subset \uU_I$ 
is a choice of domain with
$\uU_{IJ}\cap\upsi_I^{-1}(F_I)=\upsi_I^{-1}(F_J)$.

\begin{rmk}\label{rmk:change} \rm
(i)
In the case of trivial isotropy and with trivial covering $\rho_{IJ}=:\phi_{IJ}^{-1}$, this definition  
is the notion of coordinate change in \cite{MW2} with $\TU_{IJ}=\phi_{IJ}(U_{IJ})$.
Because $U_{IJ}\subset U_I$ is open,  the index condition together with the condition that $\TU_{IJ}$ is a submanifold of $U_J$ implies that $\TU_{IJ}$ is 
an open subset of $s_J^{-1}(E_I)$.
\MS

\NI (ii)
The following diagram of group embeddings and group coverings
is associated to each coordinate change:
\begin{equation}\label{eq:phiIJ}
\begin{array}{clcc}
&(\TU_{IJ},\Ga_J) 
& \stackrel {(\Tphi_{IJ}, \id)}{\longhookrightarrow } &(U_J,\Ga_J)
\vspace{.1in} \\
&\;\; \; \;\downarrow  (\rho_{IJ}, \rho_{IJ}^\Ga) &&\vspace{.1in} \\
(U_I,\Ga_I)\ \longhookleftarrow &(U_{IJ},\Ga_I)  &&
\end{array}
\end{equation}

\MS
\NI (iii) Since $\urho_{IJ}: \und{\TU}_{IJ} \to \uU_{IJ}$ is a homeomorphism by Lemma~\ref{le:vep}~(iii), each coordinate change $(\phi_{IJ},\Hat\phi_{IJ},\rho_{IJ}): \bK_{I}|_{\uU_{IJ}}\to \bK_J$
induces an injective map 
$$
\und{\phi}_{IJ}:=\und{\Tphi}_{IJ}\circ  \urho_{IJ}^{-1} : \uU_{IJ}\to \uU_J
$$ 
on the domain of the intermediate chart.
In fact, there is an induced 
coordinate change $\und{\Hat\Phi}_{IJ} : \ubK_{I}|_{\uU_{IJ}}\to \ubK_J$ 
between the intermediate charts, given by the bundle map $\und{\Hat\Phi}_{IJ} : \und{U_{IJ}\times E_I} \to \und{U_{J}\times E_J}$ which is induced by the multivalued map $(\Tphi_{IJ}\circ\rho_{IJ}^{-1})\times\Hat\phi_{IJ}$ and hence covers 
$\und{\Tphi}_{IJ} \circ  \urho_{IJ}^{-1}=:\uphi_{IJ}$.
This is a topological coordinate change in the sense of  \cite[Definition~2.2.1]{MW1}.
This means in particular that the map 
$$
\und{\Hat\Phi}_{IJ}: \und{U_{IJ}\times E_I} 
=: \uEE_I|_{\uU_{IJ}}:=\pr_I^{-1}(\uU_{IJ}) \to \uEE_J
$$
 is a topological embedding (i.e.\ homeomorphism to its image) 
 that satisfies the following:
\begin{itemize}
\item
It is a bundle map, i.e.\ we have $\pr_J \circ \und{\Hat\Phi}_{IJ} = \uphi_{IJ} \circ\pr_I |_{\pr_I^{-1}(\uU_{IJ})}$ for a topological embedding $\uphi_{IJ}: \uU_{IJ}\to \uU_J$, and it is linear in the sense that  $0_J \circ \uphi_{IJ}  =\und{ \Hat\Phi}_{IJ} \circ 0_I|_{U_{IJ}}$, where $0_I$ denotes the zero section 
$0_I: \uU_I\to \uEE_I$ in the chart $\ubK_I$.
\smallskip

\item
It intertwines the sections and footprints maps, i.e.
$$
\us_J \circ \uphi_{IJ}  = \und{\Hat\Phi}_{IJ} \circ \us_I|_{U_{IJ}},\qquad 
\uphi_{IJ}|_{\upsi_I^{-1}(F_I\cap F_J)}=\upsi_J^{-1} \circ\upsi_I.
$$
\end{itemize}
However, $\und{\Hat\Phi}_{IJ}$ has more smooth structure than a general topological coordinate change since $\uphi_{IJ}:\uU_{IJ}\to \uU_J$  preserves the orbifold structure and $\und{\Hat\Phi}_{IJ}$ is a map of orbibundles.

\MS
\NI (iv)
Conversely, suppose 
we are given a topological coordinate change $\und{\Hat\Phi}_{IJ}: \ubK_{I} \to \ubK_J$ with domain ${\uU_{IJ}}$.
Then any coordinate change from $\bK_I$ to $\bK_J$ that induces $\und{\Hat\Phi}_{IJ}$ is determined by the $\Ga_J$-invariant set $\TU_{IJ}:=\pi_J^{-1}(\uphi_{IJ}(\uU_{IJ}))$ and a choice of $\Ga_I$-equivariant homeomorphism between
$\qu{\TU_{IJ}}{\Ga_{J\less I}}$ and $U_{IJ} := \pi_I^{-1}(\uU_{IJ})$.
If we can choose this homeomorphism to be smooth, then we obtain a  smooth 
coordinate change $\bK_{I} \to \bK_J$ with domain ${\uU_{IJ}}$ provided that the index condition is satisfied, which is a condition on the relation between the set $\TU_{IJ}$ and the section $s_J$.
When constructing coordinate changes in the Gromov--Witten setting
in \cite{MW:gw}, we will see that there is a natural choice of this diffeomorphism since the covering maps $\rho_{IJ}$ are given by forgetting certain added marked points.
Further, the index condition is automatically satisfied
in this setting. 

\MS
\NI (v) Because $\TU_{IJ}$ is defined to be a subset of $U_J$ it is sometimes convenient 
to think of an element $\Tx\in 
\TU_{IJ}$ as an element in $U_J$, omitting the notation for the inclusion 
map $\Tphi_{IJ}:\TU_{IJ}\to U_J$. $\hfill\er$
 \end{rmk}

The next step is to consider restrictions and composites of coordinate changes. Restrictions 
exist analogously to
\cite[Lemma~5.2.6]{MW2}:
For $I\subset J$, given
a coordinate change $\Hat\Phi_{IJ}: \bK_I|_{\uU_{IJ}}\to \bK_J$ and
restrictions $\bK_I': = \bK_I|_{\uU_I'}$ and $\bK_J': = \bK_J|_{\uU_J'}$ whose footprints $F_I'\cap F_J'$ have nonempty intersection,
there is an induced {\bf restricted coordinate change}  
$\Hat\Phi_{IJ}|_{\uU_{IJ}'}:
\bK_I'|_{\uU_{IJ}'}\to \bK_J'$ for any
open
subset $\uU_{IJ}'\subset \uU_{IJ}$ satisfying the conditions
\begin{equation}\label{eq:coordres}
 \uU_{IJ}'\subset \uU_I'\cap \uphi_{IJ}^{-1}(\uU_J'),
\qquad  \uU_{IJ}'\cap \us_I^{-1}(0) = \upsi_I^{-1}(F_I'\cap F_J').
\end{equation}
However, coordinate changes now do not directly compose due to the coverings involved. The induced coordinate changes on the intermediate charts still compose directly, but the analog of  
\cite[Lemma~5.2.7]{MW2} is the following.

\begin{lemma}\label{le:compos} 
Let $I\subset J\subset K$ 
(so that automatically $F_I\supset F_J\supset F_K$)
and suppose that $\Hat\Phi_{IJ}: \bK_I\to \bK_J$ and  $\Hat\Phi_{JK}: \bK_J\to \bK_K$ are coordinate changes with domains ${\uU_{IJ}}$ and $\uU_{JK}$ respectively.
Then the following holds.
\begin{enumerate}
\item
The domain
$\uU_{IJK}:= \uU_{IJ}\cap \uphi_{IJ}^{-1}(\uU_{JK})\subset \uU_I$
defines a restriction 
$\bK_I|_{\uU_{IJK}}$ of $\bK_I$ 
 to $F_K$. 
\item
The composite $\rho_{IJK}: = \rho_{IJ}\circ \rho_{JK}:\TU_{IJK} \to U_{IJK}: =\pi_I^{-1}(\uU_{IJK})$ is defined on $\TU_{IJK}: = \pi_K^{-1}\bigl((\uphi_{JK}\circ\uphi_{IJ})(\uU_{IJK})\bigr)$ via the natural identification of $\rho_{JK}(\TU_{IJK})\subset U_J$ with a subset of $\TU_{IJ}$.
Together with the natural projection $\rho_{IK}^\Ga: \Ga_K\to \Ga_I$ with kernel $\Ga_{K\less I}$
(which factors $\rho_{IK}^\Ga= \rho_{IJ}^\Ga\circ\rho_{JK}^\Ga$),
this forms a group covering $(\TU_{IJK},\Ga_K,\rho_{IJK},\rho_{IK}^\Ga)$ of $(U_{IJK}, \Ga_I)$.
\item
The inclusion $\Tphi_{IJK}:\TU_{IJK}\hookrightarrow U_K$ together with 
the natural inclusion
$\Hat\phi_{IK}:E_I\to E_K$ (which factors $\Hat\phi_{IK}= \Hat\phi_{JK}\circ \Hat\phi_{IJ}$)
and $\rho_{IJK}$ satisfies  \eqref{eq:change2} and the index condition with respect to the indices $I,K$.
\end{enumerate}
Hence this defines a  {\bf composite coordinate change} $$
\Hat\Phi_{J K}\circ \Hat\Phi_{I J}: = \Hat\Phi_{IJK}= (\Tphi_{IJK}, \Hat\phi_{IK},\rho_{IJK})
$$
 from
$\bK_I$ to  $\bK_K$ with domain $\uU_{IJK}$.
\end{lemma}

\begin{proof} 
The corresponding statement for the induced coordinate changes for the intermediate charts is proved in
\cite[Lemma~2.2.5]{MW1}.   Thus claim (i) follows from part (i) of \cite[Lemma~2.2.5]{MW1}.

To see that $\rho_{IJK}$ in (ii) is well defined, we need to verify the inclusion 
$\rho_{JK}(\TU_{IJK})\subset \TU_{IJ}$, or (due to equivariance) equivalently
$\urho_{JK}(\und{\TU}_{IJK})\subset \und{\TU}_{IJ}$.
For that purpose we drop the natural identifications $\und{\Tphi}_{IJ}:\und{\TU}_{IJ}\to \uU_J$ from the notation so that the intermediate coordinate changes are $\und{\phi}_{IJ}=\urho_{IJ}^{-1} : \uU_{IJ}\to \und{\TU}_{IJ} \subset \uU_J$
and the inclusion follows from
\begin{align*}
\urho_{JK}(\und{\TU}_{IJK})
& \ =
\urho_{JK}\bigl((\uphi_{JK}\circ\uphi_{IJ})\bigl(
\uU_{IJ}\cap \uphi_{IJ}^{-1}(\uU_{JK})\bigr)\bigr)  \\
& \ =
(\urho_{JK}\circ \uphi_{JK}) \bigl( \und{\TU}_{IJ} \cap \uU_{JK}
\bigr)
\;\; = \;\;  \und{\TU}_{IJ} \cap \uU_{JK}.
\end{align*}
Next, observe that composites of group covering maps are also group covering maps.  
In particular, since $\Ga_{K\less J}$ acts freely on $\TU_{IJK}\subset \TU_{JK}$ and $\Ga_{J\less I}$ acts freely on the quotient $\qu{\TU_{IJK}}{\Ga_{K\less J}}$  
(because it 
is identified $\Ga_J$-equivariantly with a subset of $\TU_{IJ}$),
the group $\Ga_{K\less I} 
\cong \Ga_{K\less J}\times \Ga_{J\less I}$ acts freely on $\TU_{IJK}$.

To prove 
(iii), first observe that \eqref{eq:change2} holds for the index pair $IK$ because it holds for $IJ$ and $JK$:
\begin{eqnarray*}
 s_{K}\circ \Tphi_{IJK} &=&   \Hat\phi_{JK}\circ s_J\circ \rho_{JK}\big|_{\TU_{IJK}}
 \\ &=&
  \Hat\phi_{JK}\circ (\Hat\phi_{IJ}\circ s_I\circ \rho_{IJ})\circ \rho_{JK}\big|_{\TU_{IJK}} =
   \Hat\phi_{IK}\circ s_I\circ \rho_{IJK}
   \qquad \mbox{ on } \TU_{IJK},
   \\ 
   \psi_{K}\circ \Tphi_{IJK}&=& \psi_J\circ \rho_{JK}
= \psi_I\circ \rho_{IJ}\circ \rho_{JK}  = \psi_I\circ \rho_{IJK} 
\qquad\qquad\qquad \mbox{ on } {s_{K}^{-1}(0)\cap \TU_{IJK}} .
\end{eqnarray*}
Finally, it is easiest to check the index condition in the form given in \eqref{tbc}, i.e.\ we need to 
establish isomorphisms
for all $u\in \TU_{IJK}$,
\begin{equation}\label{tbc1}
\rd_{\Tphi_{IJK}(u)} s_K : \;\quotient{\rT_{\Tphi_{IJK} (u)} U_K}{\rd_u\Tphi_{IJK}  (\rT_u \TU_{IJK})} \;\stackrel{\cong}\longrightarrow \; \quotient{E_K}{\Hat\phi_{IK}(E_I)}. 
\end{equation}
Here and below we will suppress the natural embedding $\Tphi_{IJK}: \TU_{IJK}\to U_K$ from the notation, hence identifying e.g.\ $u\in \TU_{IJK}$ with $\Tphi_{IJK}(u)\in U_K$.
With that, the quotient on the left is naturally identified with the normal fiber
$\qu{\rT_u U_K}{\rT_u \TU_{IJK}}$ to the submanifold $\TU_{IJK}$ of $U_K$. 
Next, $\TU_{IJK}$ is by construction a submanifold of $\TU_{JK}$, which in turn is a submanifold of $U_K$, hence this normal fiber
is isomorphic to the direct sum of the normal fiber of  
$\TU_{IJK}$ in  $\TU_{JK}$ together with that of $\TU_{JK}$ in  $U_K$, 
$$
\qq{\rT_u U_K}{\rT_u \TU_{IJK}} \;\cong\; \qq{\rT_u U_K}{\rT_u \TU_{JK}} \oplus \qq{\rT_u \TU_{JK}}{\rT_u \TU_{IJK}}.
$$
By the index condition for $\Hat\Phi_{JK}$, the map $\rd_u s_K$ 
restricted to the first summand 
induces an isomorphism $\qu{\rT_u U_K}{\rT_u \TU_{JK}}\stackrel{\cong}\to \qu{E_K}{\Hat\Phi_{JK}(E_J)}$.
Considering the second summand, recall that on $\TU_{JK}$ we have $s_K=s_J\circ \rho_{JK}$, where $\rho_{JK}: \TU_{JK}\to U_{JK}$ is a local diffeomorphism onto an open subset of $U_J$. 
It maps $\TU_{IJK}$ to $\rho_{JK}(\TU_{IJK})  = \TU_{IJ}\cap U_{JK}$ so that, with $v: = \rho_{JK}(u)$, the map $\rd_u\rho_{JK}$ induces an isomorphism 
$\qu{\rT_u \TU_{JK}}{\rT_u \TU_{IJK}}\stackrel{\cong}\to \qu{\rT_v U_J}{\rT_v \TU_{IJ}}$.
Thus the restriction of $\rd_u s_K$ to the second summand induces the isomorphism
$$   
\rd_v s_J\circ \rd_u\rho_{JK}\;: \; \qq{\rT_u \TU_{JK}}{\rT_u \TU_{IJK}}\stackrel{\cong}\longrightarrow \qq{\rT_v U_J}{\rT_v \TU_{IJ}}
\stackrel{\cong}\longrightarrow  \qq{E_J}{\Hat\Phi_{IJ}(E_I)},
$$
where the second isomorphism results from the index condition for $\Hat\Phi_{IJ}$.
Putting this all together, $\rd_u s_K$ induces an isomorphism from $\qu{\rT_u U_K}{\rT_u \TU_{IJK}}$ to 
$$
\qq{E_K}{\Hat\Phi_{JK}(E_J)}\oplus \qq{E_J}{\Hat\Phi_{IJ}(E_I)} \; \cong \;  \qq{E_K}{\Hat\Phi_{IK}(E_I)},
$$
 where in the last step we used the fact that $\Hat\Phi_{JK}: E_J \to E_K$ is the natural inclusion. 
This establishes the isomorphism \eqref{tbc1} and thus completes the proof.  
\end{proof}

\begin{rmk}\label{rmk:coord}\rm  
The composition $ \Hat\Phi_{IJK}: \bK_I\to \bK_K$ induces a coordinate change $\und{ \Hat\Phi}_{IJK}: \ubK_I\to \ubK_K$ 
on the intermediate charts.  This agrees with the composition of the intermediate coordinate changes $\und{ \Hat\Phi}_{IJ}, \und{ \Hat\Phi}_{JK}$ as defined for topological charts in 
 \cite[Lemma~2.2.5]{MW1}.
$\hfill\er$
\end{rmk}

Next, the cocycle conditions from \cite[Definition~2.3.2]{MW1} have direct generalizations.

\begin{defn}  \label{def:cocycle}
Let $\bK_\al$ for $\al = I,J,K$ be Kuranishi charts with $I\subset J\subset K$ and let $\Hat\Phi_{\al\be}:\bK_\al|_{\uU_{\al\be}}\to \bK_\be$ for $(\al,\be) \in \{(I,J), (J,K), (I,K)\}$ be coordinate changes. We say that this triple
$\Hat\Phi_{I J}, \Hat\Phi_{J K}, \Hat\Phi_{I K}$ satisfies the
\begin{itemlist} \item {\bf weak cocycle condition} if $\Hat\Phi_{J K}\circ \Hat\Phi_{I J} \approx \Hat\Phi_{I K}$ are equal on the overlap in the sense
\begin{align}\label{eq:wcocycle}
\rho_{IK} &= \rho_{IJ}\circ \rho_{JK} \quad \mbox{ on } \; \TU_{IK}\cap \rho_{JK}^{-1}(\TU_{IJ}\cap U_{JK});
\end{align}
\item {\bf cocycle condition}
if $\Hat\Phi_{J K}\circ \Hat\Phi_{I J} \subset \Hat\Phi_{I K}$, i.e.\  $\Hat\Phi_{I K}$ extends the composed coordinate change in the sense that \eqref{eq:wcocycle}  holds and
\begin{eqnarray}\label{eq:cocycle}
& \uU_{IJ}\cap \uphi_{IJ}^{-1}(\uU_{JK})\subset  \uU_{IK};
\end{eqnarray}
\item {\bf strong cocycle condition}
if $\Hat\Phi_{J K}\circ \Hat\Phi_{I J} = \Hat\Phi_{I K}$ are equal as coordinate changes, that is if \eqref{eq:wcocycle} 
  holds and
\begin{eqnarray}\label{strong cocycle}
& \uU_{IJ}\cap \uphi_{IJ}^{-1}(\uU_{JK})=  \uU_{IK}.
\end{eqnarray}
\end{itemlist}
 \end{defn}

We stated these last two conditions on the level of the intermediate category because, as we now show, they imply corresponding identities on the level of the Kuranishi atlas.

\begin{lemma}\label{le:compos0}
\begin{enumerate}
\item Condition \eqref{eq:wcocycle} implies
$$
\uphi_{IK} = \uphi_{JK} \circ \uphi_{IJ}\quad  \mbox{ on } \; \uU_{IK}\cap \bigl(\uU_{IJ}\cap \uphi_{IJ}^{-1}(\uU_{JK})\bigr);
$$
\item The cocycle condition \eqref{eq:cocycle} implies that 
$$
\rho_{IK} \; =\; \rho_{IJ}\circ \rho_{JK} \quad \mbox{ on } \;  \rho_{JK}^{-1}(\TU_{IJ}\cap U_{JK})\subset \TU_{IK}.
$$
\item The strong cocycle condition \eqref{strong cocycle} implies that 
$$
\rho_{IK} \; =\; \rho_{IJ}\circ \rho_{JK} \quad \mbox{ on } \;  \rho_{JK}^{-1}(\TU_{IJ}\cap U_{JK})= \TU_{IK}.
$$
\end{enumerate}
\end{lemma}

\begin{proof}  
By definition we have $\urho_{\al\be} \circ \pi_\be = \pi_\al \circ\rho_{\al\be}$ when $\al\subset \be$, 
so that condition \eqref{eq:wcocycle}  implies $\urho_{IK}= \urho_{IJ}\circ \urho_{JK}$ on 
$\pi_K\bigl(\TU_{IK}\cap \rho_{JK}^{-1}(\TU_{IJ}\cap U_{JK})\bigr)$.
The identity $\uphi_{\al\be} =  \urho_{\al\be}^{-1}$ from Remark~\ref{rmk:change}~(iii) then implies $\uphi_{IK} = \uphi_{JK} \circ \uphi_{IJ}$ on 
\begin{eqnarray*}\label{eq:setss1}  
\urho_{IK}\Bigl(\pi_K\bigl(\TU_{IK}\cap \rho_{JK}^{-1}(\TU_{IJ}\cap U_{JK})\bigr)\Bigr)
&=&
\pi_I \Bigl(\rho_{IK}\bigl(\TU_{IK} \cap \rho_{JK}^{-1}(\TU_{IJ}\cap U_{JK})\bigr)\Bigr)\\
&=&
\pi_I \Bigl(\rho_{IK}(\TU_{IK}) \cap \rho_{IJ} \circ\rho_{JK}\bigl( \rho_{JK}^{-1}(\TU_{IJ}\cap U_{JK})\bigr)\Bigr)\\
&=& \pi_I  \Bigl(U_{IK} \cap \rho_{IJ} (\TU_{IJ}\cap U_{JK})\Bigr)
\\ &=&
\uU_{IK}\cap\bigl(\uU_{IJ}\cap  \urho_{IJ}(\uU_{JK})\bigr)
\\ &=& \uU_{IK}\cap \bigl(\uU_{IJ}\cap \uphi^{-1}_{IJ}(\uU_{JK})\bigr),
\end{eqnarray*}
where the second equality uses the fact that $\rho_{IK} = \rho_{IJ}\circ \rho_{JK}$ on $\TU_{IK} \cap \rho_{JK}^{-1}(\TU_{IJ}\cap U_{JK})$ and the last uses $\urho_{IJ} = \uphi_{IJ}^{-1}$.
This  proves (i).

Using in addition the identities 
$U_{\al\be} = \pi_\al^{-1}(\uU_{\al\be})$  and $\TU_{\al\be} = 
\pi_\be^{-1}\bigl(\uphi_{\al\be}(\uU_{\al\be})\bigr)$ 
the cocycle condition \eqref{eq:cocycle} implies the inclusion claimed in (ii),
\begin{align*}
 \rho_{JK}^{-1}(\TU_{IJ}\cap U_{JK}) 
 &= ( \pi_J \circ \rho_{JK})^{-1}\bigl( \uphi_{IJ}(\uU_{IJ}) \cap  \uU_{JK}  \bigr) \\
 &= ( \uphi_{IJ} \circ \urho_{IK} \circ \pi_K )^{-1}\bigl( \uphi_{IJ}(\uU_{IJ}) \cap  \uU_{JK}  \bigr) \\
 &= ( \urho_{IK} \circ \pi_K )^{-1}\bigl( \uU_{IJ} \cap \uphi_{IJ}^{-1}( \uU_{JK})  \bigr) \; \subset \; \pi_K^{-1}\bigl(  \urho_{IK}^{-1}(\uU_{IK} )\bigr)  \;=\; \TU_{IK} .
 \end{align*}
The proof of (iii) is the same, with the strong cocycle condition implying equality in the second to last step.
\end{proof}

\subsection{Kuranishi atlases} \label{ss:Ks} \hspace{1mm}\\ \vspace{-3mm}

With the notions of Kuranishi charts and coordinate changes with nontrivial isotropy in place, we can now directly extend the notion of smooth Kuranishi atlas from \cite[Definition~6.1.3]{MW2}.
For comparison with the notions of smooth and topological Kuranishi atlas from \cite{MW1,MW2}, see Remark~\ref{rmk:conv}.

\begin{defn}\label{def:Ku}
A {\bf (weak) Kuranishi atlas of dimension $\mathbf d$} on a compact metrizable space $X$ is a tuple
$$
\Kk=\bigl(\bK_I,\Hat\Phi_{I J}\bigr)_{I, J\in\Ii_\Kk, I\subsetneq J}
$$
consisting of a covering family of basic charts $(\bK_i)_{i=1,\ldots,N}$ of dimension $d$
and transition data $(\bK_J)_{|J|\ge 2}$, $(\Hat\Phi_{I J})_{I\subsetneq J}$ for $(\bK_i)_{i=1,\ldots,N}$, where:
\begin{itemlist}
\item
A {\bf covering family of basic charts} for $X$ is a finite collection $(\bK_i)_{i=1,\ldots,N}$ of Kuranishi charts for $X$ whose footprints cover $X=\bigcup_{i=1}^N F_i$.
\item
{\bf Transition data} for a covering family $(\bK_i)_{i=1,\ldots,N}$ is a collection of Kuranishi charts $(\bK_J)_{J\in\Ii_\Kk,|J|\ge 2}$ and coordinate changes $(\Hat\Phi_{I J})_{I,J\in\Ii_\Kk, I\subsetneq J}$ as follows:
\begin{enumerate}
\item
$\Ii_\Kk$ denotes the set of subsets $I\subset\{1,\ldots,N\}$ for which the intersection of footprints is nonempty,
$$
F_I:= \; {\textstyle \bigcap_{i\in I}} F_i  \;\neq \; \emptyset \;.
$$
\item  For each $J\in\Ii_\Kk$ with $|J|\ge 2$, 
$\bK_J$ is a Kuranishi chart for $X$ with footprint $F_J=\bigcap_{i\in J}F_i$,
group  $\Ga_J =\prod_{j\in J} \Ga_j$, and obstruction space
 $ E_J
= {\textstyle \prod_{j\in J}}  E_j$.
Further,  for one element sets $J=\{i\}$ we denote $\bK_{\{i\}}:=\bK_i$.
\item
$\Hat\Phi_{I J} = \bigl(\rho_{IJ}, \rho^\Ga_{IJ},\Hat\phi_{IJ}\bigr)$ is a coordinate change $\bK_{I} \to \bK_{J}$ for every $I,J\in\Ii_\Kk$ with $I\subsetneq J$,
where $\rho^\Ga_{IJ}: \Ga_J\to \Ga_I$ is the natural projection $\prod_{j\in J} \Ga_j \to \prod_{i\in I} \Ga_i$ and
$\Hat\phi_{IJ}:E_I\to E_J$ is the natural inclusion $\prod_{i\in I}  E_j\to \prod_{j\in J}  E_j$.
\end{enumerate}
\end{itemlist}
Moreover, for a weak atlas  we require that the  weak cocycle condition in Definition~\ref{def:cocycle} hold
for every triple $I,J,K\in\Ii_K$ with $I\subsetneq J \subsetneq K$, while for an atlas the cocycle condition must hold for all such triples.
\end{defn}

\begin{rmk}\label{rmk:addit}\rm  
Note that we have built {\bf additivity} in the sense of \cite[Definition~6.1.4]{MW2} into the above definitions.  
Namely, for each $I\in \Ii_\Kk$ the natural embeddings 
$\Hat \phi_{iI}: E_i\to E_I = {\textstyle \prod_{\ell\in I}}  E_\ell $ 
induce the identity isomorphism
\begin{equation}\label{eq:addd}
{\textstyle \prod_{i\in I}} \;\Hat\phi_{iI}: \; {\textstyle \prod_{i\in I}} \; E_i \;\stackrel{\cong}\longrightarrow \; E_I  
= {\textstyle \prod_{\ell\in I}}  E_\ell ,
\end{equation}
and for $I\subset J$ the linear map $\Hat \phi_{IJ}: E_I\to E_J$ is the induced inclusion $\prod_{i\in I} E_i\to \prod_{i\in J} E_i$.
Further, each group $\Ga_I$ is the product $\prod_{i\in I} \Ga_i$ and we use the 
natural projections $\rho^\Ga_{IJ}: \Ga_J\to \Ga_I$  in the group covering maps of the coordinate changes.
Hence, when $I\subset J\subset K$ the projections $\rho^\Ga_{\bullet\bullet}$ and linear inclusions $\Hat\phi_{\bullet\bullet}$ are automatically compatible:
 $$
 \rho^{\Ga}_{IK} = \rho^{\Ga}_{IJ}\circ \rho^{\Ga}_{JK}, \qquad
\Hat\phi_{IK} = \Hat\phi_{JK}\circ \Hat\phi_{IJ}. 
$$
Thus  when $I\subset J$ we will almost always write  $E_I\subset E_J$ for the subspace  
$\phi_{IJ}(E_I)\subset E_J$, and similarly
we
have a natural identification of
$\Ga_J$ with $\Ga_I\times \Ga_{J\less I}$.
$\hfill\er$
\end{rmk}

\begin{rmk}\rm
Although it seems that many
interdependent 
choices are needed in order to construct a
Kuranishi atlas, this is somewhat deceptive.   For example, in the Gromov--Witten case  considered in 
\cite{MW:gw} (see also \cite{Mcn}), the geometric choices involved in the construction of a family of basic charts $(\bK_i)_{i=1,\dots, N}$ essentially induce the transition data as well.
Namely, each basic chart $\bK_i$  is constructed by adding a certain tuple $\vec w_i$ of marked points to the domains of the stable maps $(f,\bz)$, given by the preimages of a fixed hypersurface of $M$ in a fixed set of disjoint disks.
The group $\Ga_i$ acts by permuting these disks, which has a rather nontrivial effect when viewing the chart in a local slice -- in which the first three marked points are fixed.
However, the transition charts $\bK_J$ are constructed very similarly: Elements of the domain $U_J$ consist of stable maps $(f,\bz)$ 
together with $|J|$ sets of added tuples of marked points $(\vec w_j)_{j\in J}$,
each lying in an appropriate set of disks and mapping to certain hypersurfaces.
Each factor $\Ga_j$ of the group $\Ga_J$ acts by permuting the components of the $j$-th set of disks,  leaving the others alone.
Moreover, the covering map $\TU_{IJ}\to U_I$ simply forgets the extra tuples $(\vec w_j)_{j\in J\less I}.$
Thus it is immediate from the construction that the group $\Ga_{J\less I}$ acts freely on the subset $\TU_{IJ}$ of $U_J$, and that the covering map is equivariant in the appropriate sense.
Further, when $I\subset J\subset K$ the compatibility condition
$\rho_{IK} = \rho_{IJ}\circ \rho_{JK}$ holds whenever both sides are defined.  

Furthermore, the stabilization process explained in \cite{MW:gw} allows us to directly work with products of obstruction spaces $E_I: = \prod_{i\in I}  E_i$;  there is no need for a transversality requirement such as Sum Condition II$'$ in \cite[\S4.3]{MW0}. 
In fact, already each $E_i$ is a product of the form $E_i = \prod_{\ga\in \Ga_i}  (E_{0i})_\ga$, on which $\Ga_i$ acts by permutation of the $|\Ga_i|$ copies of a vector space $E_{0i}$.  

Therefore, just as in the case with no isotropy, once given the geometric choices that determine the basic charts, we naturally obtain an additive weak Kuranishi atlas in which the only new choices are those of the domains $\uU_I=\uU_{II}$ and $\uU_{IJ}$  of the transition charts and coordinate changes which are required to intersect the zero set $\us_I^{-1}(0)$ in $\upsi_I^{-1}(F_J)$.
Note that there is no simple hierarchy by which one could organize these choices 
to automatically fulfill the cocycle condition. Hence concrete constructions will usually only satisfy a weak cocycle condition.
However, we show below that any weak (automatically additive) atlas can be ``tamed" so that it satisfies the strong cocycle condition, and hence in particular gives a Kuranishi atlas.
$\hfill\er$
\end{rmk}

Given a (weak) atlas  $\Kk=\bigl(\bK_I,\Hat\Phi_{I J}\bigr)_{I, J\in\Ii_\Kk, I\subsetneq J}$,
we define the associated {\bf intermediate atlas}
$\uKk:=(\bigl(\ubK_I,\und{\Hat\Phi}_{I J}\bigr)_{I, J\in\Ii_\Kk, I\subsetneq J})$ 
to consist of the intermediate charts and coordinate changes.
The next Lemma shows that the intermediate atlas is a (weak) 
topological atlas in the sense of \cite[Definition~3.1.1]{MW1}, 
and that it is 
{\bf filtered} in the sense that there are closed sets $\uEE_{IJ}\subset \uEE_J: =  \und{U_J\times E_J}$ for each $I\subset J$ 
that satisfy the following conditions (c.f.\ \cite[Definition~3.1.3]{MW1})
\begin{enumerate}
\item $\uEE_{JJ}= \uEE_J$ and $\uEE_{\emptyset J} = \im 0_J$ for all $J\in\Ii_\Kk$;
\item 
$\und{\Hat\Phi}_{JK}\bigl(
\pr_J^{-1}(\uU_{JK})\cap
\uEE_{IJ}\bigr) = \uEE_{IK}\cap \pr_K^{-1}(\im \uphi_{JK})$ for all $I,J,K\in\Ii_\Kk$ with
${I\subset J\subsetneq K}$;
\item  $\uEE_{IJ}\cap \uEE_{HJ} = \uEE_{(I\cap H)J}$ for all $I,H,J\in\Ii_\Kk$ with $I, H \subset J$;
\item  $\im \uphi_{IJ}$ is an open subset of $\s_J^{-1}(\uEE_{IJ})$
for all $I,J\in \Ii_\Kk$ with $I\subsetneq J$.
\end{enumerate}

\begin{lemma} \label{le:Ku3}  
Let $\Kk$ be a weak Kuranishi atlas.
Then the intermediate atlas $\uKk$ is a filtered  weak topological Kuranishi atlas,
with filtration $\uEE_{IJ}: = \und{U_J\times \Hat\phi_{IJ}(E_I)}$, using the conventions 
$E_{\emptyset}:=\{0\}$ and $\Hat\phi_{JJ}:=\id_{E_J}$. 
\end{lemma}

\begin{proof}    
Lemma~\ref{le:topinterm} and Remark~\ref{rmk:change}~(iii) assert that $\uKk$ consists of topological Kuranishi charts and coordinate changes. The intermediate basic charts cover $X$ since they have the same footprints as the basic charts of $\Kk$, and this also implies that the intermediate transition charts have the prescribed footprints.
Moreover, the weak cocycle condition for $\Kk$ transfers to $\uKk$ by Lemma~\ref{le:compos0}~(i), and the same holds for the cocycle condition since its definition \eqref{eq:cocycle} is in terms of the intermediate domains.

Next, to see that $\uEE_{IJ}$ defines a filtration on $\uKk$, we need
a mild generalization of \cite[Lemma~6.3.1]{MW2}.
First note that $\und{U_J\times \Hat\phi_{IJ}(E_I)}\subset \und{U_J\times E_J}$ is closed since 
$U_J\times \Hat\phi_{IJ}(E_I)\subset U_J\times E_J$ is closed and the projection 
$U_J\times E_J\to \und{U_J\times E_J}$ is 
a closed 
map by Lemma~\ref{le:vep}~(i).
The filtration
property (i) above holds by definition, and property 
(iii) holds because additivity implies
$$
\Hat\phi_{IJ}(E_I)  \; \cap \; \Hat\phi_{HJ}(E_H) \;=\; 
 \Hat\phi_{(I\cap H) J}(E_{I\cap H }) .
 $$
Moreover, because $\und{\Hat\Phi_{JK}}=\und{\phi_{JK}\times\Hat\phi_{JK}}$, property (ii) follows by quotienting  the next identity by the group $\Ga_K$:
\begin{align*}
\Hat\Phi_{JK}\bigl( U_{JK}\times \Hat\phi_{IJ}(E_I)\bigr) \;
&= \; \im\phi_{JK} \times \Hat\phi_{JK}\bigl(\Hat\phi_{IJ}(E_I)\bigr) \;=\;  \im\phi_{JK} \times \Hat\phi_{IK}(E_I) \\
&= \; \bigl(U_K\times \Hat\phi_{IK}(E_I)\bigr) \cap \bigl( \im \phi_{JK} \times E_K \bigr).
\end{align*}

Finally, to check property (iv) we first apply \cite[Lemma~5.2.5]{MW2} to the embedding $\Tphi_{IJ}:\TU_{IJ}\to U_J$, which satisfies the index condition, i.e.\ identifies kernel and cokernel of $\rd s_J$ and $\rd s_I$ (the latter being pulled back with the covering $\rho_{IJ}$). 
It implies that $\im \Tphi_{IJ}$ is an open subset of $s_J^{-1}(E_I)$.
This openness is preserved in the $\Ga_J$ quotient, since Lemma~\ref{le:vep} applies to the projection $s_J^{-1}(E_I)\to \qu{s_J^{-1}(E_I)}{\Ga_J}= \s_J^{-1}(\und{U_J\times E_I})=\s_J^{-1}(\uEE_{IJ})$, which maps $\im \phi_{IJ}$ to $\im \uphi_{IJ}$.
\end{proof}

If $\Kk$ is a Kuranishi atlas, then the topological atlas  $\uKk$ also satisfies the cocycle conditions, and 
hence by \cite[Lemma~2.3.7]{MW1} there is an
{\bf intermediate domain category} $\bB_{\uKk}$ with objects 
$
\Obj_{\bB_{\uKk}} : = {\textstyle  \bigsqcup_{I\in \Ii_\Kk}} \uU_I
$
equal to the disjoint  union of the intermediate domains, and  morphisms
$$
\Mor_{\bB_{\uKk}} : = {\textstyle  \bigsqcup_{I\subset J}} \uU_{IJ}
$$
given by the 
intermediate
coordinate changes $\uphi_{IJ}: \uU_{IJ}\to \uU_J$, where 
the identity maps $\uphi_{II}$ on $\uU_{II}=\uU_I$ are included.
Thus the source and target maps are
$$
s\times t: \;\uU_{IJ} \to \uU_I\times \uU_J\subset \Obj_{\bB_{\uKk}}\times \Obj_{\bB_{\uKk}},\quad
(I,x)\mapsto \bigl( (I,x), (J,\uphi_{IJ}(x))\bigr).
$$
The following gives the analogous categorical interpretation for the Kuranishi atlas itself.

\begin{defn}\label{def:catKu}
Given a  Kuranishi atlas $\Kk$ we define its {\bf domain category} $\bB_\Kk$ to consist of
the space of objects 
$$
\Obj_{\bB_\Kk}:=  \bigsqcup_{I\in \Ii_\Kk} U_I \ = \ \bigl\{ (I,x) \,\big|\, I\in\Ii_\Kk, x\in U_I \bigr\}
$$
and the space of morphisms
$$
\Mor_{\bB_\Kk}:= \bigsqcup_{I,J\in \Ii_\Kk, I\subset J} \TU_{IJ}\times \Ga_I \ = \
 \bigl\{ (I,J,y,\gamma) \,\big|\, I\subset J, y\in \TU_{IJ}, \gamma\in \Ga_I \bigr\}.
$$
Here we denote $\TU_{II}:= U_I$ for $I=J$, and for $I\subsetneq J$ use
the lifted domain $\TU_{IJ}\subset U_J$ of the restriction $\bK_I|_{\uU_{IJ}}$ to $F_J$
that is part of the coordinate change $\Hat\Phi_{IJ} : \bK_I|_{\uU_{IJ}}\to \bK_J$.
Source and target of these morphisms are given by
\begin{equation}\label{eq:Bcomp}
(I,J,y,\gamma)\in \Mor_{\bB_\Kk}\bigl((I,\ga^{-1} \rho_{IJ}(y)), \ (J,\Tphi_{IJ}(y))\bigr),
\end{equation}
where we denote $\Tphi_{II} = \id$.
Composition\footnote
{
Note that we write compositions in the categorical ordering here.
Moreover, recall that $\Tphi_{JK}:\TU_{JK}\to U_K$ is the canonical inclusion of the subset $\TU_{JK}\subset U_K$.
We then identify $z =\Tphi_{IK}^{-1}(\Tphi_{JK}(z))$, since 
composability of the morphisms implies $z\in \rho_{JK}^{-1}( \TU_{IJ} \cap U_{JK})$
and the cocycle condition ensures that  $\rho_{JK}^{-1}( \TU_{IJ} \cap U_{JK})$ is contained in $\TU_{IK}$, where both are considered as subsets of $U_K$.
} 
is defined by
$$
\bigl(I,J,y, \ga\bigr)\circ \bigl(J,K,z,\de\bigr)
:= \bigl(I,K,z = \Tphi_{IK}^{-1}(\Tphi_{JK}(z)),
\rho^\Ga_{IJ}(\de) \ga \bigr)
$$
whenever $\de^{-1} \rho_{JK}(z)=\Tphi_{IJ}(y)$.
\MS

The {\bf obstruction category} $\bE_\Kk$ is defined in complete analogy to $\bB_\Kk$ to consist of the spaces of objects $\Obj_{\bE_\Kk}:=\bigsqcup_{I\in\Ii_\Kk} U_I\times E_I$ and morphisms
$$
\Mor_{\bE_\Kk}: = {\textstyle \bigsqcup}_{I\subset J, I,J\in\Ii_\Kk} \TU_{IJ}\times E_I\times \Ga_I,
$$
with source and target maps
$$
 (I,J,y, e,\ga) \mapsto \bigl(I,\ga^{-1}\rho_{IJ}(y), \ga^{-1} e\bigr) , \qquad  (I,J,y,e,\ga) \mapsto \bigr(J,\Tphi_{IJ}(y),\Hat\phi_{IJ}(e)),
$$
and composition defined by
$$
\bigl(I,J,y, e, \ga\bigr)\circ \bigl(J,K,z,f,\de\bigr)
:= \bigl(I,K, \Tphi_{IK}^{-1}(\Tphi_{JK}(z)),f, \rho^\Ga_{IJ}(\de) \ga \bigr)
$$
for any $I\subset J \subset K$ and $(y,e,\ga)\in \TU_{IJ}\times E_I\times \Ga_I, (z,f,\de)\in  \TU_{JK}\times E_J\times \Ga_J$ such that $\rho^\Ga_{IJ}(\de^{-1}) \rho_{JK}(z)=\Tphi_{IJ}(y)$
and $\de^{-1}f = e$.
\end{defn}

\begin{lemma}\label{le:Kcat}  
If $\Kk$ is a Kuranishi atlas, then 
the categories $\bB_{\Kk}, \bE_{\Kk}$  are well defined. 
\end{lemma}
\begin{proof}  
We must check that the composition of morphisms in $\bB_{\Kk}$ is well defined, 
has identities, 
and is associative; the proof for $\bE_\Kk$ is analogous.
We begin by checking that $z=\Tphi_{IK}^{-1}(\Tphi_{JK}(z))$ lies in the lifted domain $\TU_{IK}$ of $\Hat\Phi_{IK}$.
For that purpose we drop the natural inclusions $\Tphi_{* *}$ from the notation and
note that the composition $\bigl(I,J,y,\ga\bigr)\circ \bigl(J,K,z,\de\bigr)$ is defined only when the target of $\bigl(I,J,y,\ga\bigr)$ equals the source of $\bigl(J,K,z,\de\bigr)$; i.e.\ when $y = \de^{-1} \rho_{JK}(z)$.
So the cocycle condition in Lemma~\ref{le:compos0}~(ii) implies that $z\in \rho_{JK}^{-1}\bigl(\de  y\bigr)$ is contained in 
$\rho_{JK}^{-1}\bigl(\TU_{IJ}\cap U_{JK}\bigr)\subset \TU_{IK}$, as claimed.
This means that $(I,K,z,\rho^\Ga_{IJ}(\de) \ga)$ is a well defined morphism of $\bB_\Kk$. Its source is 
$$
\bigl(\rho^\Ga_{IJ}(\de) \ga \bigr)^{-1}\rho_{IK}(z)
= 
\ga^{-1} \rho^\Ga_{IJ}(\de)^{-1} \rho_{IJ} ( \de y )
=
\ga^{-1} \rho_{IJ}(y),
$$
which coincides with the source of $\bigl(I,J,y,\ga\bigr)$ as required.
Finally, the target of the composed morphism, $z=\Tphi_{IK}\bigl(\Tphi_{IK}^{-1}(\Tphi_{JK}(z)\bigr)$ coincides with the target $\Tphi_{JK}(z)$ of $\bigl(J,K,z,\de\bigr)$.
This shows that composition is well defined.
The identity morphisms are given by $\bigl(I,I,x,\id\bigr)$ for all $x\in U_{II}:=U_I$.
To check associativity we consider $I\subset J \subset K\subset L$ and suppose that the three morphisms
$\bigl(I,J,y,\ga\bigr), \bigl(J,K,z,\de\bigr),\bigl(K,L,w,\si \bigr)$ are composable. Then we have
\begin{align*}
\bigl(I,J,y,\ga\bigr)\circ \Bigl(\bigl(J,K,z,\de\bigr) \circ \bigl(K,L,w,\si)\bigr) \Bigr) 
&\;=\;
\bigl(I,J,y,\ga\bigr)\circ \bigl(J,L,w, \rho^\Ga_{JK}(\si)\de \bigr)\\
&\;=\;
\bigl(I,L,w,   \rho^\Ga_{IJ}\bigl(\rho^\Ga_{JK}(\si) \de \bigr)  \ga\bigr),
\end{align*}
and associativity follows from comparing this expression with
\begin{align*}
\Bigl( \bigl(I,J,y,\ga\bigr)\circ \bigl(J,K,z,\de\bigr) \Bigr) \circ \bigl(K,L,w,\si\bigr)
&\;=\;
\bigl(I,K,z, \rho^\Ga_{IJ}(\de)\ga \bigr) \circ \bigl(K,L,w,\si\bigr)\\
&\;=\;
\bigl(I,L,w, \rho^\Ga_{IK}(\si) \rho^\Ga_{IJ}(\de)  \ga \bigr).
\end{align*}
This completes the proof.
\end{proof}

For the rest of this subsection we will make the standing assumption that $\Kk$ is a Kuranishi atlas, 
i.e.\ satisfies the cocycle condition (not just the weak cocycle condition).
Given the categorical interpretation of domains and obstruction spaces of Kuranishi charts, we can now express the bundles, sections, and footprint maps as functors.

\begin{itemlist}
\item
The obstruction category $\bE_\Kk$ is a bundle over $\bB_\Kk$ in the sense that there is a functor
$\pr_\Kk:\bE_\Kk\to\bB_\Kk$ that is given on objects and morphisms by projection 
$(I,x,e)\mapsto (I,x)$ and $(I,J,y,e,\ga)\mapsto(I,J,y,\ga)$.

\item
The  sections $s_I$ induce a smooth section of this bundle, i.e.\ a functor  $\s_\Kk:\bB_\Kk\to \bE_\Kk$ which acts  
smoothly on the spaces of objects and morphisms, and whose composite with the projection $\pr_\Kk: \bE_\Kk \to \bB_\Kk$ is the identity. More precisely, $\s_\Kk$ is given by $(I,x)\mapsto (I,x, s_I(x))$ on objects and by $(I,J,y,\ga)\mapsto (I,J,y, s_I(y),\ga)$ on morphisms.
\item
The zero sections also fit together to give a functor $0_\Kk: \bB_\Kk\to \bE_\Kk$
given by $(I,x)\mapsto (I,x, 0)$ on objects and by $(I,J,y,\ga)\mapsto (I,J,y,0,\ga)$ on morphisms.

\item
The footprint maps $\psi_I$ give rise to a surjective functor $\psi_\Kk: \s_\Kk^{-1}(0):=\bigsqcup_{I\in\Ii_\Kk} s_I^{-1}(0) \to \bX$ 
to the category $\bX$ with object space $X$ and trivial morphism spaces.
It is given by $(I,x)\mapsto \psi_I(x)$ on objects and by $(I,J,y,\ga)\mapsto {\rm id}_{\psi_J(\Tphi_{IJ}(y))} = {\rm id}_{\psi_I(\ga^{-1} \rho_{IJ}(y))}$ on morphisms.
\end{itemlist}

As in \cite{MW1}, we denote by $|\Kk|$ resp.\  $|\uKk|$ the {\bf realization} of the category $\bB_\Kk$ resp.\ $\bB_\uKk$.
This is the topological space obtained as the quotient of the object space by the equivalence relation generated by the morphisms.
The next Lemma fits the quotient maps $\pi_\Kk: \Obj_{\bB_\Kk} \to |\Kk|, \; (I,x)\mapsto [I,x]$ and
$\pi_\uKk: \Obj_{\bB_\uKk} \to |\uKk|, \; (I,\ux)\mapsto [I,\ux]$
into a commutative diagram that will allow us to identify the realizations $|\Kk|\cong |\uKk|$ as topological spaces.

\begin{lemma}\label{le:und}   
If $\Kk$ is a Kuranishi atlas, then there is a functor
$\rho_\Kk: \bB_\Kk\to \bB_\uKk$ that is given on objects by the
quotient maps $U_I \to \uU_I, x\mapsto \ux$,  and on morphisms by the group coverings $\rho_{IJ}$ together with a quotient,
$$
 \TU_{IJ}\times \Ga_I \; \to \;  \uU_{IJ}, \qquad
 \bigl(I,J, y, \ga\bigr) \; \mapsto \; \bigl(I,J,\ul{\rho_{IJ}(y)}\bigr) .
$$
It induces a homeomorphism $|\rho_\Kk| : |\Kk| \to |\uKk|$ between the realizations that fits into a commutative diagram
\[
\xymatrix{
 \Obj_{\bB_\Kk}   \ar@{->}[d]^{\pi_\Kk} \ar@{->}[r]^{\rho_\Kk}   & \Obj_{\bB_\uKk} \ar@{->}[d]^{\pi_\uKk}   \\
|\Kk| \ar@{->}[r]^{|\rho_\Kk|}  & |\uKk| .
}
\]
\end{lemma}

\begin{proof}  
To see that $\rho_\Kk$ is a functor, recall that $(y,\ga)\in \TU_{IJ}\times \Ga_I$ represents a morphism from $\ga^{-1}\rho_{IJ}(y)$ to $y\in U_J$. On the other hand, $\und{\rho_{IJ}(y)}=\urho_{IJ}(\uy)\in \uU_{IJ}$ represents a morphism from $\und{\rho_{IJ}(y)}=\und{\ga^{-1}\rho_{IJ}(y)}$ to $\uphi_{IJ}\bigl(\und{\rho_{IJ}(y)}\bigr) = \uy$, which shows compatibility of $\rho_\Kk$ with source and target maps. Compatibility with composition as in \eqref{eq:Bcomp} follows from $\urho_{IK}(\uz)=\urho_{IJ}(\uy)$ when $\uy= \urho_{JK}(\uz)$.

Next, any functor such as $\rho_\Kk$ induces a map $|\rho_\Kk|$ between the realizations that is defined exactly by the above commutative diagram.
The map $|\rho_\Kk|$ is surjective because the functor $\rho_\Kk$ is surjective on the level of objects. It is injective because $\rho_\Kk$ is surjective on the level of morphisms.

To check that $|\rho_\Kk|$ is open and continuous note that $|\rho_\Kk|(U)=V$ is equivalent to 
$\rho_\Kk^{-1}\bigl( \pi_\Kk^{-1}(U)\bigr) = \pi_\uKk^{-1}(V)$. 
Since $\rho_\Kk$ is continuous and open by Lemma~\ref{le:vep}~(i), and $|\Kk|,|\uKk|$ are equipped with the quotient topologies, the openness of $U\subset |\Kk|$, $\pi_\Kk^{-1}(U)$, $\pi_\uKk^{-1}(V)$, and $V\subset |\uKk|$ are all equivalent. This proves that $|\rho_\Kk|$ is a homeomorphism.
\end{proof}

\begin{rmk}\rm  
\rm (i)   
If $\Kk$ is a Kuranishi atlas with trivial isotropy groups $\Ga_I=\{\id\}$, then the intermediate atlas $\uKk$ has the exact same object space and naturally diffeomorphic morphism spaces, only the direction of the maps in the coordinate changes are reversed from $\rho_{IJ}:\TU_{IJ} \to U_{IJ}\subset U_I$ to $\uphi_{IJ}=\rho_{IJ}^{-1}:U_{IJ} \to \TU_{IJ}\subset U_J$.
In this special case, $\uKk$ is a Kuranishi atlas in the sense of \cite{MW2}, and Lemma~\ref{le:und} identifies the atlases and their realizations.

\NI
\rm (ii)
In general, the spaces of objects and morphisms of the intermediate category are orbifolds, and there is at most one morphism between any pair of objects. 
However, just as in the case of trivial isotropy, we do not attempt to make this category into a groupoid by formally inverting the morphisms 
and then adding all resulting composites, since doing so would in general give components of the morphism space 
without
orbifold structure; c.f.\ \cite[Remark~6.1.7]{MW2}.
This objection does not apply if all the obstruction spaces are trivial. It is shown in \cite{Mcn,Mcorb} 
that every such atlas can be completed to a groupoid without changing its realization.
$\hfill\er$
\end{rmk}

In complete analogy to Lemma~\ref{le:und}, the obstruction categories $\bE_\Kk$ and $\bE_\uKk$ of the Kuranishi atlas $\Kk$ and the intermediate atlas $\uKk$ also fit into a commutative diagram that identifies their realizations $ |\bE_\Kk|\cong  |\bE_\uKk|$.
Moreover, these two diagrams also intertwine the section functors $\s_\Kk, \s_\uKk$ and their realizations:
\begin{equation}\label{eq:master}
\xymatrix{
& \Obj_{\bB_\Kk} \ar@{->}[l]_{\rho_\Kk}   \ar@{->}[d]^{\pi_{\Kk}} \ar@{->}[r]^{\s_{\Kk}}  & 
\Obj_{\bE_\Kk}   \ar@{->}[d]^{\pi_{\bE_\Kk}} \ar@{->}[r]  & \Obj_{\bE_\uKk} \ar@{->}[d]^{\pi_{\bE_\uKk}} \ar@{<-}[r]^{\s_{\uKk}}  
& \Obj_{\bB_\uKk}   \ar@{->}[d]^{\pi_{\uKk}}  \ar@{<-}[r]^{\rho_\Kk}  & \\
& |\Kk|  \ar@{->}[l]_{|\rho_\Kk|}\ar@{->}[r]^{|\s_{\Kk}|}  & |\bE_\Kk| \ar@{->}[r]  & |\bE_\uKk| 
& |\uKk| \ar@{->}[l]_{|\s_{\uKk}|}  \ar@{<-}[r]^{|\rho_\Kk|} & 
}
\end{equation}
There are analogous diagrams for the projection functors $\pr_\Kk, \pr_\uKk$ and zero sections $0_\Kk, 0_\uKk$, which identify the induced maps between the realizations as stated below.

\begin{lemma} \label{le:realization} 
Let $\Kk$ be a Kuranishi atlas. 
\begin{enumilist}
\item
The functors ${\rm pr}_\Kk:\bE_\Kk\to\bB_\Kk$ and ${\pr}_\uKk: \ubE_\Kk\to \ubB_\Kk$ induce the same continuous map
$$
|{\rm pr}_\Kk|:|\bE_\Kk| \to |\Kk|,
$$
which we call the {\bf obstruction bundle} of $\Kk$, although its fibers generally do not have the structure of a vector space.  
\item
The zero sections $0_\Kk:\bB_\Kk\to \bE_\Kk$, $0_\uKk:\bB_\uKk\to \bE_\uKk$ 
as well as the
section functors $\s_\Kk:\bB_\Kk\to \bE_\Kk$, $\s_\uKk:\bB_\uKk\to \bE_\uKk$ 
induce the same continuous maps
$$
|0_\Kk|\cong |0_\uKk| \, : \;  |\Kk|\to |\bE_\Kk| ,  \qquad  |\s_\Kk|\cong |\s_\uKk| \, : \;  |\Kk|\to |\bE_\Kk|,
$$
which are sections in the sense that
$|\pr_\Kk|\circ|0_\Kk| = {\rm id}_{|\Kk|} = |\pr_\Kk|\circ|\s_\Kk|$. 
\item[\rm (iii)]
There is a natural homeomorphism from the realization of the subcategory $\s_\Kk^{-1}(0)$
to the zero set of $|\s_\Kk|$, with the relative topology induced from $|\Kk|$,
$$
\bigr| \s_\Kk^{-1}(0)\bigr| \;=\; \quotient{\s_\Kk^{-1}(0)}{\sim}
\;\overset{\cong}{\longrightarrow}\;
|\s_\Kk|^{-1}(|0_\Kk|) \,:=\; \bigl\{[I,x] \,\big|\, |\s_\Kk|([I,x])= |0_\Kk|([I,x])  \bigr\}  \;\subset\; |\Kk| .
$$
\end{enumilist}
\end{lemma}
\begin{proof}
The induced maps on the realizations are identified by commutative diagrams such as \eqref{eq:master}.
The continuity and other identities are proven exactly as in \cite[Lemma~6.1.9]{MW2} for the case of trivial isotropy.
\end{proof}

Next, we extend the notion of metrizability to Kuranishi atlases with nontrivial isotropy.
In the case of trivial isotropy, recall from \cite[Definition~6.1.13]{MW2} that an admissible metric
is a bounded metric $d$ on the set $|\Kk|$ such that for each $I\in \Ii_\Kk$ the pullback metric $d_I:=(\pi_\Kk|_{U_I})^*d$ on $U_I$ induces the given topology on the manifold $U_I$. 
However, in the presence of isotropy, it makes no sense to try to pull this metric back to $U_I$ since the pullback of a metric by a noninjective map is no longer a metric.
Instead, we use the fact that the realizations $|\Kk|\cong |\uKk|$ of the Kuranishi atlas and its intermediate atlas are canonically identified, which allows us to work with admissible metrics on $|\uKk|$, which is the realization of a topological Kuranishi atlas $\uKk$ with trivial isotropy and given metrizable topologies on the domains $\uU_I=\qu{U_I}{\Ga_I}$.

\begin{defn}\label{def:metrizable}  
Let $\Kk$ be a Kuranishi atlas. Then an {\bf admissible metric} on $|\Kk|\cong |\uKk|$ is a bounded metric on this set (not necessarily compatible with the topology of the realization) such that for each $I\in \Ii_\Kk$ the pullback metric $\und d_I:=(\pi_\uKk|_{\uU_I})^*d$ on $\uU_I$ induces the given quotient topology on $\uU_I=\qu{U_I}{\Ga_I}$.

A {\bf metric Kuranishi atlas} is a pair $(\Kk,d)$ consisting of a Kuranishi atlas $\Kk$ together with a choice of  admissible metric $d$ on $|\Kk|$.
\end{defn}

We finish this subsection with two comparisons of our notion of Kuranishi atlas -- on the one hand with orbifolds, and on the other hand with Kuranishi structures.

\begin{example}\label{ex:foot}\rm  
If the obstruction spaces are trivial, i.e.\ $E_I = \{0\}$ for all $I$, then the two categories 
$\bB_\Kk, \bE_\Kk$ are equal, and their realization is an orbifold.  
A first nontrivial example is a ``football" $X = S^2$ with two basic Kuranishi charts $(U_1, \Ga_1=\Z_2, \psi_1)$, $(U_2, \Ga_2=\Z_3,\psi_2)$ covering neighbourhoods $\upsi_i(\uU_i)\subset S^2$ of the northern resp.\ southern hemisphere with isotropy of order $2$ resp.\ $3$ at the north resp.\ south pole.
We may moreover assume that the overlap 
$\upsi_1(\uU_1)\cap \upsi_2(\uU_2)=\uA$ 
is 
an annulus around the equator.
The restrictions of the basic charts to $\uA$ are $(A_1, \Z_2)$ and $(A_2, \Z_3)$, where both 
$A_i= \psi_i^{-1}(\uA)$ are annuli, but the freely acting isotropy groups are different.
There is no functor between these restrictions because the coverings $A_1\to \uA$ and $A_2\to \uA$ are incompatible.  
However, they both have functors (i.e.\ coordinate changes) to a common free covering, namely the pullback defined by the diagram
\[
\xymatrix{
U_{12}   \ar@{->}[d] \ar@{->}[r]   & A_1 \ar@{->}[d]^{\pi_1}   \\
A_2 \ar@{->}[r]^{\pi_2}  & X
}
\]
i.e.\ $U_{12} := \{(x,y)\in A_1\times A_2 \,|\, \pi_1(x) = \pi_2(y)\}$
with group $\Ga_{12}: = \Ga_1\times \Ga_2 = \Z_2 \times \Z_3$.
The corresponding footprint map $\psi_{12}:U_{12}\to \uA$ is the $6$-fold covering of the annulus, and the  coordinate changes from $(U_i,\Ga_i,\psi_i)|_{\uA}$ to $(U_{12},\Ga_{12},\psi_{12})$ are the coverings $\TU_{i,12} := U_{12} \to A_i =: U_{i,12}$ in the diagram.
Therefore the category $\bB_\Kk$ in this example has index set $\Ii_\Kk = \{1,2, 12\}$, 
objects the disjoint union $\bigsqcup_{I\in \Ii_\Kk} U_I$, 
and
morphisms
$$ 
\Bigl({\textstyle  \bigsqcup_{I\in \Ii_\Kk}} U_I\times \Ga_I  \Bigr) \cup \Bigl({\textstyle  \bigsqcup_{i=1,2} U_{12}\times \Ga_i  }\Bigr) ,
$$
where for $i=1,2$ the elements in $U_{12}\times \Ga_i$ represent the morphisms from $U_{i} $ to $U_{12}$.

This
simple construction does not work for arbitrary orbifolds since the (set theoretic) pullback $U_{12}$ considered above will not be a smooth manifold if any point in 
$\psi_1(U_1)\cap \psi_2(U_2)$ 
has nontrivial stabilizer.  However, we show in \cite{Mcorb} that the construction can 
be generalized to show that every orbifold has a Kuranishi atlas with trivial obstruction spaces.
$\hfill\er$
\end{example}

\begin{rmk}[Relation to Kuranishi structures]  
\label{rmk:FOOO1}\rm 
A Kuranishi structure in the sense of \cite[App.~A]{FOOO}
and \cite{FOOO12} 
consists of a Kuranishi chart $\bK_p$ at every point $p\in X$ and coordinate changes $\bK_q|_{U_{qp}}\to \bK_p$ whenever $q\in F_p$, that satisfy a suitable weak cocycle condition.
Much as in the case of Kuranishi atlases with trivial isotropy (see \cite[Remark~6.1.15]{MW2}),
a weak 
Kuranishi atlas in the sense of Definition~\ref{def:Ku} induces a Kuranishi structure. 
Indeed, given a covering family of basic charts $(\bK_i)_{i=1,\dots,N}$ with footprints $F_i$, we may choose a family of compact subsets $C_i\subset F_i$ that also cover $X$.
Then we use the transition data $(\bK_I,\Hat\Phi_{IJ})$ and weak cocycle conditions to obtain a Kuranishi structure as follows:

\begin{itemlist}
 \item
For any $p\in X$ we define $\bK_p:= \bK_{I_p}|_{U_p}$ to be a restriction of $\bK_{I_p}$, where $I_p: =  \{i \,|\, p\in C_i\}$, and $U_p\subset U_{I_p}$ is an open subset such that the footprint $F_p:=\psi_{I_p}(s_{I_p}^{-1}(0)\cap U_p)$ is a neighbourhood of $p$ and contained in $\cap_{i\in I_p} F_i\less \cup_{i\notin I_p} C_i$.
Here we use a more general notion of restriction than Definition~\ref{def:restr} in that we allow for a domain $U_p$ that is invariant only under a subgroup $\Ga_p\subset \Ga_{I_p}$ such that the induced map $\qu{U_p}{\Ga_p}\to \qu{U_{I_p}}{\Ga_{I_p}}$ is a homeomorphism to its image.
More precisely, to satisfy the minimality requirements of \cite[App~A1.1]{FOOO}, we choose a lift $x_p\in \pi^{-1}(p)\cap U_{I_p}$, set $\Ga_p:= \Ga^{x_p}_{I_p}$ to be its stabilizer in $\Ga_{I_p}$, and take the domain $U_p\subset U_{I_p}$ to be a $\Ga^{x_p}_{I_p}$-invariant neighbourhood of $x_p$, which exists with the required topological properties by Lemma~\ref{le:vep}~(ii).

\item 
For $q\in F_p$ we 
have $I_q\subset I_p$ since by construction $F_p\cap C_i=\emptyset$ for $i\not\in I_p$. 
So we obtain a coordinate change\footnote{
While \cite{FOOO} denotes this coordinate change by $\phi_{pq}$, we will write $\Hat\Phi_{qp}$ for  
consistency with our notation $\Phi_{IJ}:\bK_I\to \bK_J$.}
 $\Hat\Phi_{qp}: \bK_q\to \bK_p$ from a suitable restriction of $\Hat\Phi_{I_q I_p}$ to a $\Ga^{x_q}_q$-invariant neighbourhood $U_{qp}\subset U_q$ of $x_q$.  
 More precisely, we choose $U_{qp}\subset U_q$ small enough so that the projection $\rho_{I_q I_p}: U_p\cap \TU_{I_q I_p}\to U_{I_q I_p}$ has a continuous section over $U_{qp}$.
We denote its image by $\TU_{qp}$ and thus obtain an embedding $\phi_{qp}:= \rho_{I_q I_p}^{-1} : U_{qp}\to \TU_{qp}\subset U_p \cap \TU_{I_q I_p}$. 
Since the projection $\rho_{I_q I_p}$ induces an isomorphism on stabilizer subgroups by Lemma~\ref{le:vep}~(iii), this is equivariant with respect to a suitable injective homomorphism $h_{qp}:\Ga_q\to \Ga_p$ and induces an injection $\uphi_{qp}: \uU_{qp}: = \qu{U_{qp}}{\Ga_q} \to \uU_{p}: = \qu{U_{p}}{\Ga_p}$.
Since the map $\uU_{qp}= \qu{U_{qp}}{\Ga_q}\to \uU_{I_q} = \qu{U_{I_q}}{\Ga_{I_q}}$ is a homeomorphism to its image by construction of $\uU_q\to \uU_{I_q}$ above, and similarly for $p$,
we can identify $\uphi_{qp}$ with a suitable restriction of the map $\uphi_{I_qI_p}$ underlying the coordinate change $\Hat\Phi_{I_qI_p}$ in the given Kuranishi atlas.
The coordinate change $\Hat\Phi_{qp} = (U_{qp},\uphi_{qp})$ is then given by the domain $U_{qp}$ and the restriction of $\uphi_{I_qI_p}$ to $\uU_{qp}\subset \uU_q$.

Further, the weak cocycle condition for $\Kk$ implies the compatibility condition required by \cite{FOOO}, namely for all triples 
$p,q,r\in X$ with  $ q\in F_p$ and $r\in  \psi_q(U_{qp}\cap s_q^{-1}(0))\subset F_q\cap F_p$, the equality
$\uphi_{qp}\circ \uphi_{rq} = \uphi_{rp}$  holds  on the common domain
$\uphi_{rq}^{-1}(\uU_{qp})\cap \uU_{rp}$ of the maps in this equation.

\item 
This atlas satisfies the effectivity condition required by \cite{FOOO} only if the action of $\Ga_p$ on $U_p$ is 
locally effective in the sense that $s_p^{-1}(0)$ has a $\Ga_p$-invariant open neighbourhood that is disjoint from the interior of the fixed point set $\Fix(\ga)\subset U_p$ for each $\ga\in \Ga_p\less \{\id\}$.
\end{itemlist}

With this construction, we lose the distinction between basic charts and transition charts, and also
in general can no longer recover the original  transition charts with their group actions from the Kuranishi structure.  Indeed, \cite{FOOO12} works with a ``good coordinate system" (an analog of our notion of reduction in Definition \ref{def:vicin}), that is defined on the orbifold level, i.e.\ on the level of the intermediate category.
 Though it is not clear how relevant the extra information  contained in a Kuranishi atlas is to 
 the question of how to define Gromov--Witten invariants
 for closed curves,
 it might prove useful in other situations, for example  in the case of orbifold Gromov--Witten invariants, or in the recent work of Fukaya et al \cite{FOOO15} where the authors consider a process that rebuilds a Kuranishi structure  from a coordinate system. 
 Further, our categorical formulation makes it very easy to give an 
  explicit description and construction for sections as in Definition~\ref{def:sect} below.
$\hfill\er$
\end{rmk}

\subsection{Kuranishi cobordisms and concordance}\label{ss:cob}\hspace{1mm}\\ \vspace{-3mm}

This section extends the notions of cobordism and concordance developed in \cite[\S4]{MW1} and  \cite[\S6.2]{MW2} to the case of smooth Kuranishi atlases with nontrivial isotropy. It is a straightforward generalization that can be skipped until precise concordance notions are needed in the proof of Theorem~\ref{thm:K}.
We begin by summarizing the topological cobordism notions from \cite[\S4.1]{MW1}.

A {\bf collared cobordism} $(Y, \io^0_Y,\io^1_Y)$ is a separable, locally compact, metrizable space $Y$ together with disjoint (possibly empty) closed subsets $\p^0 Y,$ $ \p^1 Y\subset Y$ and {\bf collared neighbourhoods}
$$
\io_Y^0:  [0,\eps)\times  \p Y^0  \to Y, \qquad  \io_Y^1:  (1-\eps, 1]\times  \p Y^1  \to Y
$$
for some $\eps>0$. The latter are homeomorphisms onto disjoint open neighbourhoods of $\p^\al Y\subset Y$, extending the inclusions $\io_Y^\al(\al,\cdot) : \p^\al Y\hookrightarrow Y$  for $\al=0,1$.
We call $\p^0 Y$ and $\p^1 Y$ the {\bf boundary components} of $(Y, \io^0_Y,\io^1_Y)$.
The main example is the {\bf trivial cobordism} $Y = [0,1]\times X$ with the natural inclusions $\io_Y^\al: A^\al_\eps\times X \to [0,1]\times X$, where we denote
$$
A_\eps^0: = [0,\eps)  \qquad\text{and} \qquad A_\eps^1: = (1-\eps,1] \qquad\text{for} \ 0<\eps<\tfrac 12 .
$$
Next, a subset $F\subset Y$ is {\bf collared} if there is $0<\delta\le\eps$ such that for $\al=0,1$ we have
\begin{equation}\label{eq:collset}
F \cap \im (\io_Y^\al)\ne \emptyset
\;\; \Longleftrightarrow\;\;
F \cap \im (\io_Y^\al)
= \io_Y^\al( A^\al_\delta\times \p^\al F) ,
\end{equation}
where the intersections with the boundary components $\partial^\al F :=  F \cap \p^\al Y $ may be empty.

In the notion of Kuranishi cobordism, we will require all charts and coordinate changes to be of product form  in sufficiently small collars, as follows.

\begin{defn} \label{def:CCC}  
Let $(Y, \io_Y^0, \io_Y^1)$ be a compact collared cobordism.
\begin{itemlist}
\item
Given a Kuranishi chart $\bK^\al=(U^\al,E^\al,\Ga^\al,s^\al,\psi^\al)$ for $\p^\al Y$ and an open subset $A\subset[0,1]$, the {\bf product chart} for $[0,1] \times \p^\al Y$ with footprint $A\times F^\al$ is
$$
A\times \bK^\al  :=\bigl(A \times U^\al, E^\al, \,\Ga^\al,  s^\al\circ{\rm pr}_{U^\al},\, \id_{A}\times \psi^\al \bigr) ,
$$
where $\Ga^\al$ acts trivially on the first factor of $A\times U^\al$ and ${\rm pr}_{U^\al}:A \times U^\al\to U^\al$ is the evident projection.
\item
Given a coordinate change $\Hat\Phi^\al_{IJ} = (\Tphi_{IJ}^\al, \Hat\phi_{IJ}^\al,\rho_{IJ}^\al) :\bK^\al_I\to\bK^\al_J$ between Kuranishi charts for $\p^\al Y$ with lifted domain $\TU_{IJ}^\al$, 
and open subsets $A_I,A_J\subset[0,1]$, the {\bf product coordinate change}
$(A_I\cap A_J)\times \bK^\al_I \to A_J\times  \bK^\al_J$  
is 
$$
\id_{A_I\cap A_J} \times \Hat\Phi^\al_{IJ}  :  \bigl(\id_{A_I\cap A_J}\times \Tphi_{IJ}^\al, \Hat\phi_{IJ}: = \Hat\phi_{IJ}^\al,\ \id_{A_I\cap A_J}\times \rho_{IJ}^\al\bigl)
$$
with  the lifted domain $(A_I\cap A_J)\times \TU_{IJ}^\al$.
\item
A  {\bf Kuranishi chart with collared boundary} for $(Y, \io_Y^0, \io_Y^1)$
is a tuple $\bK = (U,E ,\Ga, s,\psi )$ as in Definition~\ref{def:chart}, with the following collar form requirements:
\begin{enumerate}
\item
The footprint $F\subset  Y$ is collared with at least one nonempty boundary $\p^\al F$.
\item 
The domain is a collared cobordisms $(U,\io^0_U, \io^1_U)$, whose boundary components $\partial^\al U$ are nonempty iff $\p^\al F\ne \emptyset$. It is smooth in the sense that $U$ is a manifold with boundary $\p U= \p^0 U \sqcup \p^1 U$ and $\io^\al_U$ are tubular neighbourhood diffeomorphisms.
\item 
If $\partial^\al F \neq\emptyset$ then there is a {\bf restriction of $\bK$ to the boundary $\p^\al Y$}, that is a Kuranishi chart $\partial^\al\bK  = (\p^\al U^\al, E, \Ga, s^\al, \psi^\al)$ for $\p^\al Y$ with the isotropy group $\Ga$ and obstruction space $E$ of $\bK$ and footprint $\p^\al F$, and an embedding of the product chart $A_\eps^\al\times \p^\al \bK$ into $\bK$ for some $\eps>0$ in the following sense:
The boundary embedding $\iota_U^\al$ is $\Ga$-equivariant and the following diagrams commute:
 $$
  \begin{array} {ccc}
\phantom{right} A_\eps^\al\times  \p^\al U & \stackrel{\io_U^\al} \longrightarrow & \!\!\!U \\
 s^\al \circ\pr_{\p^\al U} \downarrow&&\downarrow{s}  \\
\phantom{rightright}E & \stackrel{\id_E} \longrightarrow & E \; .
\end{array}
\qquad
  \begin{array} {ccc}
(\id_{A_\eps^\al}\times  s^\al)^{-1}(0) & \stackrel{\io_U^\al} \longrightarrow &{s^{-1}(0)} \\
 \id_{A_\eps^\al}\times \psi^\al\downarrow\;\;\;\;\;&&\downarrow{\psi}  \\
\phantom{right}{A_\eps^\al\times \p^\al Y} & \stackrel{\io_Y^\al} \longrightarrow &{Y} \; .
\end{array}
$$
\end{enumerate}
\item
Let $\bK _I,\bK _J$ be Kuranishi charts for $(Y, \io_Y^0, \io_Y^1)$ 
such that only $\bK _I$ or both $\bK _I,\bK _J$ have collared boundary.
Then a {\bf coordinate change with collared boundary} $\Hat\Phi_{IJ} :\bK _I\to\bK _J$ 
with domain $\uU_{IJ}$ satisfies the conditions in Definition~\ref{def:change}, with the following collar form requirement:
\begin{enumerate}
\item
The lifted domain $\TU _{IJ}\subset U _J$ as well as $U _{IJ}\subset U _I$ are collared subsets.
\item
If $F_J\cap \p^\al Y \ne \emptyset$ then $F_I\cap \p^\al Y \ne \emptyset$ and there is a 
{\bf restriction of $\Hat\Phi_{IJ} $ to the boundary $\p^\al Y$}, that is a
coordinate change $\partial^\al\Hat\Phi_{IJ} : \partial^\al\bK _I \to \partial^\al\bK _J$
such that the restriction of $\Hat\Phi_{IJ}$ to $\und {U_{IJ}\cap \io_{U_I}^\al(A^\al_\eps\times \p^\al U_I)}$  pulls back via the collar inclusions $\io^\al_{U_I},  \io^\al_{U_J}$ to the product coordinate change ${\rm id}_{A^\al_\eps} \times \partial^\al\Hat\Phi_{IJ} $ for some $\eps>0$.
In particular we have 
\begin{align*}
(\iota_{U_J}^\al)^{-1}(\TU _{IJ})
\cap \bigl(A^\al_\eps\times \partial^\al U _J \bigr)
&\;=\; A^\al_\eps\times \partial^\al \TU _{IJ},  \\
(\iota_{U_I}^\al)^{-1}(U_{IJ})
\cap \bigl(A^\al_\eps \times \partial^\al U _I \bigr)
&\;=\;
A^\al_\eps \times \partial^\al U _{IJ} .
\end{align*}
\item
If $F_J\cap \p^\al Y= \emptyset$ but $F_I\cap \p^\al Y\ne \emptyset$, then $U_{IJ}\subset U_I$ is collared with $\p^\al U_{IJ}=\emptyset$. As a consequence we have $U_{IJ} \cap \io_{U_I}^\al(A^\al_\eps\times \p^\al \TU _I)=\emptyset$ for some $\eps>0$.
\end{enumerate}
\end{itemlist}
\end{defn}

\begin{defn}\label{def:CKS}
A {\bf (weak) Kuranishi cobordism} on a compact collared  cobordism  $(Y, \io_Y^0, \io_Y^1)$
is a tuple $\Kk  = \bigl( \bK_{I} , \Hat\Phi_{IJ} \bigr)_{I,J\in \Ii_{\Kk}}$
of basic charts and transition data as in Definition~\ref{def:Ku}  with the following collar form requirements:
\begin{itemlist}
\item
The charts of 
$\Kk$  are either Kuranishi charts with collared boundary or standard Kuranishi charts whose footprints are precompactly contained in 
$Y\less  (\p^0 Y\cup \p^1 Y)$.
\item
The coordinate changes  
$\Hat\Phi_{IJ}: \bK_{I} \to \bK_{J}$ 
are either standard coordinate changes on $Y\less  (\p^0 Y\cup \p^1 Y)$  
between pairs of standard charts, or coordinate changes with collared boundary between 
pairs of charts, of which at least the first has collared boundary.
\end{itemlist}
We say that $\Kk$ has {\bf uniform collar width $\de> 0$} if all domains and coordinate changes have the required collar form over intervals $A^\al_\eps$ of length $\eps> \de$.
\end{defn}

\begin{rmk}\label{rmk:restrict}\rm 
Let $\Kk$ be a (weak) Kuranishi cobordism on  $(Y, \io_Y^0,\io_Y^1)$.

\noindent
(i)
$\Kk$ induces by restriction 
(weak) Kuranishi atlases $\partial^\al\Kk$ on the boundary components $\p^\al Y$ for $\al=0,1$ with
\begin{itemlist}
\item 
basic charts $\p^\al\bK_i$ given by restriction of  basic charts of $\Kk$ with $F_i\cap \p^\al Y\neq\emptyset$;
\item 
index set $\Ii_{\p^\al\Kk}=\{I\in\Ii_{\Kk}\,|\, F_I\cap  \p^\al Y\neq\emptyset\}$;
\item 
transition charts $\p^\al\bK_I$ given by restriction of transition charts of $\Kk$;
\item
coordinate changes $\p^\al\Hat\Phi_{IJ}$ given by restriction of coordinate changes of $\Kk$.
\end{itemlist}

\noindent
(ii)
The charts and coordinate changes of $\Kk$ induce intermediate charts and coordinate changes as in Definition~\ref{def:quotlev} and Remark~\ref{rmk:change}~(iii). These fit together to form a filtered (weak) topological cobordism $\uKk$ in the sense of \cite[Definitions~4.1.12]{MW1} by a direct generalization of Lemma~\ref{le:Ku3}.
Its boundary restrictions are the intermediate Kuranishi atlases $\p^\al\uKk = \und{\p^\al\Kk}$ induced by the boundary restrictions $\p^\al\Kk$. 

\noindent
(iii)
As in \cite[Remark~4.1.11]{MW1} we can think of the virtual neighbourhood $|\Kk|$ as a collared cobordism with boundary components $\p^0|\Kk|\cong|\p^0\Kk|$ and $\p^1|\Kk|\cong|\p^1\Kk|$, with the exception that $|\Kk|$ is usually not locally compact or metrizable.
More precisely, if $\Kk$ has collar width $\eps>0$, then the inclusions $\iota^\al_{U_I} : A^\al_\eps\times U^\al_I \hookrightarrow U_I$ induce topological embeddings
$$
\iota_{|\Kk|}^0: [0,\eps) \times  |\p^0\Kk|  \hookrightarrow |\Kk|,  
\qquad
\iota_{|\Kk|}^1: (1-\eps,1]\times  |\p^1\Kk|   \hookrightarrow |\Kk| 
$$
to open neighbourhoods of the closed subsets $\p^\al|\Kk| := \quo{\bigsqcup_{I\in\Ii_{\p^\al\Kk}} \io^\al_{U_I}(\{\al\}\times U^\al_I)}{\sim} \subset |\Kk|$. 
$\hfill\er$
\end{rmk}

With this language in hand, one obtains cobordism relations between (weak) Kuranishi atlases in complete 
analogy with \cite[Definition~4.1.8]{MW1} and \cite[Definition~6.2.10]{MW2}.
For the uniqueness results in this paper, the more important notion is the following.
Here we use the notion of tameness, a refinement of the strong cocycle condition that is formalized in Definition~\ref{def:tame} below.

\begin{defn}\label{def:Kcobord}
Two (weak/tame) Kuranishi atlases $\Kk^0, \Kk^1$ on the same compact metrizable space $X$ are said to be  {\bf (weakly/tamely) concordant} if there exists a (weak/tame) Kuranishi cobordism $\Kk$ on the trivial cobordism $Y = [0,1]\times X$ whose boundary restrictions are $\p^0\Kk=\Kk^0$ and $\p^1\Kk=\Kk^1$.
More precisely, there are injections $\iota^\al:\Ii_{\Kk^\al} \hookrightarrow \Ii_{\Kk}$ for $\al=0,1$ such that $\im\iota^\al=\Ii_{\partial^\al\Kk}$ and we have
$$
\bK^\al_I = \p^\al \bK_{\iota^\al(I)}, \qquad
\Hat\Phi^\al_{IJ} = \p^\al \Hat\Phi_{\iota^\al(I) \iota^\al (J)} 
\qquad
\forall  I,J\in\Ii_{\Kk^\al}.
$$
Moreover, two metric Kuranishi atlases $(\Kk^0,d_0), (\Kk^1,d_1)$ are called {\bf metric concordant} if they are concordant as above with $\Kk$ a Kuranishi cobordism whose realization $|\Kk|\cong|\uKk|$ supports an admissible, $\eps$-collared metric $d$ in the sense of \cite[Definition~4.2.1]{MW1} for the intermediate cobordism atlas $\uKk$ such that $d|_{\p^\al|\Kk|} = d_\al$ for $\al=0,1$.
 \end{defn}

\subsection{Tameness and shrinkings}\label{ss:tame}\hspace{1mm}\\ \vspace{-3mm}

As in the case of trivial isotropy, we must adjust the Kuranishi atlas in order for its realization $|\Kk|$ to have good topological properties, for example, so that it is Hausdorff and has ``enough" compact subsets.
We essentially already dealt with these problems in \cite{MW1} by 
\begin{itemlist}
\item
introducing notions of tameness and preshrunk shrinking for topological Kuranishi atlases which ensure the desired topological properties of the realization;
\item
constructing tame shrinkings of filtered weak topological Kuranishi atlases;
\item
proving that tame shrinkings are unique up to tame concordance.
\end{itemlist}
In order to apply these results to smooth Kuranishi atlases with nontrivial isotropy, recall first that we built additivity into the notion of Kuranishi atlas, and showed in Lemma~\ref{le:Ku3} that the resulting intermediate atlases are naturally filtered by $\bigl(\uEE_{IJ}: = \und{U_{J}\times \Hat\phi_{IJ}(E_I)}\bigr)_{I\subset J}$.
The same holds for Kuranishi cobordisms by Remark~\ref{rmk:restrict}~(ii).
We can thus extend the notions of tameness to the case of nontrivial isotropy by working at the level of the intermediate category.

\begin{defn} \label{def:tame}
A weak Kuranishi  atlas or cobordism is {\bf tame} if 
its intermediate atlas is tame in the sense of \cite[Definition~3.1.10]{MW1},
that is 
for all $I,J,K\in\Ii_\Kk$ we have
\begin{align}\label{eq:tame1}
\und{U}_{IJ}\cap \und{U}_{IK}&\;=\; \und{U}_{I (J\cup K)}\qquad\qquad\quad\;\;\,\forall I\subset J,K ;\\
\label{eq:tame2}
\und{\phi}_{IJ}(\und{U}_{IK}) &\;=\; \uU_{JK}\cap \us_J^{-1}\bigl( \uEE_{IK}\bigr) \qquad\forall I\subset J\subset K.
\end{align}
Here we allow equalities between $I,J,K$, using the notation 
$\und{U}_{II}:=\und{U_I}$ and $\und{\phi}_{II}:={\rm Id}_{\und{U}_I}$.
\end{defn}

Similarly, a shrinking of a Kuranishi atlas or cobordism will arise exactly from a shrinking $\bigl(\uU'_I\sqsubset \uU_I\bigr)_{I\in\Ii_\Kk}$ of the intermediate atlas in the sense of \cite[Definition~3.3.2]{MW1}.
Recall here that shrinkings of cobordisms are necessarily given by collared subsets $\uU'_I\sqsubset \uU_I$.

\begin{defn}
\label{def:shr}
Let $\Kk=(\bK_I,\Hat\Phi_{I J})_{I, J\in\Ii_\Kk, I\subsetneq J}$ be a weak Kuranishi atlas or cobordism. 
A {\bf shrinking} of $\Kk$ is a weak Kuranishi atlas resp.\ cobordism
$\Kk'=(\bK_I',\Hat\Phi_{I J}')_{I, J\in\Ii_{\Kk'}, I\subsetneq J}$ as follows:
\begin{enumerate}
\item 
The footprint cover $(F_i')_{i=1,\ldots,N}$ is a shrinking of the cover $(F_i)_{i=1,\ldots,N}$, that is 
$F_i'\sqsubset F_i$ are precompact open subsets so that $X=\bigcup_{i=1,\ldots,N}F_i'$ and $F_I': = \bigcap_{i\in I} F_i'$ is nonempty whenever $F_I$ is, so that the index sets $\Ii_{\Kk'} = \Ii_\Kk$ agree.
\item
For each $I\in\Ii_\Kk$ the chart $\bK'_I$ is the restriction of $\bK_I$ to a precompact domain $\und U_I'\sqsubset \und U_I$ as in Definition \ref{def:restr}.
\item
For each $I,J\in\Ii_\Kk$ with $I\subsetneq J$ the coordinate change $\Hat\Phi_{IJ}'$ is the restriction of $\Hat\Phi_{IJ}$  to the open subset $\und U'_{IJ}: =  \und{\phi}_{IJ}^{-1}(\und U'_J)\cap \und U'_I$
 as in equation \eqref{eq:coordres}.
\end{enumerate}

Moreover, a {\bf tame shrinking} of $\Kk$ is a shrinking that is tame in the sense of Definition~\ref{def:tame}.
Finally, a {\bf preshrunk tame shrinking} of $\Kk$ is a tame shrinking $\Kk''$ that is obtained as shrinking of a tame shrinking $\Kk'$ of $\Kk$.
\end{defn}

With this language in place, we can directly generalize \cite[Theorem~6.3.9]{MW2}.
Recall here that by \cite[Example~2.4.5]{MW1} the quotient topology on $|\Kk|$ is never metrizable except in the most trivial cases. In fact, for any point $x\in \ov{U_{IJ}}\less U_{IJ}$ where $\dim U_I <\dim U_J$ the projection $\pi_\Kk(x)$ does not have a  countable neighbourhood basis in $|\Kk|$ with respect to the quotient topology. So an admissible metric will almost always induce a different topology on $|\Kk|$, which we will make no use of in the following statement.

\begin{thm}\label{thm:K}
\begin{enumerate}\item
Any weak Kuranishi atlas or cobordism $\Kk$ has a preshrunk tame shrinking $\Kk'$.
\item
For any tame Kuranishi atlas or cobordism $\Kk'$, the realizations $|\Kk'|$ and $|\bE_{\Kk'}|$ are Hausdorff in the quotient topology, and for each $I\in \Ii_{\Kk'}$ the projection maps $\pi_{\uKk'}: \uU_I'\to |\Kk'|$ and $\pi_{\bE_{\uKk'}}:\und {U'_I\times E_I}\to |\bE_{\Kk'}|$ are homeomorphisms onto their images.
\item
For any preshrunk tame shrinking $\Kk'$ as in (i), there exists an admissible metric on the set $|\Kk'|$.
If $\Kk$ is a cobordism, then the metric can also be taken to be collared.
\item
Any two metric preshrunk tame shrinkings of a weak Kuranishi atlas are metric tame concordant.
\end{enumerate}
\end{thm}

\begin{proof}
Since tameness, shrinking, and admissible metrics are all defined on the level of intermediate atlases, and we are only concerned with homeomorphism properties of the intermediate projections, we can simply quote in the case of Kuranishi atlases
\cite[Proposition~3.3.5]{MW1} for (i), 
\cite[Proposition~3.1.13]{MW1} for (ii), and
\cite[Proposition~3.3.8]{MW1} for (iii).
Moreover, \cite[Proposition~4.2.3]{MW1} proves (iv) as well as (i), (iii) for Kuranishi cobordisms, and (ii) for 
cobordisms is established in \cite[Lemma~4.1.15]{MW1}.
\end{proof}

\section{From Kuranishi atlases to the Virtual Fundamental Class}

In this section, \S\ref{ss:orient} discusses orientations,  \S\ref{ss:red} establishes the notions of reductions and perturbations. The main result here is Theorem~\ref{thm:zero} which shows that the zero set of a suitable
perturbation $\s_\Kk + \nu$ of the canonical section $\s_\Kk$ has the structure of a compact weighted branched manifold. 
The construction of such perturbations is deferred to Proposition~\ref{prop:ext}, and is followed by the construction of VMC and VFC in Theorem~\ref{thm:VMCF}.

\subsection{Orientations} \label{ss:orient}   \hspace{1mm}\\ \vspace{-3mm}

This section extends the theory of orientations of weak Kuranishi atlases from \cite[\S8.1]{MW2} to the case with nontrivial isotropy. Since we use the method of determinant bundles, we first need to generalize the notions of vector bundles and isomorphisms.

\begin{defn} \label{def:bundle}
A {\bf vector bundle} $\La=\bigl(\La_I,\Ti\phi_{IJ}^\La\bigr)_{I,J\in\Ii_\Kk}$ {\bf over a weak Kuranishi atlas} $\Kk$  consists of local bundles and compatible transition maps as follows:
\begin{itemlist}
\item
For each $I\in \Ii_\Kk$ a vector bundle $\La_I \to U_I$ with an action of $\Ga_I$ on $\La_I$ that covers the given action on $U_I$.
\item
For each $I\subsetneq J$ a $\Ga_J$-equivariant map
$\Ti\phi_{IJ}^\La:  \rho^*_{IJ}(\La_I|_{U_{IJ}})\to \La_J$
that is a linear isomorphism on each fiber and covers the embedding $\Tphi_{IJ}: \TU_{IJ}\to U_J$.
Here $\Ga_J\cong\Ga_I\times \Ga_{J\less I}$ acts on $\rho^*_{IJ}(\La_I|_{U_{IJ}}) \to \TU_{IJ}$ by the pullback action of $\Ga_I$ together with the natural identification of the fibers of $\rho^*_{IJ}(\La_I|_{U_{IJ}})$ along $\Ga_{J\less I}$-orbits in $\TU_{IJ}$.
\item
For each $I\subsetneq J\subsetneq K$, we have the weak cocycle condition 
$$
\Tilde \phi^\La_{IK} =  \Tilde \phi^\La_{JK}\circ \rho_{JK}^*( \Tilde \phi^\La_{IJ})\;\mbox{ on } \;\rho^{-1}_{JK}(\Tphi_{IJ}(\TU_{IJ}))\cap \TU_{IK}.
$$ 
\end{itemlist}  
A {\bf section} of a vector bundle $\La$ over $\Kk$ is a collection of smooth $\Ga_I$-equivariant sections $\si=\bigl( \si_I: U_I\to \La_I \bigr)_{I\in\Ii_\Kk}$ that are compatible with the pullbacks $\rho_{IJ}^*$ and bundle maps $\Ti\phi^\La_{IJ}$ in the sense that there are commutative diagrams for each $I\subsetneq J$,
\[
\xymatrix{
\La_I|_{U_{IJ}}    &   \ar@{->}[l]_{\rho_{IJ}}
  \rho_{IJ}^*(\La_I|_{U_{IJ}})   \ar@{->}[r]^{\;\;\;\;\;\;\Tilde\phi^\La_{IJ}}    &  \La_J  \\
U_{IJ} \ar@{->}[u]^{\si_I}       &  \ar@{->}[l]_{\rho_{IJ}}
  \TU_{IJ}  \ar@{->}[u]^{\rho_{IJ}^*(\si_I)}   \ar@{->}[r]^{\Tphi_{IJ}}  & U_J  \ar@{->}[u]_{\si_J} .
}
\]
\end{defn}

\begin{defn} \label{def:prodbun}
If $\La=\bigl(\La_I,\Ti\phi_{IJ}^\La\bigr)_{I,J\in\Ii_\Kk}$ is a bundle over $\Kk$ and $A\subset [0,1]$ is an interval, then the {\bf product bundle} $
A\times \La$ over $A\times \Kk$ is the tuple $\bigl(A\times \La_I,\id_A\times\Ti\phi_{IJ}^\La\bigr)_{I,J\in\Ii_\Kk}$. 
Here and in the following we denote by $A\times \La_I\to A\times U_I$ the pullback bundle of $\La_I\to U_I$ under the projection $\pr_{U_I}: A\times U_I\to U_I$.
\end{defn}

\begin{defn} \label{def:cbundle}  
A {\bf vector bundle over a weak Kuranishi cobordism} $\Kk$ 
is a collection $\La= \bigl(\La_I,\Ti\phi_{IJ}^\La\bigr)_{I,J\in\Ii_\Kk}$
of vector bundles and bundle maps as in Definition~\ref{def:bundle}, together with a choice of isomorphism from its collar restriction to a product bundle.
More precisely, this requires for $\al=0,1$ the choice of a {\bf restricted vector bundle} $\La|_{\p^\al\Kk}= \bigl( \La^\al_I \to \partial^\al U_I, \Ti\phi^{\La,\al}_{IJ}\bigr)_{I,J \in \Ii_{\p^\al\Kk}}$ over $\p^\al\Kk$, and, for some $\eps>0$ less than the collar width of $\Kk$, a choice of lifts of the embeddings $\io^\al_I$ for $I\in\Ii_{\p^\al\Kk}$ to $\Ga_I$-equivariant bundle isomorphisms 
$\ti\io^{\La,\al}_I : A^\al_\eps\times \La^\al_I \to \La_I|_{\im\io^\al_I}$ such that, with $A: = A^\al_\eps$
and 
$\rho^*\ti\io^{\La,\al}_I :=   \rho_{IJ}^* \circ \ti\io^{\La,\al}_I \circ \bigl(\id_A \times (\rho_{IJ}^\al)_*\bigr)$
the following diagrams commute
\[
\xymatrix{ A\times \La_I^\al \ar@{->}[d]   \ar@{->}[r]^
{\ti\io^{\La,\al}_I}    &   \La_I|_{\im\io^\al_I} \ar@{->}[d]   \\
A\times  \partial^\al U_I \ar@{->}[r]^{\io^\al_I}   & \im\io^\al_I \subset U_I
}
\qquad\qquad
\xymatrix{
A\times    (\rho_{IJ}^\al)^*\bigl(\La_I^\al|_{\p^\al U_{IJ}}\bigr) 
  \ar@{->}[r]
   ^{\rho^*\ti\io^{\La,\al}_I} \ar@{->}[d]_  
 {\id_A\times \Ti\phi^{\La,\al}_{IJ}}  & 
  \rho_{IJ}^* \bigl(\La_I |_{\io^\al_I(A\times \p^\al U_{IJ})}\bigr)
  \ar@{->}[d]^
  {\Ti\phi^\La_{IJ}}  \\
 A\times   \La_J^\al\ar@{->}[r]^{\ti\io^{\La,\al}_{J}}
    &  \La_J|_{\im\io^\al_J}  
}
\]

A {\bf section} of a vector bundle $\La$ over a Kuranishi cobordism as above is a compatible collection $\bigl(\si_I:U_I\to \La_I\bigr)_{I\in\Ii_\Kk}$ of  
equivariant 
sections as in Definition~\ref{def:bundle} that in addition have product form in the collar. 
That is we require that for each $\al=0,1$ there is a {\bf restricted section} $\si|_{\p^\al\Kk}= ( \si^\al_I :\partial_\al U_I \to \La^\al_I)_{I\in\Ii_{\p^\al\Kk}}$ of $\La|_{\p^\al\Kk}$ such that for $\eps>0$ sufficiently small we have $(\ti\io^{\La,\al}_I)^*\si_I   = \id_{A^\al_\eps}\times  \si^\al_I$.
 \end{defn}

In the above definition,  we implicitly work with an isomorphism $(\ti\io^{\La,\al}_I)_{I\in\Ii_{\p^\al\Kk}}$, that satisfies all but the product structure requirements of the following notion of isomorphisms on Kuranishi cobordisms.

\begin{defn} \label{def:buniso}
An {\bf isomorphism} $\Psi: \La\to \La'$ between vector bundles over $\Kk$ is a collection
$(\Psi_I: \La_I\to \La'_I)_{I\in \Ii_\Kk}$ of $\Ga_I$-equivariant bundle isomorphisms covering the identity on $U_I$, that intertwine the transition maps,
i.e.\ $\Ti\phi^{\La'}_{IJ}\circ\rho^*_{IJ}(\Psi_I) = \Psi_J \circ \Ti \phi^\La_{IJ}|_{\TU_{IJ}}$
for all $I\subsetneq J$.

If $\Kk$ is a Kuranishi cobordism then we additionally require $\Psi$ to have product form in the collar. That is we require that for each $\al=0,1$ there is a restricted isomorphism $\Psi|_{\p^\al\Kk}= ( \Psi^\al_I :\La^\al_I \to \La'_I\,\!\!^\al)_{I\in\Ii_{\p^\al\Kk}}$ from $\La|_{\p^\al\Kk}$ to $\La'|_{\p^\al\Kk}$ such that for $\eps>0$ sufficiently small we have 
$(\ti\io'_I)^{\La,\al} \circ \bigl(\id_A\times \Psi^\al_I \bigr) = \Psi_I \circ \ti\io^{\La,\al}_I$ 
on $A^\al_\eps\times \partial^\al U_I $.
\end{defn}

%
%
%
%

Note that, although the compatibility conditions are the same, the canonical section 
$\s_\Kk = ( s_I : U_I\to E_I)_{I\in\Ii_\Kk}$ of a Kuranishi atlas does not form a section of a 
vector bundle since the obstruction spaces $E_I$ are in general not of the same dimension, 
hence no bundle isomorphisms $\Ti\phi_{IJ}^\La$ as above exist.
Nevertheless, we will see that, 
there is a natural bundle associated with the section $\s_\Kk$, namely its determinant line bundle, 
and that this line bundle is isomorphic to a bundle constructed by combining the determinant lines of the obstruction spaces $E_I$ and the domains $U_I$.
 
\begin{rmk}\rm 
If $\La$ is a bundle over 
a Kuranishi atlas $\Kk$ (rather than a weak atlas),
then it is straightforward to verify that the union $\sqcup_I \La_I$ of the  local bundles form the objects of a category with projection to the Kuranishi category $\bB_\Kk$.  We did not formulate the above definitions in this language since 
orientations in applications to moduli spaces (e.g.\ Gromov--Witten as in \cite{MW:gw}) 
will usually be constructed on a weak Kuranishi atlas, which does not form a category.
\hfill$\er$
\end{rmk}
 
Here and in the following we will exclusively work with finite dimensional vector spaces.
First recall that the determinant line of a vector space $V$ is its maximal exterior power $\lm V := \wedge^{\dim V}\,V$, with $\wedge^0\,\{0\} :=\R$.
More generally, the {\bf determinant line of a linear map} $D:V\to W$ is defined to be 
\begin{equation}\label{eq:detlin}
\det(D):= \lm\ker D \otimes \bigl( \lm \bigl( \qu{W}{\im D} \bigr) \bigr)^*.
\end{equation}
In order to construct isomorphisms between determinant lines, we will need to fix various conventions, in particular pertaining to the ordering of factors in their domains and targets.
We begin by noting that every isomorphism $F: Y \to Z$ between finite dimensional vector spaces induces an isomorphism
\begin{equation}\label{eq:laphi}
\La_F :\; \lm Y   \;\overset{\cong}{\longrightarrow}\; \lm Z , \qquad
y_1\wedge\ldots \wedge y_k \mapsto F(y_1)\wedge\ldots \wedge F(y_k) .
\end{equation}
For example, 
the fact that  $ \ga\circ s_I: = s_I\circ \ga:U_I\to E_I$ for all $\ga\in \Ga_I$ implies that
\begin{equation}\label{eq:bunga}
\ga_\La: = \La_{\rd_x \ga}\otimes \bigl(\La_{[\ga]^{-1}}  \bigr)^*: \det(\rd_x s_I)\to \det(\rd_{\ga x} s_I)
\end{equation}
is an isomorphism, where $[\ga]: \qu{E_I}{\coker \rd_{ x} s_I}\to \qu{E_I}{\coker \rd_{\ga x} s_I}$ is the induced map.
Further,
if $I\subsetneq J$ and  $\Tx\in \TU_{IJ}$ is such that $\rho_{IJ}(\Tx) = x$, 
then because 
$$
s_I\circ \rho_{IJ}=:s_{IJ}: \TU_{IJ}\to E_I
$$ 
the derivative $\rd_{\Tx}\rho_{IJ}:  \ker \rd s_{IJ}\to \ker \rd s_I$ induces an isomorphism 
$$
\La_{\rd_\Tx\rho_{IJ}}\otimes \La_{\Id}:\det(\rd_\Tx s_{IJ})  \to \det(\rd_x s_{I})
$$
and composition with pullback by $\rho_{IJ}$ defines an isomorphism
\begin{equation}\label{eq:bunrho}
P_{IJ}(\Tx) : \det(\rd_\Tx s_{IJ})  \to \rho_{IJ}^*(\det \rd_x s_I) .
\end{equation}
Further, it follows from 
the index condition in Definition~\ref{def:change} that  with $y:= \Tphi_{IJ}(\Tx)$ the map 
\begin{equation}\label{eq:bunIJ}
\TLa_{IJ}(\Tx): = \La_{\rd_\Tx\Tphi_{IJ}} \otimes 
\bigl(\La_{[\Hat\phi_{IJ}]^{-1}}\bigr)^*
\, :\; \det(\rd_\Tx s_{IJ}) \to \det(\rd_{y} s_J)
\end{equation}
is an isomorphism, induced by the isomorphisms $\rd\Tphi_{IJ}:\ker\rd s_{IJ}\to\ker\rd s_J$ and
$[\Hat\phi_{IJ}] : \qu{E_I}{\im\rd s_I}\to\qu{E_J}{\im\rd s_J}$.  
We can therefore define the determinant bundle $\det(\s_\Kk)$ of a Kuranishi atlas. A second, isomorphic, determinant line bundle $\det(\Kk)$ with fibers $\lm \rT_x U_I \otimes \bigl( \lm E_I \bigr)^*$ will be constructed in Proposition~\ref{prop:orient}.

\begin{defn} \label{def:det} 
The {\bf determinant line bundle} of a weak Kuranishi atlas (or cobordism) $\Kk$ is the vector 
bundle $\det(\s_\Kk)$ given by the line bundles 
$$
\det(\rd s_I):=\bigcup_{x\in U_I} \det(\rd_x s_I) \;\to\; U_I \qquad 
\text{for}\; I\in\Ii_\Kk, 
$$
with $\Ga_I$ actions given by the isomorphisms $\ga_\La$ of \eqref{eq:bunga},
and the isomorphisms 
$\Tphi^\La_{IJ}(\Tx):= \La_{IJ}(\Tx) \circ P_{IJ}(\Tx)^{-1}$  
in \eqref{eq:bunrho} and \eqref{eq:bunIJ} for $I\subsetneq J$ and 
$\Tx\in U_{IJ}$.
\end{defn}

To show that  $\det(\s_\Kk)$ is well defined, in particular that 
$\Tx\mapsto \La_{IJ}(\Tx)$
 is smooth, 
we introduce some further natural\footnote{
Here a ``natural" isomorphism is one that is functorial, i.e.\ it commutes with 
the action on both sides induced by a vector space isomorphism.}
isomorphisms 
and fix various ordering conventions.

\begin{itemlist}
\item
For any subspace $V'\subset V$ the {\bf splitting isomorphism}
\begin{equation}\label{eq:VW}
\lm V\cong \lm V'\otimes \lm\bigl( \qu{V}{V'}\bigr)
\end{equation}
is given by completing a basis $v_1,\ldots,v_k$ of $V'$ to a basis $v_1,\ldots,v_n$ of $V$ and
mapping $v_1\wedge \ldots \wedge v_n \mapsto (v_1\wedge \ldots \wedge v_k) \otimes ([v_{k+1}]\wedge \ldots\wedge [v_n])$.
\item
For each isomorphism $F:Y\overset{\cong}{\to} Z$ the {\bf contraction isomorphism} 
\begin{equation} \label{eq:quotable}
\mathfrak{c}_F \,:\; \lm Y  \otimes  \bigl( \lm Z \bigr)^* \;\overset{\cong}{\longrightarrow}\; \R , 
\end{equation}
is given by the map
$\bigl(y_1\wedge\ldots \wedge y_k\bigr) \otimes \eta \mapsto \eta\bigl(F(y_1)\wedge \ldots 
\wedge F(y_k)\bigr)$.
\item
For any space $V$ we use the {\bf duality isomorphism}
\begin{equation}\label{eq:dual} 
\lm V^* \;\overset{\cong}{\longrightarrow}\; (\lm V)^*, \qquad
 v_1^*\wedge\dots\wedge v_n^* 
\;\longmapsto\;
(v_1\wedge\dots\wedge v_n)^* ,
\end{equation}
which corresponds to the natural pairing
$$
 \lm V \otimes \lm V^*   \;\overset{\cong}{\longrightarrow}\;  \R , \qquad
 \bigl(v_1\wedge\dots\wedge v_n\bigr) \otimes \bigl(\eta_1\wedge\dots\wedge \eta_n\bigr)
 \;\mapsto\;
 \prod_{i=1}^n \eta_i(v_i) 
$$
via the general identification (which in the case of line bundles $A,B$ maps $\eta\neq0$ to a nonzero homomorphism, i.e.\ an isomorphism)
\begin{equation}\label{eq:homid}
\Hom(A\otimes B,\R) \;\overset{\cong}{\longrightarrow}\; \Hom(B, A^*)
,\qquad
H \;\longmapsto\; \bigl( \; b \mapsto H(\cdot \otimes b) \; \bigr) .
\end{equation}
\end{itemlist}

\MS\NI
Next, we combine the above isomorphisms to obtain a more elaborate 
contraction isomorphism.

\begin{lemma}[Lemma~8.1.7  in \cite{MW2}]\label{lem:get} 
Every linear map $F:V\to W$ together with an isomorphism $\phi:K\to \ker F$ induces an isomorphism
\begin{align}\label{Cfrak}
\mathfrak{C}^{\phi}_F \,:\; \lm V \otimes \bigl(\lm W \bigr)^* 
&\;\overset{\cong}{\longrightarrow}\;  \lm K \otimes \bigl(\lm \bigl( \qu{W}{F(V)}\bigr) \bigr)^*  
\end{align}
given by
\begin{align}
(v_1\wedge\dots v_n)\otimes(w_1\wedge\dots w_m)^* &\;\longmapsto\;
\bigl(\phi^{-1}(v_1)\wedge\dots \phi^{-1}(v_k)\bigr)\otimes \bigl( [w_1]\wedge\dots [w_{m-n+k}] \bigr)^* ,
\notag
\end{align}
where $v_1,\ldots,v_n$ is a basis for $V$ with ${\rm span}(v_1,\ldots,v_k)=\ker F$, and $w_1,\dots, w_m$ is a basis for $W$ whose last $n-k$ vectors are $w_{m-n+i}=F(v_i)$ for $i=k+1,\ldots,n$.

In particular, for every linear map $D:V\to W$ we may pick $\phi$ as the inclusion $K=\ker D\hookrightarrow V$ to obtain an isomorphism
$$
\mathfrak{C}_{D} \,:\;  \lm V \otimes \bigl(\lm W \bigr)^* \;\overset{\cong}{\longrightarrow}\;  \det(D) .
$$ 
\end{lemma}

\begin{rmk}\label{rmk:get}\rm 
 If $F$ is equivariant with respect to  actions of the group $\Ga$ on $V,W$ and we equip $K$ with the induced $\Ga$ action so that $\phi$ is also equivariant,
 then  the above isomorphism $\mathfrak{C}^{\phi}_F$ is  equivariant with respect to the action of $\Ga$ on $\lm V \otimes \bigl(\lm W \bigr)^* $ given  by the maps $\La_\ga\otimes (\La_{\ga^{-1}})^*
$  on  $\lm V \otimes \bigl(\lm W \bigr)^*$ and by the corresponding maps $\La_\ga\otimes (\La_{[\ga]^{-1}})^*
$ on $\lm K \otimes \bigl(\lm \bigl( \qu{W}{F(V)}\bigr) \bigr)^* $, where  $\La_{[\ga]}$ is as in \eqref{eq:bunga}.
\hfill$\er$
\end{rmk}

With these notations in hand, we can now prove one of the main results of this section.

\begin{prop}\label{prop:det0}  
For any weak Kuranishi atlas, $\det(\s_\Kk)$ is a well defined line bundle over $\Kk$.
Further, if $\Kk$ is a weak Kuranishi cobordism, then $\det(\s_\Kk)$ can be given product form on the collar of $\Kk$ with restrictions $\det(\s_\Kk)|_{\p^\al\Kk} = \det(\s_{\p^\al\Kk})$ for $\al= 0,1$.
The required bundle isomorphisms from the product $A^\al_\eps\times \det(\s_{\p^\al\Kk})$ to the collar restriction  $(\io^\al_\eps)^*\det(\s_\Kk)$ are given in \eqref{orient map}.
\end{prop}
\begin{proof}  
We use the same local trivializations of $\det(\rd s_I)$ as in the proof of the analogous result 
 \cite[Proposition~8.1.8]{MW2} for trivial isotropy, and must check that  these are compatible with   
 the isotropy group actions and 
 coordinate changes.  We will begin by defining these trivializations, referring to \cite{MW2} for many details of proofs.

Let $x_0\in U_I$, and  denote its  stabilizer subgroup  in $\Ga_I$ by $\Ga_I^{x_0}$.  
Take a subspace of $E_I$ that covers the cokernel of $\rd_{x_0} s_I$, 
sweep it out to obtain a  $\Ga_I^{x_0}$-invariant subspace $E'\subset E_I$,
then choose an  isomorphism $\R^N\cong E'$ and equip $\R^N$ with the 
pullback action of $\Ga_I^{x_0}$ denoted $(\ga,v)\mapsto \ga\cdot_{x_0} v$.  
The resulting equivariant map $R_I: (\R^N,\Ga_I^{x_0})\to (E_I,\Ga_I^{x_0})$
covers the cokernel of $\rd_{x} s_I$ for all $x$ in some neighbourhood $O$ of $x_0$.

Thus  $\rd_x s_I\oplus R_I$ is surjective for $x\in O$, and as in equation~(8.1.9) of \cite{MW2} 
we may define a trivialization of $\det(\rd s_I)|_O$ as follows:
\begin{align}\notag
\Hat T_{I,x} \,:\; 
\lm \ker(\rd_x s_I \oplus R_I) 
&\;\overset{\cong}{\longrightarrow}\; 
\qquad\qquad \det(\rd_x s_I)  \\ \label{eq:HatTx}
\ov v_1\wedge \ldots \wedge \ov v_n
&\;\longmapsto\; 
 (v_1\wedge\dots v_k)\otimes  \bigl( [R_I(e_1)]\wedge\dots [R_I(e_{N-n+k})] \bigr)^*,
\end{align}
where $\ov v_i=(v_i,r_i)$ is a basis of $\ker(\rd_x s_I \oplus R_I)\subset \rT_x U_I\times \R^N$ such that $v_1,\ldots,v_k$ span $\ker \rd_x s_I$ (and hence $r_1=\ldots=r_k=0$), and $e_{1},\ldots, e_{N}$ is a positively ordered normalized basis of $\R^N$ (that is $e_1\wedge\ldots e_N = 1 \in \R \cong \lm\R^N$) such that $R_I(e_{N-n+i}) = \rd_x s_I(v_i)$ for $i = k+1,\ldots, n$. In particular, the last $n-k$ vectors span $\im \rd_x s_I \cap \im R_I \subset E_I$, and thus the first $N-n+k$ vectors $[R_I(e_1)],\dots, [R_I(e_{N-n+k})]$ span the cokernel $\qu{E_I}{\im\rd_x s_I}\cong\qu{\im R_I}{\im\rd_x s_I\cap \im R_I}$.
We prove in \cite[Proposition~8.1.8]{MW2} that 
 these trivializations do not depend on the choice of injection $R_I:\R^N\to E_I$. In other words,
 if $R_I':\R^{N'}\to E_I$ is another $\Ga_I$-equivariant injection that also maps onto the cokernel of $\rd_{x_0} s_I$, then 
 there is a bundle isomorphism
 $$
 \Psi: \lm \ker(\rd s_I\oplus R_I)|_O\to \lm \ker(\rd s_I\oplus R_I')|_O
 $$
which is necessarily   $\Ga_I$-equivariant and such that
$\Hat T_I = \Hat T_I' \circ \Psi$.
Thus  $\det(\rd s_I)$ is a smooth line bundle over $U_I$ for each $I\in\Ii_\Kk$.

It remains to check that the action $\ga\in \Ga_I$ on $$
\det(\rd s_I) =\lm( \ker \rd s_I)\otimes \lm (\coker \rd s_I)^*
$$ 
is smooth.
We prove this by choosing suitable trivializations near $x_0$ and $\ga x_0$ and then lifting the action of $\ga$ to a smooth action on the  domains
$\ker (\rd s_I\oplus R_I)$ of the trivializations.  
To this end, first consider the trivialization $T_{I,x}$ defined near $x_0\in U_I$ by 
a $\Ga_I^{x_0}$-equivariant injection $R_I: (\R^N,\Ga_I^{x_0}) \to (E_I,\Ga_I^{x_0})$, 
and for $\ga\in \Ga_I$ the associated trivialization 
$T_{I,\ga x}'$ defined near $\ga x_0\in U_I$ by 
$$
R_I': = \ga \circ R_I: (\R^N,\Ga_I^{\ga x_0}) \to (E_I,\Ga_I^{\ga x_0}),
$$
where 
$(\R^N,\Ga_I^{x_0})$ denotes $\R^N$ with the $\Ga_I^{x_0}$-action $\de: v \mapsto \de \cdot_{x_0} v$ and 
$(\R^N,\Ga_I^{\ga x_0})$ denotes $\R^N$ with the $\Ga_I^{\ga x_0}$-action $$
\de': v \mapsto  \de' \cdot_{\ga x_0} v: =  \ga^{-1}\de' \ga \cdot_{x_0} v, 
$$
 which is well defined because conjugation by $\ga$ defines an isomorphism
$
c_\ga: \Ga_I^{\ga x_0}\to \Ga_I^{x_0},\,\de'\mapsto \ga^{-1} \de' \ga.
$
Then 
$R_I' = \ga \circ R_I$ is $\Ga_I^{\ga x_0}$-equivariant because when $\de'\in \Ga_I^{\ga x_0}$
\begin{align*}
R_I'(\de'\cdot_{\ga x_0} v)&  = R_I'(\ga^{-1} \de'\ga\cdot_{x_0} v) = \ga R_I( \ga^{-1} \de'\ga\cdot_{x_0}  v) = 
\ga ({\ga^{-1} \de'\ga}) R_I(v)\\
&  = \ga\ga^{-1} \de' \ga R_I(v)
=   \de' \ga R_I(v) = \de'\circ R_I'(v),
\end{align*}
where the fourth equality holds because the full group $\Ga_I$ acts on $E_I$.
Thus the diagram
\[
\xymatrix{
&(\R^N,\Ga_I^{x_0})\ar@{->}[d]^{(\id,c_{\ga}^{-1})}\ar@{->}[r]^{(R_I,\id)}& (E_I, \Ga_I^{x_0})\ar@{->}[d]^{(\ga,c_{\ga}^{-1})}\\
&(\R^N,\Ga_I^{\ga x_0})\ar@{->}[r]^{(R_I',\id)}& (E_I, \Ga_I^{\ga x_0})
}
\]
commutes; in other words, the action of the element $\ga\in \Ga_I$ on $E_I$ lifts to the identity map of $\R^N$.  Hence the definition \eqref{eq:HatTx} of the maps $\Hat T_{I,x}$ implies that  the following diagram commutes
\[
\xymatrix{
\lm\bigl(\ker (\rd_x s_I\oplus R_I)\bigr)\ar@{->}[d]^{\La_{\rd_x\ga\oplus \id_{\R^N}}}   \ar@{->}[r]^{\;\;\;\qquad\Hat T_{I,x}}    &  
\det \rd_x s_I   \ar@{->}[d]^ {\La_{\rd_x\ga} \otimes (\La_{[ \ga^{-1}]})^*}  \\
\lm\bigl(\ker (\rd_{\ga x} s_I\oplus R_I')\bigr)      
 \ar@{->}[r]^{\qquad\;\;\Hat T_{I,\ga x}}  &  \det \rd_{\ga x} s_I   .
}
\]
Since the map $\La_{\rd_x\ga\oplus \id_{\R^N}}$  is smooth, 
so is $\ga_\La: = \La_{\rd_x\ga} \otimes (\La_{[ \ga^{-1}]})^*$.  
Thus
 $\det(\rd s_I)$ is a $\Ga_I$-equivariant smooth line bundle over $U_I$ for each $I\in\Ii_\Kk$.

Next note that because $\Ga_{J\less I}$ acts freely on $\TU_{IJ}$, the stabilizer subgroup 
$\Ga_J^{\Tx_0}$
of a point $\Tx_0
\in \rho_{IJ}^{-1}(x_0)$ is taken isomorphically to $\Ga_I^{x_0}$ by the projection $\rho_{IJ}^\Ga:\Ga_J\to \Ga_I$.
For simplicity we will identify these groups. Since $s_{IJ}: \TU_{IJ}\to E_I$ is the composite $s_I\circ \rho_{IJ}$, 
we may therefore  trivialize the bundle $\det(\rd_{\Tx} s_{IJ})$ near $\Tx_0\in \rho_{IJ}^{-1}(x_0)$ by
using the same injection $R_I:\R^N\to E_I$, now considered as a $\Ga_J^{\Tx_0}$-equivariant map.  
Since the  diagram
\[
\xymatrix{
\lm\bigl(\ker (\rd_\Tx s_{IJ}\oplus R_I)\bigr)\ar@{->}[d]^{\La_{\rho_{IJ}\oplus \id_{\R^N} }}  \ar@{->}[r]^{\;\;\;\qquad\Hat T_{I,\Tx}}    &  
\det \rd_\Tx s_{IJ}   \ar@{->}[d]^ {\La_{\rho_{IJ}} \otimes (\La_\id)^*}  \\
\lm\bigl(\ker( \rd_{x} s_I\oplus R_I)\bigr)      
 \ar@{->}[r]^{\qquad\;\;\Hat T_{I,x}}  &  \det \rd_{x} s_I   
}
\]
commutes, 
the isomorphism 
$P_{IJ}(\Tx):\det \rd_{\Tx} s_{IJ}\to \det \rd_{x} s_{I}$ of \eqref{eq:bunrho} is  smooth.
Moreover
the equivariance of the covering map $\rho_{IJ}: (\TU_{IJ},\Ga_J) \to (U_{IJ},\Ga_I)$ and
the identity $s_I\circ \rho_{IJ} = s_{IJ}: \TU_{IJ}\to E_I$ imply that it is  equivariant.
Therefore, to complete the proof that the transition maps $\Tphi_{IJ}^\La$ are smooth, we 
must check that the map  
$$
\TLa_{IJ}(\Tx): = \La_{\rd_\Tx\Tphi_{IJ}} \otimes 
\bigl(\La_{[\Hat\phi_{IJ}]^{-1}}\bigr)^*
\, :\; \det(\rd_\Tx s_{IJ}) \to \det(\rd_{y} s_J)
$$
in
\eqref{eq:bunIJ}
is  equivariant and smooth.  Its equivariance follows from the equivariance of  its constituent maps
$\Tphi_{IJ}$ and 
$\Hat\phi_{IJ}$.    To see that it is smooth, it suffices to 
show that the composite
$\La_{IJ}(x) : = \TLa_{IJ}(\rho_{IJ}^{-1}(x))$ is smooth in some neighbourhood $O$ of $x_0\in U_{IJ}$ where $\rho_{IJ}^{-1}: O\to \TU_{IJ}$ is a local inverse for the covering map $\rho_{IJ}$.    But if we define $\phi_{IJ}(x) : = \Tphi_{IJ}(\rho_{IJ}^{-1}(x)): O\to U_J$, then $\La_{IJ}(x) = 
\La_{\rd_x\phi_{IJ}} \otimes 
\bigl(\La_{[\Hat\phi_{IJ}]^{-1}}\bigr)^*$ 
is identical to the map of the same name in \cite[equation~(8.1.11)]{MW2}, so that smoothness follows by the Claim proved as part of \cite[Proposition~5.1.8]{MW2}.
This  completes the proof that $\det(s_\Kk)$ is a vector bundle over $\Kk$.

In the case of a weak Kuranishi cobordism $\Kk$, 
Proposition~8.1.8 in \cite{MW2} also constructs smooth bundle isomorphisms from
the collar restrictions to the product bundles $A^\al_\eps\times \det(\s_{\p^\al\Kk})$
of the form
\begin{equation}\label{orient map}
\ti\io^{\La,\al}_I (t,x)
 :=  \bigl( \La_{\rd_{(t,x)}\io^\al_I} \circ \wedge_1 \bigr) 
\otimes \bigl(\La_{\id_{E_I}}\bigr)^*
\;:\; A^\al_\eps\times  \det(\rd_x s^\al_I) \;\to\;  \det(\rd_{\io^\al_I(x,t)} s_I) ,
\end{equation}
where  $\wedge_1 \,:\; \lm \ker \rd_x s^\al_I \;\to\; \lm \bigl(\R\times  \ker \rd_x s^\al_I \bigr) , 
\eta \;\mapsto\; 1\wedge \eta$.
These are equivariant because they are induced by the equivariant map $\io^\al_I$, and are compatible with the coordinate changes because the collar embeddings $\io^\al_I$ are.  This completes the proof.
\end{proof}

We next use the determinant bundle $\det(s_\Kk)$ to define the notion of an orientation of a Kuranishi atlas.

\begin{defn}\label{def:orient} 
A  weak Kuranishi atlas or Kuranishi cobordism $\Kk$ is {\bf orientable} if there exists a nonvanishing section $\si$ of the bundle $\det(\s_\Kk)$ (i.e.\ with $\si_I^{-1}(0)=\emptyset$ for all $I\in\Ii_\Kk$).
An {\bf orientation} of $\Kk$ is a choice of nonvanishing section $\si$ of $\det(\s_\Kk)$. 
An {\bf oriented Kuranishi atlas or cobordism} is a pair $(\Kk,\si)$ consisting of a  Kuranishi atlas or cobordism and an orientation $\si$ of $\Kk$.

For an oriented Kuranishi cobordism $(\Kk,\si)$, the {\bf induced orientation of the boundary} $\p^\al\Kk$ for $\al=0$ resp.\ $\al=1$ is the orientation of $\p^\al\Kk$,
$$
\p^\al\si \,:=\; \Bigl( \bigl( (\ti\io^{\La,\al}_I)^{-1}  
\circ\si_I \circ \io^\al_I \bigr)\big|_{\{\al\}\times \partial^\al U_I  } \Bigr)_{I\in\Ii_{\p^\al\Kk}}
$$
given by the isomorphism $(\ti\io^{\La,\al}_I)_{I\in\Ii_{\p^\al\Kk}}$ in \eqref{orient map} between a collar neighbourhood of the boundary in $\Kk$ and the product Kuranishi atlas $A^\al_\eps\times \p^\al\Kk$, followed by restriction to the boundary $\p^\al \Kk=\p^\al\bigl(A^\al_\eps\times \p^\al \Kk\bigr)$, where we identify $\{\al\}\times \partial^\al U_I  \cong \partial^\al U_I$.

With that, we say that two oriented weak Kuranishi atlases $(\Kk^0,\si^0)$ and $(\Kk^1,\si^1)$ are {\bf oriented cobordant} if there exists a weak Kuranishi cobordism $\Kk$ from $\Kk^0$ to $\Kk^1$ and a section $\si$ of $\det(\s_{\Kk})$ such that  $\partial^\al\si=\si^\al$ for $\al=0,1$.
\end{defn}

\begin{rmk}\label{rmk:orientb}\rm  
Here we have defined the induced orientation on the boundary $\p^\al \Kk$ of a cobordism so that it is completed to an orientation of the collar by adding the  positive unit vector $1$ along $A^\al_\eps\subset \R$  rather than the more usual outward normal vector.  
In particular, by formula \cite[(8.1.12)]{MW2} $\eta_1,\dots, \eta_n$ is a positively ordered basis for $\rT_x U^\al_I$ exactly if  $1, \eta_1,\dots, \eta_n$ is a positively ordered basis for $\rT_x (A^\al_\eps\times U^\al_I)$.  
$\hfill\er$
\end{rmk}

\begin{lemma}\label{le:cK}
Let $(\Kk,\si)$ be an oriented weak Kuranishi atlas or cobordism. 
\begin{enumerate}\item
The orientation  $\si$ induces a canonical orientation $\si|_{\Kk'}:=(\si_I|_{U'_I})_{I\in\Ii_{\Kk'}}$ on each shrinking $\Kk'$ of $\Kk$ with domains $\bigl(U'_I\subset U_I\bigr)_{I\in\Ii_{\Kk'}}$.
\item
In the case of a Kuranishi cobordism $\Kk$, the restrictions to boundary and shrinking commute, that is
$(\si|_{\Kk'})|_{\p^\al\Kk'} = (\si|_{\p^\al\Kk})|_{\p^\al\Kk'}$.
\item 
In the case of a weak Kuranishi atlas $\Kk$, the orientation $\si$ on $\Kk$ induces an 
orientation  
$\si^{[0,1]}$ 
on $[0,1]\times \Kk$, 
which induces the given orientation $\p^\al\si^{[0,1]}=\si$ of the boundaries $\p^\al([0,1]\times \Kk) = \Kk$ for $\al=0,1$.
\end{enumerate}
\end{lemma}

\begin{proof}  See the proof of Lemma~8.1.11 in \cite{MW2}.
\end{proof}

As in \cite{MW2}, in order to orient the zero sets of a perturbed section $\s_\Kk+ \nu$  
we will work with a ``more universal" determinant bundle $\det(\Kk)$ over $\Kk$ that is
%
%
%
constructed from the determinant bundles of the zero sections in each chart. 
Since the zero section $0_\Kk$ does not satisfy the index condition, 
we need to construct different transition maps for $\det(\Kk)$, which will now depend on the section $\s_\Kk$.
For this purpose, we
again use contraction isomorphisms from Lemma~\ref{lem:get}.  
On the one hand, this provides families of isomorphisms
\begin{equation}\label{Cds}
\mathfrak{C}_{\rd_x s_I} \,:\; \lm \rT_x U_I \otimes \bigl( \lm E_I \bigr)^* \;\overset{\cong}{\longrightarrow}\;  \det(\rd_x s_I)
\qquad\text{for} \; x\in U_I ,
\end{equation}
which, by Remark~\ref{rmk:get}, are equivariant with the respect to the action of $\ga\in \Ga_I$ on 
$\lm \rT U_I \otimes \bigl( \lm E_I \bigr)^*$ given by
\begin{equation}\label{eq:bunga1}
\Hat\ga_\La: = \La_{\rd_x \ga}\otimes \bigl(\La_{\ga^{-1}}  \bigr)^*:  
\; \lm \rT_x U_I \otimes \bigl( \lm E_I \bigr)^*\to 
\lm \rT_{\ga x} U_I \otimes \bigl( \lm E_I \bigr)^*
\end{equation}
and the corresponding action on $\det(\rd_x s_I)$  in equation  \eqref{eq:bunga}.

%
%

On the other hand, recall that the tangent bundle condition \eqref{tbc} implies that $\rd s_J$ restricts to an isomorphism $\qu{\rT_y U_J}{\rd_\Tx\Tphi_{IJ}(\rT_\Tx \TU_{IJ})}\overset{\cong}{\to} \qu{E_J}{\Hat\phi_{IJ}(E_I)}$ for $y=\Tphi_{IJ}(\Tx)$. 
\footnote{
Here and subsequently, we will distinguish between the manifold
$\TU_{IJ}$ and its image $\im \Tphi_{IJ}$ in $U_J$,
denoting points of $\TU_{IJ}$ by $\Tx$, with $y = \Tphi_{IJ}(x)\in U_J$ and $x=\rho_{IJ}(\Tx) \in U_{IJ}$.}
Therefore, if we choose a $\Ga_J$-equivariant smooth normal bundle $N_{IJ} = \bigcup_{y\in \im \Tphi_{IJ}} N_{IJ,y}
\subset  \rT_{y} U_J$ to 
the submanifold 
$\im \Tphi_{IJ} \subset U_J$, then 
the subspaces 
$\rd_y s_J(N_{IJ,y})$ 
form a smooth family of subspaces of $E_J$ that are complements to
$\Hat\phi_{IJ}(E_I)$.
Hence letting $\pr_{N_{IJ}}(y) : E_J \to \rd_y s_J(N_{IJ,y}) \subset E_J$
be the 
smooth family of 
projections with kernel  
$\Hat\phi_{IJ}(E_I)$,
we obtain 
a smooth family of 
linear maps
$$
F_{\Tx} \,:= \;
\pr_{N_{IJ}}(y) 
\circ \rd_y s_J \,:\; \rT_y U_J \;\longrightarrow\; 
E_J
\qquad\text{for}\;  y=  \Tphi_{IJ} (\Tx)
$$
with images $\im F_\Tx=\rd_y s_J(N_{IJ,y})$,
and also isomorphisms to their kernel
$$
\phi_\Tx  \,:= \;  \rd_\Tx\Tphi_{IJ} \,:\; \rT_\Tx \TU_{IJ} \;\overset{\cong}{\longrightarrow}\;  \ker F_\Tx =  \rT_y (\im \Tphi_{IJ}) \;\subset\; \rT_y U_J .
$$
By Lemma~\ref{lem:get} these induce isomorphisms 
$$
\mathfrak{C}^{\phi_\Tx}_{F_\Tx} \,:\; \lm \rT_{\Tphi_{IJ}(\Tx)} U_J \otimes \bigl(\lm E_J \bigr)^* 
 \;\overset{\cong}{\longrightarrow}\;  \lm \rT_\Tx \TU_{IJ} \otimes  \Bigl(\lm \Bigl(\qq{E_J}
{\im F_\Tx}
 \Bigr) \Bigr)^* .
$$
We may combine this with 
the  isomorphism $\lm  \rT_\Tx \TU_{IJ}\to\rho_{IJ}^*\bigl( \lm  \rT_{x} U_{I}\bigr)$
induced by $\rd_\Tx\rho_{IJ}$, where $x:=\rho_{IJ}(\Tx)$, and the dual of the  isomorphism $\lm \bigl(\qu{E_J}
{\rd_y s_J(N_{IJ,y})} \bigr) \cong \lm E_I$ induced 
via \eqref{eq:laphi} by 
$\pr^\perp_{N_{IJ}}(y)  \circ\Hat\phi_{IJ} : E_I \to \qu{E_J} {\rd_y s_J(N_{IJ,y})}$ 
to obtain for each $\Tx\in \TU_{IJ}$  an isomorphism
\begin{align} \label{CIJ}
\Tilde{\mathfrak{C}}_{IJ}(\Tx) \,: \;
 \lm \rT_{y} U_J \otimes \bigl(\lm E_J \bigr)^*  
\;\overset{\cong}{\longrightarrow}\;  \rho_{IJ}^*(\lm \rT_{x} U_I) \otimes \bigl(\lm E_I \bigr)^*  
\end{align}
with
$$
y: = \Tphi_{IJ}(\Tx), \qquad x: = \rho_{IJ}(\Tx),
$$
given by  the composite of $ \mathfrak{C}^{\phi_\Tx}_{F_\Tx}$ with the map 
\begin{align*}
& \bigl(\La_{\rd_\Tx\rho_{IJ}})
\otimes 
\bigl(\La_{(\pr^\perp_{
N_{IJ}}(y)\circ\Hat\phi_{IJ})^{-1}}\bigr)^* : \\
& \qquad  \lm \rT_\Tx \TU_{IJ} \otimes  \Bigl(\lm \Bigl(\qq{E_J}
{\im F_\Tx}
 \Bigr) \Bigr)^*  \;{\longrightarrow}\;
\rho_{IJ}^*(\lm \rT_{ x} U_I) \otimes \bigl(\lm E_I \bigr)^* .
\end{align*}

\begin{prop}\label{prop:orient} 
\begin{enumerate}
\item 
Let $\Kk$ be a weak Kuranishi atlas. 
Then there is a well defined line bundle $\det(\Kk)$ over $\Kk$ given by the line bundles $\La^\Kk_I := \lm \rT U_I\otimes \bigl(\lm E_I\bigr)^* \to U_I$ for $I\in\Ii_\Kk$,
with group actions as in \eqref{eq:bunga1}
 and the transition maps $\Tilde{\mathfrak{C}}_{IJ}^{-1}:
 \rho_{IJ}^*(\La^\Kk_I |_{U_{IJ}}) \to \La^\Kk_J |_{\im\Tphi_{IJ}}$ 
from \eqref{CIJ} for $I\subsetneq J$. 
In particular, the latter isomorphisms are independent of the choice of 
normal bundle $N_{IJ}$.

Furthermore, the contractions $\mathfrak{C}_{\rd s_I}: \La^\Kk_I \to \det(\rd s_I)$ from \eqref{Cds} define 
an isomorphism $\Psi^{s_\Kk}:=\bigl(\mathfrak{C}_{\rd s_I}\bigr)_{I\in\Ii_\Kk}$ from $\det(\Kk)$ to $\det(\s_\Kk)$.
\item 
If $\Kk$ is a weak Kuranishi cobordism, then the determinant bundle $\det(\Kk)$ defined as in (i)  
can be given a product structure on the collar 
such that its 
boundary restrictions are
$\det(\Kk)|{_{\p^\al\Kk}} = \det(\p^\al\Kk)$ for $\al= 0,1$. 

Further, the isomorphism $\Psi^{s_\Kk}: \det(\Kk) \to \det(s_\Kk)$ defined as in (i) has product structure on the collar with restrictions $\Psi^{s_\Kk}|_{\p^\al \Kk}=\Psi^{s_{\p^\al\Kk}}$ for $\al=0,1$.
\end{enumerate}
\end{prop}

\begin{proof}
To begin, note that each $\La^\Kk_I = \lm \rT U_I \otimes \bigl(\lm E_I\bigr)^*$ is a smooth line bundle over $U_I$, since it inherits local trivializations 
from the tangent bundle $\rT U_I\to U_I$. Moreover the action of  $\Ga_I$ on $U_I\times E_I$ induces a smooth action on  $\La_I^\Kk$ given by \eqref{eq:bunga1}
that covers its action on $U_I$.  Thus $\La^\Kk_I\to U_I$ is a smooth $\Ga_I$-equivariant bundle.
We showed in \cite[Proposition~8.1.12]{MW2}  that 
the isomorphisms $\mathfrak{C}_{\rd_x s_I}$ from \eqref{Cds} are smooth in this trivialization,
 where $\det(\rd s_I)$ is trivialized via the maps $\Hat T_{I,x}$ as in Proposition~\ref{prop:det0}.
Since  $\mathfrak C_{\rd s_I}$ is equivariant, we can define preliminary transition maps 
\begin{equation}\label{tiphi}
\Ti\phi^\La_{IJ}:= \mathfrak{C}_{\rd s_J}^{-1} \circ \TLa_{IJ} \circ \rho_{IJ}^*( \mathfrak{C}_{\rd s_I})
\,:\; \rho_{IJ}^*(\La^\Kk_I|_{U_{IJ}})\to \La^\Kk_J \qquad\text{for}\; I\subsetneq J \in \Ii_\Kk
\end{equation}
by the transition maps \eqref{eq:bunIJ} of $\det(\s_\Kk)$, the isomorphisms \eqref{Cds} and the pullback by $\rho_{IJ}$.
These define a line bundle $\La^\Kk:=\bigl(\La^\Kk_I, \Ti\phi^\La_{IJ} \bigr)_{I,J\in\Ii_\Kk}$ since the weak cocycle condition follows directly from that for the $\TLa_{IJ}$. Moreover, this automatically makes the family of bundle isomorphisms $\Psi^\Kk:=\bigl(\Tilde{\mathfrak{C}}_{\rd s_I}\bigr)_{I\in\Ii_\Kk}$ an isomorphism from $\La^\Kk$ to $\det(\s_\Kk)$. 
It remains to see that $\La^\Kk=\det(\Kk)$ and $\Psi^\Kk=\Psi^{\s_\Kk}$, i.e.\ we claim equality of transition maps $\Ti\phi^\La_{IJ}=\Tilde{\mathfrak{C}}_{IJ}^{-1}$. 
This also shows that $\Tilde{\mathfrak{C}}_{IJ}^{-1}$ and thus $\det(\Kk)$ is independent of the choice of normal bundle $N_{IJ}$ in \eqref{CIJ}.

So to finish the proof of (i), it suffices to establish the following commuting diagram at a fixed 
$\Tx\in U_{IJ}$ with $x = \rho_{IJ}(\Tx), y=\Tphi_{IJ}(\Tx)$,
\begin{equation}\label{cclaim}
\xymatrix{
\lm  \rT_x U_I \otimes \bigl( \lm E_I \bigr)^* \quad
 \ar@{->}[r]^{ \qquad\quad\mathfrak{C}_{\rd_x s_I}} 
 &
\quad \det(\rd_x s_I)  \\
\rho_{IJ}^*(\lm \rT_x U_{I}) \otimes \bigl( \lm E_I \bigr)^*\quad
 \ar@{->}[r]^{\qquad\quad\rho_{IJ}^*(\mathfrak{C}_{\rd_x s_I})} 
   \ar@{->}[u]^{\rho_{IJ}}  
&
\quad\rho_{IJ}^*(\det(\rd_{x} s_I)) \ar@{->}[u]^{\rho_{IJ}}   \ar@{->}[d]^{\TLa_{IJ}(\Tx)}
\\
 \lm \rT_y U_J \otimes \bigl( \lm E_J \bigr)^*\quad
 \ar@{->}[r]^{\qquad\quad \mathfrak{C}_{\rd_y s_J}}  
 \ar@{->}[u]^{\Tilde{\mathfrak{C}}_{IJ}(\Tx)}
&
\quad\det(\rd_y s_J) .
}
\end{equation}
However, the composition $y  \mapsto \rho_{IJ}\circ \Tilde{\mathfrak{C}}_{IJ}(\Tphi_{IJ}^{-1}(\Tx))$ of the lefthand vertical maps is precisely the
map denoted by $y\mapsto \mathfrak{C}_{IJ}(x)$ in equation (8.1.15) of \cite{MW2}, while, as in the proof of 
Proposition~\ref{prop:det0} above, the right hand vertical maps combine to
$\La_{IJ}(x) = \TLa(\rho_{IJ}^{-1}(x)): \det(\rd_x s_I)\to \det(\rd_y s_J)$, where $\rho_{IJ}^{-1}$ is a local inverse to $\rho_{IJ}$.  Therefore the desired result follows from the commutativity of the following diagram:
$$
\xymatrix{
\lm  \rT_x U_I \otimes \bigl( \lm E_I \bigr)^* 
 \ar@{->}[r]^{ \qquad\quad\mathfrak{C}_{\rd_x s_I}}  
 &
\det(\rd_x s_I)  \ar@{->}[d]^{\La_{IJ}
(x)} \\
 \lm \rT_y U_J \otimes \bigl( \lm E_J \bigr)^*
 \ar@{->}[r]^{\qquad\quad\mathfrak{C}_{\rd_y s_J}}  
 \ar@{->}[u]^{\mathfrak{C}_{IJ}
 (x)} 
&
\det(\rd_y s_J) ,
}
$$
which is established in  \cite[Proposition~8.1.12]{MW2}.

For part (ii) the same arguments apply to define a bundle $\det(\Kk)$ and isomorphism $\Psi^{\s_\Kk}$.  The required product structure on a collar follows as in \cite{MW2}.
%
%
%
This completes the proof.
\end{proof}

We end this section by explaining how orientations of a Kuranishi atlas induce compatible orientations on local zero sets of transverse sections.

\begin{lemma}  \label{le:locorient}
Let $(\Kk,\si)$ be a d-dimensional oriented, tame Kuranishi atlas/cobordism,
and for some $I\in\Ii_\Kk$ let $f:W\to E_I$ be a smooth section over an open subset $W\subset U_I$ that is transverse to $0$. 
\begin{enumerate}
\item
The zero set $Z_f : = f^{-1}(0)\subset U_I$ inherits the structure of a smooth oriented $d$-dimensional submanifold.  
\item
The action of any $\ga\in \Ga_I$ on $U_I$ induces an orientation preserving diffeomorphism $Z_f\to Z_{\ga*f}$ to the zero set of $\ga\!*\!f: \ga(W) \to E_I,\;  x\mapsto \ga f(\ga^{-1}(x))$.
\item
Suppose in addition that $f(W)\subset\Hat\phi_{HI}(E_H)$, $\TW_{HI}:=W\cap\TU_{HI}\neq\emptyset$, and $\rho_{HI}|_{\TW_{HI}}$ is injective for some ${H\subset I}$. Then $\rho_{HI}$ induces an orientation preserving diffeomorphism $Z_f\to Z_{\rho_{HI}*f}$ to the zero set of $\rho_{HI}\!*\!f: \rho_{HI}(W) \to E_H, x\mapsto \Hat\phi_{HI}^{-1}\bigl( f(\rho_{HI}^{-1}(x))\bigr)$.
\item
If $\Kk$ is a cobordism, suppose in addition that $\bK_I$ is a chart that intersects the boundary $\p^\al\Kk$, with $W=\io_I^\al(A^\al_\eps \times W^\al)$ for some $W^\al\subset \p^\al W$, and $f(\io_I^\al(t,x))=f^\al(x)$ for some transverse section $f^\al: W^\al \to E_I$.
Then $(\p^\al\Kk,\p^\al\si)$ induces an oriented smooth structure on $Z_{f^\al}\subset W^\al$ by (i), $Z_f \subset U_I$ is a submanifold with boundary and $j^\al_I:=\io_I(\al,\cdot)$ is a diffeomorphism $Z_{f^\al} \to \p Z_f$ that preserves resp.\ reverses orientations in case $\al=1$ resp.\ $\al=0$. 
\end{enumerate}
\end{lemma}

\begin{proof} 
Except for (ii) these local claims follow directly from the corresponding parts of the proof of \cite[Proposition~8.1.13]{MW2}.
For (iii) note that the injectivity assumption allows us to write $\rho_{HI}\!*\!f = \phi_{HI}^*f$ for an embedding $\phi_{HI}:\rho_{HI}(\TW_{HI})\to W$.
Before we can prove (ii), recall that the orientation on $Z_f$ is induced from the orientation of the Kuranishi atlas/cobordism $\si_I:U_I\to\det(\rd s_I)$ via the isomorphisms for $z\in Z_f$
\begin{align*}
 \lm \rT_z Z_f & \;=\;  \lm \ker \rd_z f \;\cong\; \lm \ker \rd_z f \otimes \R  \;=\;  \det(\rd_z f)  , \\
\mathfrak{C}_{\rd_z f}\,:\;   & \det(\rd_z f) \;\longrightarrow\; \lm \rT_z U_I \otimes \bigl( \lm E_I\bigr)^* , \\
\mathfrak{C}_{\rd_z s_I} \,:\;  &
 \det(\rd_z s_I) \;\longrightarrow\; \lm \rT_z U_I \otimes \bigl( \lm E_I\bigr)^* .
\end{align*}
Now to prove that $\ga\in\Ga_I$ acts by an orientation preserving diffeomorphism, note that a smooth group action always acts by diffeomorphisms. Restriction to $Z_f$ of the action by $\ga\in\Ga_I$ thus yields a diffeomorphism to its image, which is easily seen to be the zero set of $\ga^*f$. 
To show that this diffeomorphism is compatible with the induced orientations at $z\in Z_f$ and $\ga z \in Z_{\ga\! *\! f}$, we begin by noting that the action of $\ga$ is $\La_{\rd_z\ga} : \lm \rT_z Z_f \to \lm \rT_{\ga z} Z_{\ga\! *\! f}$.
On the other hand, the orientations $\si_I(z)$ and $\si_I(\ga z)$ are by assumption intertwined by the isomorphism $\La_{\rd_z\ga} \otimes (\La_{[ \ga^{-1}]})^* : \det \rd_z s_I  \to  \det \rd_{\ga z} s_I$, 
and by Proposition~\ref{prop:orient}~(i) this implies that their pullbacks to  $\lm \rT_x U_I \otimes  (\La^{\max}E_I)^*$ for $x=z,\ga z$ are intertwined
by $\La_{\rd_z \ga} \otimes \bigl( \La_{\ga^{-1}} \bigr)^*$.
Thus it remains to prove that the following diagram commutes:
\[
 \xymatrix{
 \lm \rT_z U_I \otimes  (\La^{\max}E_I)^*  \ar@{->}[r]^{\mathfrak{C}_{\rd_z f}\qquad} \;\; \ar@{->}[d]_{\La_{\rd_z \ga}\otimes (\La_{\ga^{-1}})^* }
 & \;\; \lm \ker (\rd_z f) \otimes \R \;\cong\; \lm \rT_z Z_f\ar@{->}[d]^{\La_{\rd_z \ga}\otimes  \id_\R}
 \\
 \lm \rT_{\ga z} U_I\otimes  (\La^{\max}E_I)^* \quad  \ar@{->}[r]_{\mathfrak{C}_{\rd_{\ga z} 
 (\ga\! *\!f)} \qquad\qquad} 
 & \quad \lm \ker (\rd_{\ga z} (\ga\! *\!f)) \otimes \R \;\cong\; \lm \rT_{\ga z} Z_{\ga\! *\! f}
}
\]
By Lemma~\ref{lem:get}, $\mathfrak{C}_{\rd_z f}$ is given by
$
(v_1\wedge\dots v_n)\otimes(w_1\wedge\dots w_m)^* \;\mapsto\;(v_1\wedge\dots v_k)$,
where $v_1,\dots,v_n$ is any basis for $\rT_zU_I$ whose first $k$ elements span $\ker \rd_z f$, and 
$w_1,\dots,w_m$ is a basis for $E_I$, and similarly for  $\mathfrak{C}_{\rd_{\ga z} (\ga\! *\!f)}$. 
Therefore, if we denote $v_i': = \rd_z \ga (v_i)$ and $w_j' :=  \ga w_j$, we find that
$(\La_{\ga^{-1}})^* (w_1\wedge\dots w_m)^* = (w_1'\wedge\dots w_m')^*$
and thus the diagram commutes as required,
\begin{align*}
& \mathfrak{C}_{\rd_{\ga z} (\ga\! *\!f)} \bigl( \La_{\rd_z \ga}\otimes (\La_{\ga^{-1}})^* \bigl( (v_1\wedge\dots v_n)\otimes(w_1\wedge\dots w_m)^* \bigr)  \bigr) \\
&\qquad= 
\mathfrak{C}_{\rd_{\ga z} (\ga\! *\!f)}\bigl(
(v_1'\wedge\dots v_n')\otimes(w_1'\wedge\dots w_m')^* \bigr) \\
& \qquad= (v_1'\wedge\dots v_k') 
\;=\;
\La_{\rd_z \ga} ( v_1\wedge\dots v_n) 
\;=\;
\La_{\rd_z \ga} \bigl( \mathfrak{C}_{\rd_z f} \bigl( (v_1\wedge\dots v_n)\otimes(w_1\wedge\dots w_m)^* \bigr)  \bigr) .
\end{align*}
This completes the proof.
\end{proof}

\subsection{Perturbed zero sets}\label{ss:red}\hspace{1mm}\\ \vspace{-3mm}

With Theorem~\ref{thm:K} providing existence and uniqueness of tame shrinkings,
the second part of the proof of Theorem~A is the construction of the VMC/VFC from the 
zero sets of suitable perturbations $\s_\Kk  + \nu$ of the canonical section $\s_\Kk$ 
of a tame Kuranishi atlas or cobordism. 
In this section, we describe a suitable class of perturbations $\nu$, and prove that the
corresponding perturbed zero sets are compact weighted branched manifolds -- a notion from
\cite{Mcbr} that we review in Appendix~\ref{ss:br}.
The existence and uniqueness of such perturbations will be established in \S\ref{ss:A}, as part of the perturbative construction of VMC and VFC.
The main work is done by the setup in this section, which will put us into a situation in which the construction of perturbations and the resulting VMC/VFC can essentially be copied from \cite{MW2}.
Since the construction of perturbations requires tameness and the notion of weighted branched manifolds requires an orientation in \cite{Mcbr}, we will -- unless otherwise stated -- work with an oriented tame Kuranishi atlas or cobordism $\Kk$.
\MS

As in the case of trivial isotropy, one cannot in general find transverse perturbations $s_I+\nu_I \pitchfork 0$ that are also compatible with the coordinate changes $\Hat\Phi_{IJ}$.
Instead, we will construct perturbations over the following notion of a reduced atlas that still covers $X$ but generally does not form a Kuranishi atlas.

\begin{defn} {\rm $\!\!$ \cite[Definition~5.1.2]{MW1}}  
\label{def:vicin}  
A {\bf (cobordism) reduction} of a tame Kuranishi atlas/cobordism $\Kk$ is an open subset $\Vv=\bigsqcup_{I\in \Ii_\Kk} V_I \subset \Obj_{\bB_\Kk}$ i.e.\ a tuple of (possibly empty) open subsets $V_I\subset U_I$, satisfying the following conditions:
\begin{enumerate}
\item $V_I = \pi_I^{-1}(\uV_I)$  for each $I\in \Ii_\Kk$, i.e.\ $V_I$ is pulled back from the intermediate category and so is $\Ga_I$-invariant;
\item
$V_I\sqsubset U_I $ for all $I\in\Ii_\Kk$, and if $V_I\ne \emptyset$ then $V_I\cap s_I^{-1}(0)\ne \emptyset$;
\item
if $\pi_\Kk(\ov{V_I})\cap \pi_\Kk(\ov{V_J})\ne \emptyset$ then
$I\subset J$ or $J\subset I$;
\item
the zero set $\iota_\Kk(X)=|s_\Kk|^{-1}(0)$ is contained in 
$
\pi_\Kk(\Vv) \;=\; {\textstyle{\bigcup}_{I\in \Ii_\Kk}  }\;\pi_\Kk(V_I).
$
\end{enumerate}
If $\Kk$ is a cobordism, then we require in addition that $\Vv$ is collared in the following sense: 
\begin{enumerate}
\item[(v)]
For each $\al\in\{0,1\}$ and $I\in \Ii_{\p^\al\Kk}\subset\Ii_{\Kk}$ 
there exists $\eps>0$ and a subset $\partial^\al V_I\subset \partial^\al U_I$ such that $\partial^\al V_I\ne \emptyset$ iff 
$V_I \cap \psi_I^{-1}\bigl( \partial^\al F_I \bigr)\ne \emptyset$,
and 
$$
(\iota^\al_I)^{-1} \bigl( V_I \bigr) \cap \bigl(A^\al_\eps \times  \partial^\al U_I \bigr)
 \;=\; A^\al_\eps \times \partial^\al V_I .
$$
\end{enumerate}
We call 
$\partial^\al\Vv := \bigsqcup_{I\in\Ii_{\p^0\Kk}} \partial^\al V_I \subset \Obj_{\bB_{\p^\al\Kk}}$  
the {\bf boundary restriction} of $\Vv$ to 
$\p^\al\Kk$.
\end{defn}

\begin{remark}\label{rmk:vicin}\rm 
(i) The notion of (cobordism) reduction is equivalent to saying that $\Vv:=\bigsqcup_{I\in \Ii_\Kk} V_I \subset \Obj_{\bB_\Kk}$ is the lift $V_I:= \pi_I^{-1}(\uV_I)$ of a (cobordism) reduction $\uVv=\bigsqcup_{I\in \Ii_\Kk} \uV_I \subset \Obj_{\bB_\uKk}$ of the intermediate Kuranishi atlas/cobordism.
Thus existence and uniqueness of reductions is proven in \cite[Theorem~5.1.6]{MW1}.
\MS

\NI
(ii)
The restrictions $\partial^\al\Vv$ of a cobordism reduction $\Vv$ of a Kuranishi cobordism $\Kk$ are reductions of the restricted Kuranishi atlases $\p^\al\Kk$ for $\al=0,1$.
In particular condition (ii) holds because part~(v) of Definition~\ref{def:vicin} implies that if $\p^\al V_I\ne \emptyset$ then $\p^\al V_I \cap \psi_I^{-1}\bigl( \partial^\al F_I\bigr)\ne \emptyset$.
Note that  condition (v) also implies that $V_I\subset U_I$ is a collared subset in the sense of 
\eqref{eq:collset}.
$\hfill\er$
\end{remark}

Given a reduction $\Vv$, we define the {\bf reduced domain category} $\bB_\Kk|_\Vv$ and the {\bf reduced obstruction category} $\bE_\Kk|_\Vv$ to be the full subcategories of $\bB_\Kk$ and $\bE_\Kk$ with objects $\bigsqcup_{I\in \Ii_\Kk} V_I$ resp.\ $\bigsqcup_{I\in \Ii_\Kk} V_I\times E_I$, and denote by 
$\s_\Kk|_\Vv:\bB_\Kk|_\Vv\to \bE_\Kk|_\Vv$ the section given by restriction of $\s_\Kk$. 
Now one might hope to find transverse perturbation functors $\s_\Kk|_\Vv+\nu : \bB_\Kk|_\Vv \to \bE_\Kk|_\Vv$ by iteratively constructing $\nu_I:V_I\to E_I$ as in \cite{MW2}, where compatibility with the morphisms can be ensured by working along the partial order $\subsetneq$ on $\Ii_\Kk$, using the separation property (iii) of a reduction. However, we also have to ensure compatibility with the morphisms given by the action of nontrivial isotropy groups $\Ga_I$. Depending on their action, we might not even be able to even find a $\Ga_I$-equivariant perturbation $\nu_I$ in a single chart such that $s_I+\nu_I\pitchfork 0$.
In general, this can be resolved by using multivalued perturbations such as in the perturbative construction of the Euler class of an orbibundle -- explained e.g.\ in \cite{FO} as motivation for perturbations in Kuranishi structures.
We could also formulate our perturbation scheme in these terms, but due to the particularly simple setup -- notably additivity $\Ga_I=\prod_{i\in I} \Ga_i$ of the isotropy groups -- we can construct the ``multivalued perturbations'' as single-valued section functors $\nu: \bB_\Kk|_\Vv^{\less\Ga} \to \bE_\Kk|_\Vv^{\less\Ga}$ over a {\bf pruned domain category} $\bB_\Kk|_\Vv^{\less\Ga}$ that is obtained in Lemma~\ref{le:prune} 
from the reduced domain category $\bB_\Kk|_\Vv$ by forgetting sufficiently many morphisms to obtain trivial isotropy. It is to this category that the iterative perturbation scheme of \cite{MW2} will be applied in \S\ref{ss:A} to obtain a suitable class of transverse perturbations $\nu$.
Once a zero set is cut out transversely from $\bB_\Kk|_\Vv^{\less\Ga}$, we will then show in Theorem~\ref{thm:zero} that adding some of the isotropy morphisms back in -- at the expense of adding weights to corresponding branches of the solution set -- yields the structure of a weighted branched manifold on the Hausdorff quotient of the perturbed solution set 
$|(\s_\Kk|_\Vv^{\less\Ga} + \nu)^{-1}(0)|\subset |\bB_\Kk|_\Vv^{\less\Ga}|$. 
This perturbed solution set is not a subset of the virtual neighbourhood $|\Kk|$, but 
its Hausdorff quotient supports a fundamental class by Proposition~\ref{prop:fclass} and the inclusion 
$\io^\nu: (\s_\Kk|_\Vv^{\less\Ga} + \nu)^{-1}(0)\to\Obj_{\Bb_\Kk}$ induces a continuous map
$|\io^\nu|_\Hh: |(\s_\Kk|_\Vv^{\less\Ga} + \nu)^{-1}(0)|_\Hh \to |\Kk|$ 
that will represent the virtual fundamental cycle of $\Kk$.

We will describe the pruned categories in terms of the sets
$$
\TV_{IJ}: = V_J\cap \rho_{IJ}^{-1}(V_I) = V_J\cap \pi_{\Kk}^{-1}\bigl(\pi_\Kk(V_I)\bigr)\subset \TU_{IJ}.
$$ 
Note that $\TV_{IJ}$ is invariant under the action of $\Ga_J$, and 
is an open subset of the closed submanifold  $\TU_{IJ}=s_J^{-1}(E_I)$ of $V_J$, where the last equality holds by
the  tameness condition \eqref{eq:tame2}.
Further
if $F\subset I\subset J$, 
\begin{equation}\label{eq:HIJ}
V_J\cap \rho_{IJ}^{-1}(\TV_{FI}) = \TV_{IJ}\cap \TV_{FJ}= V_J\cap \pi_{\Kk}^{-1}\bigl(\pi_\Kk(V_I)\cap \pi_\Kk(V_F)\bigr)\subset \TU_{FJ}.
\end{equation}
In fact, the second equality  above holds for any pair of subsets $F,I\subset J$.  However, because $\Vv$ is a reduction, the intersection is empty unless
$F,I$ are nested, i.e.\ either $F\subset I$ or $I\subset F$.
Finally, the group $\Ga_{I\less F}$ acts freely on $\TU_{FI}$ (by Definition~\ref{def:change} for a coordinate change)  and hence also on $\TV_{FI}$.  If $I=F$ we define $\Ga_{I\less F}: = \Ga_\emptyset : = \{\id\}$.

\begin{lemma} \label{le:prune}
Let $\Vv$ be a (cobordism) reduction of a tame Kuranishi atlas/cobordism~$\Kk$.
Then there are well defined categories -- the {\bf pruned domain category} $\bB_\Kk|_\Vv^{\less\Ga}$ and the 
{\bf pruned obstruction category} $\bE_\Kk|_\Vv^{\less\Ga}$ -- obtained from $\bB_\Kk$ and $\bE_\Kk$ as follows:

\begin{itemlist}
\item
The object spaces are given by restriction to the reduction $\Vv=\bigsqcup_{I\in\Ii_\Kk} V_I \subset \Obj_{\bB_\Kk}$ as
$$
\Obj_{\bB_\Kk|_\Vv^{\less\Ga}} := {\textstyle \bigsqcup_{I\in\Ii_\Kk}} V_I \; \subset\; \Obj_{\bB_\Kk}, \qquad
\Obj_{\bE_\Kk|_\Vv^{\less\Ga}} := {\textstyle \bigsqcup_{I\in\Ii_\Kk}} V_I \times E_I \; \subset\; \Obj_{\bE_\Kk}.
$$
\item 
The morphism spaces are   open 
 subsets of $\Mor_{\bB_\Kk}$ resp.\  
$\Mor_{\bE_\Kk}$ with components
$$
\Mor_{\bB_\Kk|_\Vv^{\less\Ga}} := {\textstyle \bigsqcup_{I,J\in\Ii_\Kk}} \Mor_{\bB_\Kk|_\Vv^{\less\Ga}}(V_I,V_J) , \qquad
\Mor_{\bE_\Kk|_\Vv^{\less\Ga}} := {\textstyle \bigsqcup_{I,J\in\Ii_\Kk}} \Mor_{\bE_\Kk|_\Vv^{\less\Ga}}(V_I,V_J), 
$$
given by $\Mor_{\ldots}(V_I,V_J)=\emptyset$ unless $I\subset J$, in which case the morphisms are given in terms of 
the open
 subsets $\TV_{IJ}: =V_J\cap \rho_{IJ}^{-1}(V_I)\subset \TU_{IJ}$
 as
\begin{align*}
\Mor_{\bB_\Kk|_\Vv^{\less\Ga}}(V_I,V_J)
&:= \;  \TV_{IJ}\times \{\id\}
\;\subset \;
 \TU_{IJ} \times \Ga_I  = \Mor_{\bB_\Kk}(U_I,U_J), \\
\Mor_{\bE_\Kk|_\Vv^{\less\Ga}}(V_I,V_J)
&:=\; \TV_{IJ}\times E_I\times  \{\id\}
\;\subset \;
 \TU_{IJ} \times E_I \times \Ga_I  = \Mor_{\bE_\Kk}(U_I,U_J) . 
\end{align*}
\item
All structure maps (source, target, identity, and composition) are given by restriction of the structure maps of $\bB_\Kk$ resp.\ $\bE_\Kk$ in Definition~\ref{def:catKu}.
\end{itemlist}
These pruned categories are nonsingular in the sense that there is at most one morphism between any two objects. Moreover, the projection and section functors $\pr_\Kk:\bE_\Kk \to \bB_\Kk$, $\s_\Kk:\bB_\Kk \to \bE_\Kk$ restrict to well defined functors
$\pr_\Kk|_\Vv^{\less\Ga}:\bB_\Kk|_\Vv^{\less\Ga} \to \bE_\Kk|_\Vv^{\less\Ga}$
and
$\s_\Kk|_\Vv^{\less\Ga}:\bB_\Kk|_\Vv^{\less\Ga} \to \bE_\Kk|_\Vv^{\less\Ga}$
with $\pr_\Kk|_\Vv^{\less\Ga}\circ \s_\Kk|_\Vv^{\less\Ga} =\id_{\bB_\Kk|_\Vv^{\less\Ga}}$.
\end{lemma}
\begin{proof}   
Recall that $(I,J,y,\id)\in\Mor_{\bB_\Kk|_\Vv^{\less\Ga}}$ has source $(I,\rho_{IJ}(y))$ 
and target $(J,y)$ (where as in Lemma~\ref{le:Kcat}  we suppress mention of the inclusion $\Tphi_{IJ}$). 
Now morphisms are closed under composition because the strong cocycle condition guarantees $\rho_{IJ}\circ \rho_{JK} = \rho_{IK}$ with identical domains whenever $I\subset J\subset K$.  
Moreover, the category is nonsingular because source and target determine the morphism uniquely. Similar arguments apply to $\bE_\Kk|_\Vv^{\less\Ga}$. 
Finally, the projection and section functors of $\Kk$ act trivially on the isotropy groups $\Ga_I$, thus restrict to well defined functors when we drop these.
\end{proof}

The following  combines Definitions~7.2.1, 7.2.5, 7.2.6, and~7.2.9 from  \cite{MW2}.

\begin{defn} 
\label{def:sect} 
A  {\bf (cobordism) perturbation} of $\Kk$ is a smooth functor $\nu:\bB_\Kk|_\Vv^{\less\Ga}\to\bE_\Kk|_\Vv^{\less\Ga}$ between the pruned domain and obstruction categories of some (cobordism) reduction $\Vv$ of $\Kk$, such that $\pr_\Kk|_\Vv^{\less\Ga}\circ\nu=\id_{\bB_\Kk|_\Vv^{\less\Ga}}$. 

That is, $\nu=(\nu_I)_{I\in\Ii_\Kk}$ is given by a family of smooth maps $\nu_I: V_I\to E_I$ that are compatible with coordinate changes in the sense that for all $I \subsetneq J$ we have
\begin{equation}\label{eq:compatc}
\nu_J\big|_{\TV_{IJ}}\  =\ 
 \Hat\phi_{IJ}\circ\nu_I\circ \rho_{IJ}\big|_{\TV_{IJ}}
\quad \text{on}\quad
\TV_{IJ} = V_J\cap \rho_{IJ}^{-1}(V_I).
\end{equation}
If $\Kk$ is a Kuranishi cobordism we require in addition that $\nu$ has product form in a collar neighbourhood of the boundary. That is, for $\al=0,1$ and $I\in \Ii_{\Kk^\al}\subset\Ii_{\Kk}$ there is $\eps>0$ and a map $\nu_I^\al: \p^\al V_I\to E_I$ such that 
$$
\nu_I \bigl( \io_I^\al ( t,x ) \bigr) 
= \nu_I^\al (x)  \qquad
\forall\, x\in \p^\al V_I , \   t\in A^\al_\eps .
$$ 
We say that a (cobordism) perturbation $\nu$ is
\begin{itemlist}\item
 {\bf admissible}  if we have 
$\rd_y \nu_J(\rT_y V_J) \subset\im\Hat\phi_{IJ}$ for all $I\subsetneq J$ and $y\in \TV_{IJ}$;
\item {\bf transverse}
if  $s_I|_{V_I} + \nu_I: V_I\to E_I$ is transverse to $0$ 
 for each $I\in \Ii_\Kk$;
 \item
 {\bf precompact} if there is a precompact open subset $\Cc\sqsubset \Vv$, which itself is a (cobordism) reduction, such that 
\begin{equation}\label{eq:piKC}
\pi_\Kk \bigl(\;{\textstyle  \bigcup_{I\in \Ii_\Kk} } (s_I|_{V_I} + \nu_I)^{-1}(0) \bigr) \;\subset\; \pi_\Kk(\Cc).
\end{equation}
\end{itemlist}
\end{defn}

\begin{rmk}\rm
Although $\pi_\Kk:\Obj_{\Bb_\Kk} \to |\Kk|$ is not induced by a functor on $\Bb_\Kk|_\Vv^{\less\Ga}$, we will work with $\pi_\Kk:\bigsqcup_{I\in\Ii_\Kk} U_I \to |\Kk|$ as continuous map -- in particular for the notion of precompactness. 
As in the case of trivial isotropy, we do not have a nicely controlled cover of sets $U_J\cap \pi_\Kk^{-1}(\pi_\Kk(\Cc))$ for $\Cc\subset \bigsqcup U_I$. However, because
$\Cc =\bigsqcup C_I\subset\Vv= \bigsqcup V_I\subset  \bigsqcup U_I$ 
are lifts of reductions of $|\uKk|$ as in Remark~\ref{rmk:vicin},
the morphisms between $V_J$ and $\Cc$ are better understood, yielding
\begin{equation}\label{eq:VCC}
V_J\cap \pi_\Kk^{-1}(\pi_\Kk(\Cc))
\;=\; V_J \cap \bigl( \;{\textstyle \bigcup_{H\supset J} } \rho_{JH}(C_H) \cup {\textstyle \bigcup_{H\subsetneq J} }\rho_{HJ}^{-1}(C_H) \;\bigr) .
\end{equation}
Indeed, by the reduction property, 
$\pi_\Kk(V_J)$ only intersects $\pi_\Kk(C_H)$ for $H\supset J$ or $H\subset J$. 
The morphisms between $V_J$ and $C_H$ are then given by  $\rho_{JH}$ and $\Ga_J$ resp.\ $\rho_{HJ}$ and $\Ga_H$, and the isotropy groups are absorbed by the equivariance $\Ga_J \rho_{JH}(C_H) = \rho_{JH}(\Ga_H C_H )$ and fact that
$\Ga_H C_H = C_H = \pi_H^{-1}(\und C_H)$.
As a result, we can write \eqref{eq:piKC} in terms of the covering maps $(\rho_{IJ})_{I,J\in\Ii_\Kk}$, without explicit reference to the isotropy groups $\Ga_I$, as
\begin{equation}\label{eq:zeroVCC}
(s_J|_{V_J} + \nu_J)^{-1}(0)
\;\subset\;  {\textstyle \bigcup_{H\supset J} } \rho_{JH}(C_H) \cup {\textstyle \bigcup_{H\subsetneq J} }\rho_{HJ}^{-1}(C_H) 
\qquad\forall J\in \Ii_\Kk. \qquad\quad\er 
\end{equation}
\end{rmk}

\begin{defn} \label{def:sect2}
Given a (cobordism) perturbation $\nu$, we define the {\bf perturbed zero set} $|\bZ^\nu|$ to be the realization of the full subcategory $\bZ^\nu$ of $\bB_\Kk|_\Vv^{\less\Ga}$ with object space 
$$
(\s_\Kk|_\Vv^{\less\Ga} + \nu)^{-1}(0)  
 := {\textstyle \bigsqcup_{I\in \Ii_\Kk}}(s_I|_{V_I}+\nu_I)^{-1}(0) \;\subset\;\Obj_{\bB_\Kk|_\Vv^{\less\Ga}}  
$$
given by the local zero sets $Z_I: = (s_I|_{V_I} + \nu_I)^{-1}(0)$. 
That is, we equip
$$
|\bZ^\nu| : = \bigl|( \s_\Kk|_\Vv^{\less\Ga}  + \nu)^{-1}(0)  
\bigr| \,=\; \quotient{ {\textstyle\bigsqcup_{I\in\Ii_\Kk} Z_I }}{\!\sim} 
$$
with the quotient topology generated by the morphisms of $\bB_\Kk|_\Vv^{\less\Ga}$.
Moreover, we denote by $\io^\nu: \bZ^\nu \to \bB_\Kk$ the functor induced by the inclusion
$(\s_\Kk|_\Vv^{\less\Ga} + \nu)^{-1}(0) \to  \Obj_{\bB_\Kk}$ and corresponding inclusion of morphism spaces (to a generally not full subcategory), with resulting continuous map
\begin{equation}\label{eq:Zmap} 
|\io^\nu| \,:\;  |\bZ^\nu| \;\longrightarrow\; |\Kk| .
\end{equation}
\end{defn}

\begin{rmk}\label{rmk:nurest}\rm  
If $\nu:\bB_{\Kk}|_\Vv^{\less\Ga}\to\bE_{\Kk}|_{\Vv}^{\less\Ga}$ is a cobordism perturbation of a tame Kuranishi cobordism $\Kk$, then each {\bf restriction} $\nu|_{\partial^\al\Vv} := \bigl(  \nu_I^\al \bigr)_{I\in\Ii_{\Kk^\al}}$ for $\al=0,1$ forms a perturbation of the Kuranishi atlas $\p^\al\Kk$ with respect to the boundary restriction of the reduction, $\p^\al\Vv$. 

If in addition $\nu$ is admissible/transverse/precompact, then so are the restrictions $\nu|_{\partial^\al\Vv}$.
Moreover, in the case of transversality each perturbed section $s_I|_{V_I}+\nu_I : V_I \to E_I$ for $I\in\Ii_{\p^0\Kk}\cup\Ii_{\p^1\Kk}\subset \Ii_{\Kk}$ is transverse to $0$ as a map on a domain with boundary; i.e.\ the kernel of its differential is transverse to the boundary 
$\partial V_I = \bigsqcup_{\al=0,1}\iota^\al_I ( \{\al\} \times \partial^\al V_I)$.
$\hfill\er$
\end{rmk}

\begin{thm}\label{thm:zero}  
Let $(\Kk,\si)$ be an oriented tame Kuranishi atlas/cobordism of dimension $d$ and let $\nu$ be an admissible, transverse, precompact (cobordism) perturbation of $\Kk$ with respect to nested (cobordism) reductions $\Cc\sqsubset \Vv\sqsubset\Obj_{\Bb_\Kk}$.
Then $\bZ^\nu$ can be completed to a compact, $d$-dimensional wnb (cobordism) groupoid $\Hat\bZ^\nu$ 
in the sense of Definition~\ref{def:brorb} with the same realization $|\Hat\bZ^\nu|=|\bZ^\nu|$.
In addition
$$
\La^\nu(p) := |\Ga_I|^{-1} \# \bigl\{ z\in Z_I \,\big|\, \pi_H(|z|) =p \bigr\} \qquad \text{for}\; p\in |Z_I|_\Hh
$$
defines a weighting function $\La^\nu:|\bZ^\nu|_\Hh\to \Q^+$ on the Hausdorff quotient of the perturbed zero set $|\bZ^\nu|_\Hh$. Together, these give $(|\Hat\bZ^\nu|_\Hh, \La^\nu)$ the structure of a compact, $d$-dimensional weighted branched manifold/cobordism in the sense of Definition~\ref{def:brman}.
It defines a cycle in $|\Cc|$ in the sense that the map 
$|\io^{\nu}|_\Hh:  |\Hat\bZ^\nu|_\Hh\to |\Kk|$ induced by \eqref{eq:Zmap} has image in $|\Cc|$.

Moreover, if $\Kk$ is a Kuranishi cobordism and we denote the boundary restrictions of $\nu$ by $\nu^\al:=\nu|_{\p^\al\Vv}$, then $(\Hat\bZ^\nu, \La^\nu)$ has oriented boundaries $(\Hat\bZ^{\nu^0}, \La^{\nu^0})$ and  $(\Hat\bZ^{\nu^1}, \La^{\nu^1})$ and the cycle $|\io^\nu|_\Hh: |\Hat\bZ^{\nu}|_\Hh\to |\Cc|$ restricts on the boundaries to 
$|\io^{\nu^\al}|_\Hh: |\Hat\bZ^{\nu^\al}|_\Hh\to |\p^\al\Cc|$.
\end{thm}

We begin the proof of Theorem~\ref{thm:zero} by explaining the structure of the groupoid completion $\Hat\bZ^\nu$. 
Note that the compatibility condition \eqref{eq:compatc} implies partial equivariance of the perturbation: 
$\nu_J(\al y) = \nu_J(y) $ for $y\in \TV_{IJ}, \al\in \Ga_{J\less I}$. This fact is reflected in the structure of the morphisms in the groupoid $\Hat\bZ^\nu$ which contain this action of $\Ga_{J\less I}$ on $\TV_{IJ}\cap Z_J$ as part of the morphism space $\Mor_{\Hat\bZ^\nu}(Z_J,Z_J)$.

 \begin{lemma}\label{le:zero1}  Let  $\nu$ be any (cobordism) perturbation of a tame $d$-dimensional Kuranishi 
 atlas/cobordism  $\Kk$. 
 \begin{enumerate}
 \item 
There is a unique nonsingular groupoid $\Hat\bZ^\nu$ with the same objects and realization as  $\bZ^\nu$.  Its morphism space 
 for $I\subset J$ is given by
 $$
\Mor_{\Hat\bZ^\nu}(Z_I,Z_J) :=
{\textstyle \bigcup_{\emptyset\ne F\subset I} } \; \bigl(Z_J\cap \TV_{IJ}\cap \TV_{FJ}\bigr)\times \Ga_{I\less F}
\;\subset \;
 \TU_{IJ} \times \Ga_I  = \Mor_{\bB_\Kk}(U_I,U_J).
$$
\item If $\nu$ is admissible and tranverse, then the subsets $Z_J\cap \TV_{IJ}\subset Z_J$ are open for all $I\subset J$
and the groupoid  $\Hat\bZ^\nu$  is \'etale and has dimension $d$.  Further, $\Hat\bZ^\nu$  is oriented if in addition $\Kk$ is oriented. 
\item 
If $\Kk$ is an oriented 
tame Kuranishi cobordism
and $\nu$ is admissible and tranverse, then $\Hat\bZ^\nu$ satisfies 
all conditions in 
Appendix
\S\ref{ss:br}
to be an \'etale, oriented, cobordism groupoid, except possibly that of compactness.
\end{enumerate}
 \end{lemma}
\begin{proof} 
First note that there is at most one nonsingular groupoid with the same objects and realization as $\bZ^\nu$ since any such groupoid has a unique morphism  $(I,x)\mapsto (J,y)$ whenever $(I,x)\sim (J,y)$, where $\sim$ is the equivalence relation on $\Obj_{\bZ^\nu}$ generated by $\Mor_{\bZ^\nu}$.   
To prove existence of such a groupoid, we show below that when $I\subset J$
\begin{itemize}\item[(a)] 
 each element in $\Mor_{\Hat\bZ^\nu}(Z_I,Z_J)$  
is uniquely determined by its source and target;
\item[(b)] if there is 
a morphism $(I,J,y,\al)\in \Mor_{\Hat\bZ^\nu}(Z_I,Z_J) $ with source $(I,x)$ and target $(J,y)$
then $(I,x)\sim (J,y)$; and
\item[(c)]  the set of morphisms $\bigcup_{I\subset J} \Mor_{\Hat\bZ^\nu}(Z_I,Z_J) $ together with their inverses 
(which are uniquely defined by (a)) is closed under composition.
\end{itemize}
Parts (a) and (c)  show that there is a nonsingular groupoid $\Hat\bZ^\nu$ with the given morphisms. 
Moreover, since the equivalence relation $\sim$ is generated by the morphisms 
$(I,J,y,\id)\in \Mor_{\bZ^\nu}(Z_I,Z_J)\subset \Mor_{\Hat\bZ^\nu}(Z_I,Z_J)$, 
(c) shows that if $(I,x)\sim (J,y)$ where  $I\subset J$ then $\Mor_{\Hat\bZ^\nu}((I,x), (J,y))\ne \emptyset$.
Together with (b), this implies that  $\Hat\bZ^\nu$  has realization $|\bZ^\nu|$.

To prove (a) we must check that given 
$x\in U_I, y\in Z_J\cap \TV_{IJ}$ where $I\subset J$ there is at most  one element $\al \in \Ga_I$ such that
\begin{itemize}\item  $x = \al^{-1}\rho_{IJ}(y)$;
\item there is $F\subset I$ such that $\al\in \Ga_{I\less F}$ and $y\in \TV_{FJ}\cap \TV_{IJ} $.
\end{itemize}
But if $\al_1,\al_2$ are two such elements, corresponding to $F_1, F_2$, then $\al_1^{-1}\al_2$ fixes the point
$\rho_{IJ}(y)$. On the other hand, because  the set of $F$ such that $y\in \TV_{FJ}$ is nested, we can suppose that
 $F_1\subset F_2$.  Then $\rho_{IJ}(y)\in \TV_{F_1I}$ and $\al_1^{-1}\al_2\in \Ga_{I\less F_1}$. Since $\Ga_{I\less F_1}$  acts freely on $\TV_{F_1 I}$, this implies that $\al_1 = \al_2$ as required.

To prove (b), observe  that if $I\subset J$ and $\Mor_{\Hat\bZ^\nu}((I,x),(J,y))\ne \emptyset$ then 
 there is $F\subset I$ and $\al\in \Ga_{I\less F}$ such that $x = \al^{-1}\rho_{IJ}(y)$, which implies that
$$
\rho_{FI}(x) = \rho_{FI}( \al^{-1}\rho_{IJ}(y)) = \rho_{FI}(\rho_{IJ}(y)) = \rho_{FJ}(y).
$$
Hence, the composite
$ (F,I,x,\id )\circ (I,J,y,\al)$ is well defined and equal to $(F,J,y,\id)$.   Therefore $(F,\rho_{FI}(x))\sim (I,x)$ and 
$(F,\rho_{FI}(x))= (F, \rho_{FJ}(y) )\sim (J,y)$, which gives  $(I,x)\sim (J,y)$ since $\sim$ is an equivalence relation.

Finally, to prove (c), it is convenient to consider two special kinds of morphisms: morphisms denoted $\mu^A$ with $I=J$
and  morphisms denoted $\mu^B$ with  $I\subsetneq J$ and $\al = \id$ that therefore 
belong to $\Mor_{\bZ^\nu}$.
We first observe that every morphism $(I,J,y,\al)$ in $\Mor_{\Hat\bZ^\nu}(Z_I,Z_J)$ can be written in two ways 
as a composite 
of morphisms  of  types (A), (B).  More precisely,
the identity $\mu^A_1\circ \mu^B_1=\mu^B_2\circ \mu^A_2$ holds,
 where $\mu^A_i,\mu^B_j$ are the morphisms in the following 
 diagram:
\begin{equation}\label{eq:ABIJ}
\xymatrix{
\bigl(I,\al^{-1}\rho_{IJ}(y))\bigr) \ar@{->}[d]^{\mu_1^A} \ar@{->}[r]^{\qquad\;\mu_2^B}  & (J,\al^{-1}y)\ar@{->}[d]^{\mu_2^A}\\
\bigl(I, \rho_{IJ} (y)\bigr)\ar@{->}[r]^{\;\;\;\;\mu_1^B} & \bigl(J, y\bigr).
}
\end{equation}
Therefore these morphisms $\mu^A, \mu^B$ and their inverses generate $\Mor_{\Hat\bZ^\nu}$.  The commutativity of the above diagram
 also shows that we can interchange their order: i.e.\ every morphism of the form
$\mu^A_1\circ \mu^B_1$ can also be written as  $\mu^B_2\circ \mu^A_2$, which we abbreviate below
as the identity $\mu^A_1\circ \mu^B_1=\mu^B_2\circ \mu^A_2$.

Next let us consider the other composites.  
Morphisms of type (A)  with fixed 
 $F\subset  I$ are closed under composition since  they are given by the action of  $\Ga_{I\less F}$. 
 Moreover, two morphisms of this type corresponding to different subgroups $F_1,F_2$ can be composed only if 
 the  sets $\pi_\Kk(V_I), \pi_\Kk(V_{F_1}),\pi_\Kk(V_{F_2})$ intersect.  Hence the sets $F_1,F_2$ are nested, either 
 $F_1\subset F_2$ or $F_2\subset F_1$, and in either case the composite is another morphism of this type. 
 The situation for morphisms of type (B) is more complicated 
 (which is precisely why we needed to add the morphisms of type (A) to obtain a groupoid).  We have:
 \begin{itemlist}
 \item $\mu_1^B\circ \mu_2^B = \mu_3^B$: i.e.\ if $I\subset J\subset K$ and $y = \rho_{JK}(z)$ then the following identity holds (this statement includes 
 the claim that the left hand composite is well defined):
 $$
(I,J,y,\id)\circ (J,K,z,\id) = (I,K,z,\id)\ \in\  \Mor_{\Hat\bZ^\nu}\bigl((I,\rho_{IK}(z)),\ (K,z)\bigr);
 $$
 \item  $(\mu_1^B)^{-1}\circ \mu_2^B =  
\mu^A \circ \mu_3^B$ or $=\mu^A \circ (\mu_3^B)^{-1}$: 
 \begin{itemize}\item[-] if $I\subset J\subset K$ and $\rho_{IJ}(y') = \rho_{IK}(y) = \rho_{IJ}\circ\rho_{JK}(y)$ ,
 then  $\rho_{JK}(y)$ and $y'$ lie in the same $\Ga_{J\less I}$-orbit
so that   $y' =\al^{-1} \rho_{JK}(y)$ for some $\al\in \Ga_{J\less I}$ and
 we have
\begin{align*}
(I,J,y',\id)^{-1}\circ (I,K,y,\id) & =(J,J,\rho_{JK}(y),\al)\circ (J,K,y,\id) \\
&\qquad \qquad  \in\  \Mor_{\Hat\bZ^\nu}\bigl((J,\al^{-1}\rho_{JK}(y)),\ (K,y)\bigr);
\end{align*}
\item[-] if $I\subset K\subset J$ and there are $y'\in \TV_{IJ}\cap Z_J,\;  y\in \TV_{IK}\cap Z_K$ with
$$
\rho_{IJ}(y') = \rho_{IK}(\rho_{KJ}(y')) = \rho_{IK}(y)\in Z_I, 
$$
 then
there is $\be \in \Ga_{K\less I}$ such that
$y = \be\, \rho_{KJ}( y') =  \rho_{KJ}(\be y') \in Z_K$, and we have
\begin{align*}
 (I,J,y',\id)^{-1}\circ (I,K,y,\id) \; & =\; (J,J,\be\,y',\be)\circ (K,J,\be\, y',\id)^{-1}\\
 &\qquad \qquad  \in\  \Mor_{\Hat\bZ^\nu}\bigl( (J, y'), (K,y) \bigr).
\end{align*}
\item 
One can check similarly that if $\mu_1^B = (I,J,y,\id)$ and  $\mu_2^B = (K,J,y,\id)$ 
then $\mu_1^B\circ (\mu_2^B)^{-1} = \mu^A \circ \mu_3^B$ if $I\subset K$ and 
$=  \mu^A \circ (\mu_3^B)^{-1}$ if $K\subset I$.
\end{itemize}
\end{itemlist}
Combining these identities  with
 $\mu^A_1\circ \mu^B_1=\mu^B_2\circ \mu^A_2$ and its inverse, we see that if $I\subset J$ every composite
 morphism $Z_I\to Z_J\to Z_K$ can be written
in the form $\mu^B\circ \mu^A$ if $I\subset K$, and in the form $\mu_1^A\circ (\mu_1^B)^{-1} = 
(\mu_2^B)^{-1}\circ \mu_2^A$ if $K\subset I$.
This completes the proof of (c) and hence of part (i) of the  lemma.

The claims in (ii) are proved by applying Lemma~\ref{le:locorient} with $f:W\to E_I$ given by $s_I + \nu_I:V_I\to E_I$.
 Since $s_I + \nu_I\pitchfork 0$, Lemma~\ref{le:locorient}~(i)   shows that 
 $Z_I$ is a manifold, while the admissibility of $\nu$  implies that the hypothesis of 
 Lemma~\ref{le:locorient}~(iii)  holds on $\TV_{HI}$ so that 
 the subset $Z_I\cap \TV_{HI}$ of $Z_I$ is open and
 $\rho_{HI}$ induces  a local  diffeomorphism
 from  $Z_I\cap \TV_{HI}$  to $Z_H\cap \rho_{HI}(\TV_{HI})$.  Further, by  the compatibility condition \eqref{eq:compatc}  we can identify with the zero set of $\rho_{HI}*(s_I + \nu_I) = s_H+ \rho_{HI}*(s_I)$.  Since the maps $\rho_{IJ}$ together with their inverses generate the structure maps in $\Hat\bZ^\nu$, this shows that this groupoid is \'etale. Moreover, if $\Kk$ is oriented, then  Lemma~\ref{le:locorient}~(ii),(iii) also implies that the structure maps in $\Hat\bZ^\nu$ are orientation preserving.
 
 Finally, (iii) holds by Lemma~\ref{le:locorient}~(iv). 
\end{proof}

In order to show that 
$\Hat \bZ^\nu$ represents a weighted branched manifold, we must understand its maximal Hausdorff quotient $|\Hat\bZ^\nu|_\Hh$ as defined in Lemma~\ref{le:Hquotient}.
 The morphisms in a nonsingular groupoid $\bG$ correspond bijectively to the equivalence relation $\sim_{\bG}$ on $\Obj_{\bG}$ where $x\sim_{\bG} y$ if and only if $\Mor_{\bG}(x,y)\ne \emptyset$.  A necessary condition for the quotient $|\bG|: = \qu{\Obj_{\bG}}{\sim_{\bG}}$ to be Hausdorff is that this equivalence relation is given by a closed subset of $\Obj_{\bG}\times \Obj_{\bG}$, in other words, we need the map
 $s\times t: \Mor_{\bG}\to  \Obj_{\bG}\times \Obj_{\bG}$ that takes a morphism to its source and target to have closed image.  The following lemma  shows that in the special case of the groupoid $\Hat \bZ^\nu$ this necessary condition is also sufficient.

 \begin{lemma}\label{le:zero2} Let  $\nu$ be an admissible, transverse, (cobordism) perturbation of a tame Kuranishi    
 atlas/cobordism $\Kk$. 
\begin{enumerate} 
\item Let $\Hat \bZ^\nu_\Hh$ be the  groupoid obtained from $\Hat\bZ^\nu$ by closing the relation 
 $\sim$ on $\Obj_{\Hat\bZ^\nu}$.
 Then $\Hat \bZ^\nu_\Hh$ is nonsingular, and  $|\Hat \bZ^\nu_\Hh|$ is Hausdorff.  Further, we can identify $|\Hat \bZ^\nu_\Hh|$ with 
 the maximal Hausdorff quotient $|\Hat\bZ^\nu|_\Hh$ in such a way that the canonical quotient map  $|\Hat \bZ^\nu|\to  |\Hat \bZ^\nu|_\Hh= |\Hat \bZ^\nu_\Hh|$ is induced by the functor $\io_H: \Hat\bZ^\nu\to \Hat \bZ^\nu_\Hh$.   
 \item 
For each $I\in \Ii_\Kk$, the projection $\pi_{\Hat \bZ^\nu_\Hh}: \Obj_{\Hat \bZ^\nu_\Hh}\to |\Hat \bZ^\nu_\Hh|$ takes $Z_{I}$  onto a subset of $ |\Hat \bZ^\nu_\Hh|$ that is open with respect to the quotient topology.
 Moreover this topology on $ |\Hat \bZ^\nu_\Hh|$ is metrizable.
 \item
If  $x \in Z_I$ and $p = \pi_{\Hat \bZ^\nu_\Hh}(x)\in |\Hat \bZ^\nu_\Hh|$ 
then $\bigl\{ x'\in Z_I \ | \ \pi_{\Hat \bZ^\nu_\Hh}(I,x') = p  \bigr\}$ is the $(\Ga_{I\less F_x})$-orbit of $x$ so that
$$
\# \bigl\{ x\in Z_I \ | \ \pi_{\Hat \bZ^\nu_\Hh}(I,x) = p  \bigr\} =  |\Ga_{I\less F_x}|,
$$
where $F_x = \min \bigl\{F:  Z_I\cap cl(\TV_{FI})\cap \pi_{\Hat \bZ^\nu_\Hh}^{-1}(p) \ne \emptyset\bigr\}=\min\bigl\{F: p\in \ov{\pi_{\Hat \bZ^\nu_\Hh}(Z_{F})}\bigr\}$.
 \end{enumerate}
  \end{lemma} 
\begin{proof}  We use the notation in  Lemma~\ref{le:zero1}. The components of $\Mor_{\Hat \bZ^\nu }(Z_{I}, Z_{J})$ consisting of morphisms of type (B)    are taken by 
$s\times t:  \Mor_{\Hat \bZ^\nu}(Z_{I}, Z_{J}) \to  Z_{I}\times Z_{J}\subset \Obj_{\Hat \bZ^\nu}\times \Obj_{\Hat \bZ^\nu}$ to the set of pairs 
$$
\bigl \{(\rho_{IJ}(y), y):y\in Z_{J}\cap \TV_{IJ}\cap \pi_\Kk^{-1}(V_I)\bigr\}
\subset Z_{I}\times Z_{J},
$$
where we simplify notation by writing $y$ instead of $(J,y)$, and similarly for the source.
If $(\rho_{IJ}(y_n), y_n) \to (x_\infty, y_\infty)\in Z_{I}\times Z_{J}$ 
is a convergent sequence of such points with limit 
$ (x_\infty, y_\infty)\in  Z_{I}\times Z_{J}$, then 
 $y_\infty\in  Z_J\cap \TU_{IJ}$ 
 since $y_n\in Z_J\cap \TV_{IJ}\subset 
 \TU_{IJ}$  and $ \TU_{IJ}$ is closed in $U_J$, 
which implies that $\rho_{IJ}(y_\infty)$ is defined.  We then must have $x_\infty =\rho_{IJ}(y_\infty)$ by the 
continuity of $\rho_{IJ}$.
Thus   $y_\infty\in \rho_{IJ}^{-1}(Z_I)\cap Z_J \subset \rho_{IJ}^{-1}(V_I)\cap V_J =\TV_{IJ} $.
Hence $y_\infty \in Z_J\cap \TV_{IJ}$, so that  $(I,J,y_\infty, \id)\in \Mor_{\Hat \bZ^\nu}(Z_{I}, Z_{J})$.
Therefore the graph of this set of morphisms is  closed in $Z_I\times Z_J$.

However the set of morphisms of type (A) from $Z_{I}\to Z_{I}$ is not closed in general; instead it has closure\footnote
 {
 While we usually denote the closure of a set $A$ by $\ov A$, for sets such as $\TV_{IJ}$ that involve a tilde we will write $cl(\TV_{IJ})$.}
$$
\ov{\Mor_{\Hat \bZ^\nu}(Z_I,Z_I)}: = \bigcup_{F\subsetneq I} \bigl\{(I,I,y,\al): , y\in \cl(\TV_{FI})\cap  Z_{I}, \al\in \Ga_{I\less F} \bigr\}.
$$
Notice that,  as in the proof of 
 Lemma~\ref{le:zero1}~(i), this set $\ov{\Mor_{\Hat \bZ^\nu}(Z_I,Z_I)}$
 is invariant under compositions (and inverses) because  
the intersection properties of the sets in a reduction apply to their closures: 
$\pi_\Kk(\ov{V_{F_1}})\cap \pi_\Kk(\ov{V_{F_2}})\ne \emptyset $ $ \Longrightarrow {F_1}\subset {F_2} \mbox{ or } {F_2}\subset
{F_1}$.
Next, observe that because 
 $\rho_{IJ}: \TU_{IJ}\to U_{IJ}$ is a local diffeomorphism, the map
 $\rho_{IJ}$ induces a local diffeomorphism  from
 $\TV_{IJ}\cap cl(\TV_{FJ}) \cap Z_{J}$ into $ U_{IJ}\cap cl(\TV_{FI})\cap Z_{I}$.  
Similarly, because $\cl(\TV_{FJ})\subset \TU_{FJ}$ whenever $F\subset J$, the group $\Ga_{I\less F}$ acts freely on
$ \cl(\TV_{FJ})$, and, if $F\subset I$, commutes with action of $\rho_{IJ}$ as in diagram \eqref{eq:ABIJ}.
 Therefore the closure of $\Mor_{\Hat \bZ^\nu}(Z_{I}, Z_{J})$ when $I\subset J$  is given as follows:
 \begin{align}\label{eq:morZ3}
&\quad \Mor_{\Hat \bZ^\nu_\Hh}(Z_{I}, Z_{J}) = \bigcup_{F\subset I} \bigl(Z_{J}\cap \TV_{IJ}\cap
cl(\TV_{FJ})\bigr)\times \Ga_{I\less F}, \\ \notag
&\qquad = \bigl\{(I,J,y,\al) \ \big| \ \exists F\subset I, \ \al \in \Ga_{I\less F}\ s.t.\  y\in cl(\TV_{FJ})\cap \TV_{IJ} \cap Z_{J}\bigr\}.
\end{align}
The  arguments in Lemma~\ref{le:zero2} apply to show that this set of morphisms, together with inverses, are closed under composition and
 are uniquely determined by their source and target.
 Thus $\Hat\bZ^\nu_\Hh$ is a nonsingular groupoid.  Its  realization $|\Hat\bZ^\nu_\Hh|$ is Hausdorff because 
it is the quotient of the separable, locally compact metric space 
$\Obj_{\Hat\bZ^\nu_\Hh}$ by a relation with closed graph: cf.  \cite[Ch~I,\S10,~Ex.~19]{Bourb} or \cite [Lemma~3.2.4]{MW1}.  Moreover, the space $|\Hat\bZ^\nu_\Hh|$ can be identified with the maximal Hausdorff quotient of
$|\Hat\bZ^\nu|$ because any continuous  map from $\qu{\Obj_{\Hat\bZ^\nu}}{\sim}$ to a Hausdorff space $Y$ must factor through the closure of the relation
$\sim$ induced by the morphisms in $\Hat\bZ^\nu$, and hence descends to $|\Hat\bZ^\nu_\Hh|$.
This proves (i).  
\smallskip
 
   To see that $\pi_{\Hat\bZ^\nu_\Hh}(Z_I)$ is open in $|\Hat\bZ^\nu_\Hh|$ we must show that 
  each intersection $Z_J\cap \pi_{\Hat\bZ^\nu_\Hh}^{-1}\bigl(\pi_{\Hat\bZ^\nu_\Hh}(Z_I)\bigr)$ is open.  Since 
  $\pi_{\Hat\bZ^\nu_\Hh}(Z_I)\cap \pi_{\Hat\bZ^\nu_\Hh}(Z_J) \ne \emptyset$ only if $I\subset J$ or $J\subset I$, it suffices to consider these two cases.  Now
 $Z_J\cap \pi_{\Hat\bZ^\nu_\Hh}^{-1}\bigl(\pi_{\Hat\bZ^\nu_\Hh}(Z_I)\bigr)$  consists of all elements in $Z_J$ that are targets of morphisms with source in $Z_I$.  Therefore
 if $I\subsetneq J$  then  $Z_J\cap \pi_{\Hat\bZ^\nu_\Hh}^{-1}\bigl(\pi_{\Hat\bZ^\nu_\Hh}(Z_I)\bigr) = Z_J\cap \TV_{IJ}$ which is open by Lemma~\ref{le:zero1}~(i).
  On the other hand, if $J\subset I$ then  because the set $  \rho_{JI}(\TV_{JI})$ is $\Ga_J$ invariant, we have
  $$
  Z_J\cap \pi_{\Hat\bZ^\nu_\Hh}^{-1}\bigl(\pi_{\Hat\bZ^\nu_\Hh}(Z_I)\bigr) =
  Z_J\cap 
  \rho_{JI}(\TV_{JI})
  $$
  which is open by Lemma~\ref{le:zero1}~(ii).  
  Thus $\pi_{\Hat\bZ^\nu_\Hh}(Z_I)$ is open. It follows that   the quotient topology on $|\Hat\bZ^\nu_\Hh|$
 has a countable basis because each $Z_I$ does. 
  Moreover  $|\Hat\bZ^\nu_\Hh|$ is regular.  Indeed, by \cite[Lemma~31.1]{Mun}, we need check only that  each point $p\in 
  |\Hat\bZ^\nu_\Hh|$ with neighbourhood $W\subset |\Hat\bZ^\nu_\Hh|$ has a smaller
 neighbourhood $W_1\subset W$ such that $\ov{W_1}\subset W$, and this is an immediate consequence
 of the regularity and local compactness of the sets $Z_I$ and the openness of the sets $\pi_{\Hat\bZ^\nu_\Hh}(Z_I)$.
 Hence $|\Hat\bZ^\nu_\Hh|$ is metrizable by the Urysohn metrization theorem.
 This proves (ii).
  
To prove (iii), note first that for each $x\in Z_{I}$ the  subsets $F\in \Ii_\Kk $ such that $x\in \cl(\TV_{FI})$ are nested and hence have a minimal element $F_x$.  The precompactness of  $V_I$  in $U_I$ implies that $x\in \cl(\TV_{FI})\subset  \TU_{F_xI}$ so that its orbit under $\Ga_{I\less F_x}$ is free.  Moreover, because $F_x\subset F$ for every $F$ 
  for which $x\in \cl(\TV_{FI})$, this orbit $\Ga_{I\less F_x}(x)$ contains  the targets of all the morphisms in
  $\Mor_{\Hat\bZ^\nu_\Hh}$ with source $(I,x)$.
  This proves the formula  $ |\Ga_{I\less F_x}|=$ $
  \# \bigl\{ x\in Z_I \ | \ \pi_{\Hat\bZ^\nu_\Hh}(I,x) = p  \bigr\}$.
  
   It remains to check that  $F_x$, which we defined as
$  \min \bigl\{F:  Z_I\cap cl(\TV_{FI})\cap \pi_{\Hat\bZ^\nu_\Hh}^{-1}(p) \ne \emptyset\bigr\}$, also equals $F_x': = \min\bigl\{F: p\in \ov{\pi_{\Hat\bZ^\nu_\Hh}(Z_{F})}\bigr\}$.
But if $ Z_I\cap cl(\TV_{FI})\cap \pi_{\Hat\bZ^\nu_\Hh}^{-1}(p) \ne \emptyset$ there is a sequence of elements $x_k\in Z_I\cap \TV_{FI}$ with limit $x_\infty\in
\pi_{\Hat\bZ^\nu_\Hh}^{-1}(p)$, which implies by the continuity of $\pi_{\Hat\bZ^\nu_\Hh}$ that, with $x_k': = \rho_{FI}(x_k)$,
the sequence $\pi_{\Hat\bZ^\nu_\Hh}(x_k') = \pi_{\Hat\bZ^\nu_\Hh}(x_k)$ converges to $p$.  Hence $p\in \ov{\pi_{\Hat\bZ^\nu_\Hh}(Z_{F})}$ which implies $F_x'\subset F_x$. Conversely,
if  $p\in \pi_{\Hat\bZ^\nu_\Hh}(Z_{I})\cap \ov{\pi_{\Hat\bZ^\nu_\Hh}(Z_{F})}$ then because $\pi_{\Hat\bZ^\nu_\Hh}(Z_{I})$ is open in $|\Hat\bZ^\nu_\Hh|$ there is a sequence $p_k$ of elements in $\pi_{\Hat\bZ^\nu_\Hh}(Z_{I})\cap \pi_{\Hat\bZ^\nu_\Hh}(Z_{F})$ that converges to 
$p\in \pi_{\Hat\bZ^\nu_\Hh}(Z_{I})$.     By \eqref{eq:HIJ},
this lifts to a sequence $x_k\in \TV_{FI}\subset Z_I$  whose images $\pi_\Kk(\io_{\Hat\bZ^\nu_\Hh}(x_k))$ in $|\Vv|\subset |\Kk|$ converqe to $|\io_{\Hat\bZ^\nu_\Hh}| (p)$,
where $\io_{\Hat\bZ^\nu_\Hh}$ is as in (ii). 
But the composite $\pi_\Kk\circ \io_{\Hat\bZ^\nu_\Hh} : V_I\to \pi_{\Kk}(V_I) \cong\qu{V_I}{\Ga_I}$ simply quotients out by the action of $\Ga_I$ on $V_I$.
Since the projection $V\to\qu{V_I}{\Ga_I}$ is proper by Lemma~\ref{le:vep}~(i), the
 sequence $(x_k)$ must have a convergent subsequence with limit $x_\infty\in V_I$. But then $\pi_{\Hat\bZ^\nu_\Hh}(x_\infty) = \lim_{k\to \infty}  \pi_{\Hat\bZ^\nu_\Hh}(x_k) = p$ by the uniqueness of limits in the Hausdorff space $|\Hat\bZ^\nu_\Hh|$.
 Therefore $x_\infty\in \cl(\TV_{FI})\cap  \pi_{\Hat\bZ^\nu_\Hh}^{-1}(p)$.  Hence by the minimality of $F_x$ we must have $F_x\subset F_x'$.
   This completes the proof.
\end{proof}

\begin{proof}[Proof of Theorem~\ref{thm:zero}]  Let us first consider the case when $\nu$ is an admissible, transverse, precompact  perturbation of 
an oriented tame Kuranishi  atlas $\Kk$ with respect to nested reductions $\Cc\sqsubset \Vv\sqsubset\Obj_{\Bb_\Kk}$. 
Then Lemma~\ref{le:zero1} shows that $\bZ^\nu$ can be completed to an oriented, nonsingular  \'etale groupoid 
$\Hat\bZ^\nu$.  Moreover, by Lemma~\ref{le:zero2} the maximal Hausdorff quotient $|\Hat\bZ^\nu|_\Hh$
can be identified with  the realization $|\Hat\bZ^\nu_\Hh|$ of the groupoid $\Hat\bZ^\nu_\Hh$.  To complete the proof of the first part of the theorem it remains to show that $|\Hat\bZ^\nu_\Hh|$ is compact, and that $(\Hat\bZ^\nu, \La^\nu)$ has the structure of a wnb groupoid  as in Definition~\ref{def:brorb}.

Because $|\Hat\bZ^\nu_\Hh|$ is metrizable by~Lemma\ref{le:zero1}~(ii), it suffices to prove that $|\Hat\bZ^\nu_\Hh|$ is sequentially compact.  Further we saw in  \eqref{eq:zeroVCC} that the precompactness condition for $\nu$ can be written 
without explicit mention of the isotropy groups $\Ga_I$.  Hence the proof of the 
sequential compactness of the zero set given in 
\cite[Theorem~5.2.2]{MW1} carries through without change to the current situation.

We next check that the weighting function $\La^\nu$ is well defined, and compatible with a local branching structure as required by Definition~\ref{def:brorb}.  To see that it is well defined, suppose that 
$p\in \pi_{\Hat\bZ^\nu_\Hh}(Z_I)\cap \pi_{\Hat\bZ^\nu_\Hh}(Z_J)$. As usual we may suppose that
 $I\subset J$, so that $p =  \pi_{\Hat\bZ^\nu_\Hh}(y)$ for some $y\in \TV_{IJ}\subset Z_J$.   If $F_p$ is the minimal set $F$ such that $p\in \ov{\pi_{\Hat\bZ^\nu_\Hh}(Z_{F})}$, then
  there are $|\Ga_{J\less F_p}|$ distinct elements in $Z_J$ that map to $p$.
 Hence $\La(p) = \frac{|\Ga_{J\less F_p}|}{|\Ga_J|}$, and we must check that this agrees with the calculation provided by
 replacing $J$ by $I$.  But if $x=\rho_{IJ}(y)$ then 
 because $F_p\subset I$ does not depend on $I,J$, we have
$\Ga_{J\less F_p} = \Ga_{I\less F_p} \times \Ga_{J\less I}$.  Hence $ \frac{|\Ga_{J\less F_p}|}{|\Ga_J|} =  \frac{|\Ga_{I\less F_p}|}{|\Ga_I|}$.
Thus $\La^\nu$ is well defined.

Finally we describe the local branches at $p\in  |\Hat\bZ^\nu_\Hh|$.  Given $p\in |\Hat\bZ^\nu_\Hh|$ choose a minimal $I$ such that $p\in \pi_{\Hat\bZ^\nu_\Hh}(Z_I)$ and a minimal $F_p\subset I$ such that $p\in \ov{\pi_{\Hat\bZ^\nu_\Hh}(Z_{F_p})}$.  Then $F_p\subset I$, and there is $x\in Z_I\cap \cl(\TV_{F_p I})$ such that $p= \pi_{\Hat\bZ^\nu_\Hh}(x)$.
Because $\pi_{\Hat\bZ^\nu_\Hh}(Z_I)$ is open in  $|\Hat\bZ^\nu_\Hh|$, we may choose an open neighbourhood $N\subset \pi_{\Hat\bZ^\nu_\Hh}(Z_I)$ of $p$ whose closure 
$\ov N$  is 
disjoint from all sets $\ov{\pi_{\Hat\bZ^\nu_\Hh}(Z_{F})}$ with $F\subsetneq F_p$. 
We saw in Lemma~\ref{le:zero2}~(iii) that $Z_I\cap (\pi_{\Hat\bZ^\nu_\Hh})^{-1}(p) = \Ga_{I\less F_p}(x)$.  
Hence,  by shrinking $N$ further if necessary, we may suppose that
there is an precompact  open neighbourhood  $B_x$ of $x$ in $Z_I$ such that 
\begin{itemize}\item $\bigcup_{\ga\in \Ga_{I\less F_p}}  \pi_{\Hat\bZ^\nu_\Hh}(\ga B_x) = N$; 
\item 
the closure $\ov{B_x}$ in $Z_J$ is disjoint from its images under 
the action of $\Ga_{I\less F_p}$.
\end{itemize}
 Then choose the local branches to be the disjoint subsets $\bigl(\ga B_x\bigr)_{\ga\in \Ga_{I\less F_p}}$ of $Z_I$, each with weight
$\frac{1}{|\Ga_I|}$. 
Notice that
\begin{equation}\label{eq:gaBx}
\bigcup_{\ga\in \Ga_{I\less F_p}} \ga \ov{B_x} = Z_I\cap \pi_{\Hat\bZ^\nu_\Hh}^{-1}(\ov{N})\quad\mbox { and } \quad
\bigcup_{\ga\in \Ga_I\less F_p} \ga {B_x} = Z_I\cap \pi_{\Hat\bZ^\nu}^{-1}(N).
\end{equation}
Here,  the first claim  holds because, by the minimality of $F_p$ and the choice of $B_x$,
 $$
 \ov{B_x}\cap \cl(\TV_{FI})\ne \emptyset \;\; \Longrightarrow \;\; F_p\subset F,
 $$
 so that the only morphisms in $\Hat\bZ^\nu_\Hh$ with source in $\bigcup_{\ga\in \Ga_{I\less F_p}} \ga \ov{B_x}$ and target in $Z_I$ 
 are given by the action of an element in $\Ga_{I\less F_p}$ and hence also have target in this set.  
 The second claim holds similarly.

We must check that the three conditions in Definition~\ref{def:brorb} hold.

\begin{itemlist}\item 
The covering property states that
 $(\pi_{|\Hat\bZ^\nu|}^\Hh)^{-1}(N) = \bigcup_{\ga\in \Ga_{I\less F_p}} |\ga B_x|\subset |\bZ|$.
 If  this were false 
  there would be a point $y \in Z_J$ for some $J$ such that there is a morphism in 
  $\Hat\bZ^\nu_\Hh$ from $(J,y)$ to a point $(I,x')\in \ga B_x$ for some $\ga\in \Ga_{I\less F_p}$,  but no such morphism in $\Hat\bZ^\nu$.  
  By construction, the morphisms in $\Hat\bZ^\nu_\Hh$ from $Z_J$ to $Z_I$ are composites of morphisms of type (B) from $Z_J$ to $Z_I$ 
  (which lie in $\Hat\bZ^\nu$) with  morphisms in the closure of $\Mor_{\Hat\bZ^\nu}(Z_I,Z_I)$.  Therefore it suffices to consider the case $J = I$, and $y\notin \bigcup_{\ga\in \Ga_{I\less F_p}} \ga {B_x}$.  
  But  \eqref{eq:gaBx} implies that the only  elements of $\Mor_{\Hat\bZ^\nu}(Z_I,Z_I)$ with  target in $\ga B_x$ must have source in
  some set $\al B_x$.  Therefore such $y$ does not exist.

\item For local regularity, we must check that for each $\ga$ the projection $\pi_{\Hat\bZ^\nu_\Hh}: \ga B_x \to |\Hat\bZ^\nu_\Hh|$ is a homeomorphism onto a relatively closed subset of $N$.  
But \eqref{eq:gaBx} implies that this map extends to an injective, continuous  map $f$ with compact domain $\ga \ov{B_x}$.  Hence 
$f$ is a homeomorphism onto its image because compact subsets of the Hausdorff space $|\Hat\bZ^\nu_\Hh|$ are closed.
Further, $\pi_{\Hat\bZ^\nu_\Hh}(\ga B_x ) = N\cap \pi_{\Hat\bZ^\nu_\Hh}(\ga \ov{B_x})$ is closed in $N$ because it is the intersection of a compact set with $N$.
\item Finally, note that $\La^\nu$ equals the branching function specified in Definition~\ref{def:brorb}; indeed,
the number of branches through  $q\in N$ is just  the number of preimages of $q\in N$ in 
$\bigcup_{\ga\in \Ga_I\less F_p} \ga {B_x}$, and  
we saw in Lemma~\ref{le:zero2}~(iii) 
that this  is  $|\Ga_{I\less F_q}|$, where $F_q\supset F_p$ is the minimal set $F$ such that
$q\in \ov{\pi_{\Hat\bZ^\nu_\Hh}(Z_F)}$.
\end{itemlist}
This completes the proof that $(\Hat\bZ^\nu,\La^\nu)$ is a compact wnb groupoid. It has a fundamental class by Proposition~\ref{prop:fclass}, and hence  defines a cycle in $\Cc$ as claimed.

The same arguments apply when  $\Kk$ is a Kuranishi cobordism.  In particular,  $|\Hat\bZ^\nu_\Hh|$ is compact so that, by Lemma~\ref{le:zero1}~(iii), $(\Hat\bZ^\nu,\La^\nu)$ is a wnb cobordism groupoid,  and the 
 boundary restrictions have the required properties by  Lemma~\ref{le:locorient}~(iv).
\end{proof}

We end this section by some elementary examples of this construction -- the fundamental class and the Euler class of an orbifold represented by Kuranishi atlases.

\begin{example}\label{ex:foot2}\rm  
Consider the orbifold case, i.e.\ a Kuranishi atlas $\Kk$ on $X$ with trivial obstruction spaces so that $\s_\Kk$ and $\nu$ are identically zero and $\io_\Kk: X \to |\Kk|$ is a homeomorphism. 
In this case the zero set $\bZ$ should represent the fundamental class of the oriented orbifold. 
We suppose that $X = \qu{M}{\Ga}$ is the quotient of a compact oriented smooth manifold $M$ 
by the action of a finite group $\Ga$, and that $\Kk$ is the atlas with a
single chart  with domain $M$ and $E = \{0\}$.  Then 
$\bZ = \bB_\Kk|^{\less \Ga}$ is the category with objects $M$ and only identity morphisms, because there are no pairs  $I,J\in \Ii_\Kk$ such that $\emptyset\ne I\subsetneq J$.
Therefore
 $\bZ = \bZ_\Hh$ has realization 
$|\bZ_\Hh| = M$ and 
  the weighting function $\La:M\to \Q$ is given by $\La(x) = \frac{1}{|\Ga|}$. 
  If the action of $\Ga$ is effective on every open subset of $M$ then
 the pushforward of $\La$ by $\io_{\bZ_\Hh}: M\to X$, which is defined by
$$
(\io_{\bZ_\Hh})_*\La(p): = \sum_{x\in  (\io_{\bZ_\Hh})^{-1}(p)} \La(x)
$$
takes the value $1$ at every smooth point (i.e.\ point with trivial stabilizer) of the orbifold $\qu{M}{\Ga}$.     On the other hand, if $\Ga$ acts by the identity so that the action is totally noneffective then $\io_{\bZ_\Hh}: M\stackrel{\cong}\to X$ is the identity map and 
the weighting function $X\to \Q^+$ takes the constant value $\frac 1{|\Ga|}$.  

Note that if we construct a fundamental class on $|\bZ|_\Hh$ by the method  of  Proposition~\ref{prop:fclass} then our choice of weights gives
a  class that is consistent with standard conventions.
For example, in dimension $d=0$ the branched manifold  $Z = |\bZ|_\Hh$ is a finite collection of points $\{p_1,\dots, p_k\}$, one for each equivalence class in $\Obj_{\bZ}$, where 
the point $p_i$ corresponding to an equivalence class with stabilizer $\Ga^i$ has weight  $\frac 1{|\Ga^i|}$.  If each point is positively oriented, then 
the \lq\lq number of elements" in $|\bZ|_\Hh$ is the sum
$\sum_{i=1}^k \frac 1{|\Ga^i|}$, which gives the Euler characteristic of the groupoid; 
c.f.\ \cite{Wein}.
Other more substantive  examples such as that of the football of Example~\ref{ex:foot} are discussed in \cite{Mcorb}.
\end{example}

\begin{example}\rm\label{ex:sphere} 
Examples of Kuranishi atlases with nontrivial obstruction spaces can be seen in the calculation of
the Euler class of the tangent bundle of $S^2$ and of the football
orbifold using Kuranishi atlases.

\noindent
(i)
To build a Kuranishi  atlas that models $\rT S^2$, cover $S^2$ by two discs $D_1,D_2$ whose intersection $D_1\cap D_2 = : D_{12}=:A$ 
is an annulus, and for $i=1,2$ define $\bK_i: = (U_i: = D_i,\ E_i: = \C,\ s_i : = 0,\ \psi_i: = \id)$.  For $i=1,2$, choose trivializations
$\tau_i: D_i\times \C\to \rT S^2|_{D_i}, (x,e)\mapsto \tau_{i,x}(e)$ and then define the transition chart 
$$
\bK_{12}: = \bigl( U_{12}\subset E_1\times E_2\times A ,\ E_1\times E_2,\ s_{12} = \pr_{E_1\times E_2},\ \psi_{12} = \pr_A|_{0\times 0\times A}\bigr)
$$
where
$$
U_{12}: = \bigl\{ (e_1,e_2,x) \ \big| \ x\in A, \tau_{1,x}e_1 + \tau_{2,x}e_2 = 0\bigr\}.
$$
The coordinate changes $\Hat\Psi_{i,12}$ are given by taking $U_{i,12} = A$ and
$
\psi_{i,12} (x) = (0,0,x)$.
To justify this choice
of Kuranishi atlas note that one can construct a commutative diagram 
\[
\xymatrix{
|\bE_\Kk| \ar@{->}[d]\ar@{->}[r] & \rT S^2\ar@{->}[d]\\
|\bB_\Kk| \ar@{->}[r] & S^2,
}
\]
that restricts over $U_{12}\times E_{12}$ to the map $\bigl((e_1,e_2,x),e_1',e_2')\mapsto \tau_{1,x}(e_1+e_1') + 
\tau_{2,x}(e_2+e_2')\in \rT S^2|_A$. 
This construction is generalized to other (orbi)bundles in \cite{Mcn}.

Next, in order to calculate the Euler class we 
identify $A$ with $[0,1]\times S^1$ and consider the corresponding trivialization  $\rT S^2|_A= A\times \R_t\times \R_{\theta}$ where $t\in [0,1]$ and $\theta\in S^1$ are  coordinates.  Then for $i=1,2$  there is a section $\nu_i: U_i\to E_i$ with one transverse zero such that $\tau_{i,x}(\nu_i(x)) = (x,1,0)\in \rT S^2|_A$ for $x\in A$. (Take suitably modified 
versions of the sections $\nu_1(z) = z, \nu_2(z) = -z$ where $D_i\subset \C$.)

Choose a reduction of the footprint covering with $V_{12} = (\eps,1-\eps)\times S^1$ for some $\eps\in (0,\frac 14)$ 
and so that $\TV_{1,12} = (0,0)\times (\eps,\frac 14]\times S^1 \subset U_{12}$ and
$\TV_{2,12} = (0,0)\times  (\frac 34,1-\eps)\times S^1$, and  choose a cutoff function $\be:[0,1]\to [0,1]$ that equals $1$ in $[0,\frac 14]$ and $0$ in $[\frac 34,1]$.  Then the map $\nu_{12}: V_{12}\to E_1\times E_2$ given by
$$
\nu_{12}(e_1,e_2,x) \;=\;\bigl(\be(x) \nu_1(x),(1-\be(x))\nu_2(x)\bigr)  
\;\in\; E_1\times E_2
$$
defines an admissible perturbation section that restricts to $\nu_i$ on $V_{i, 12}\subset (0,0)\times A$ for $i=1,2$.  Moreover $s_{12} +\nu_{12}$ does not vanish at any point $(e_1,e_2,x)\in V_{12}$ because the equation 
$ \tau_{1,x}(e_1) + \tau_{2,x}(e_2)=0$ together with 
$$
0 = \tau_{1,x}(e_1) + \be(x) (1,0)=\tau_{2,x}(e_2) + (1-\be(x))(1,0)\in  x\times \R_t\times \R_{\theta} \in \rT S^2|_A
$$
imply that the vector $(1,0)$ is zero, a contradiction. Hence the perturbed zero set $\bZ^\nu$ consists of two points, each with weight one.
\MS

\NI (ii)  It is easy to adjust this example to the tangent bundle of 
the \lq\lq football" discussed in Example~\ref{ex:foot}.  In this case, the zero of the section $s_i+\nu_i$ would count with weight $\frac 1{|\Ga_i|}$ so that the Euler class is $\frac 12 + \frac 13$.  For further details of this and other related examples 
see \cite[\S5]{Mcn}.
\end{example}

\subsection{Construction of the Virtual Moduli Cycle and Fundamental Class}\label{ss:A}\hspace{1mm}\\ \vspace{-3mm}

The next step in the Kuranishi regularization Theorem~A is to construct admissible, transverse, precompact perturbations $\nu$ that are unique up to interpolation by admissible, transverse, precompact cobordism perturbations. 
This -- quite complicated -- construction is developed in complete detail in \cite{MW2} in such a way that it applies directly to our present setting, in which the Kuranishi atlas $\Kk$ has nontrivial isotropy groups, but the reduced and pruned category $\Bb_\Kk|_\Vv^{\setminus\Ga}$ is nonsingular -- i.e.\ the remaining isotropy groups act freely. 
While we defer most of the proofs to \cite{MW2}, we will give full technical statements of the existence and uniqueness of perturbations, so that our constructions of VMC/VFC can be compared directly to other approaches, without reference to \cite{MW2}. 
Based on this, Definition~\ref{def:VMCF} and Theorem~\ref{def:VMCF} then define the virtual moduli cycle (VMC) as a cobordism class of closed oriented weighted branched manifolds and construct the virtual fundamental class (VFC) as \v{C}ech homology class.

For the construction of (cobordism) perturbations we will consider a metric tame Kuranishi atlas (or cobordism) $(\Kk,d,\si)$, that is we fix the following data:
\begin{itemlist}
\item 
$\Kk$ is a tame Kuranishi atlas on a compact metrizable space $X$ in the sense of Definitions~\ref{def:Ku} and \ref{def:tame}, or it is a tame Kuranishi cobordism on a compact collared cobordism $Y$ in the sense of Definitions~\ref{def:CKS} and \ref{def:tame}.
\item
$d$ is an admissible metric on $|\Kk|$ in the sense of Definition~\ref{def:metrizable}.
\item
If $(\Kk,d)$ is a metric, tame Kuranishi cobordism on $Y$, then the boundary restrictions
 $(\Kk^\al,d^\al):=(\p^\al\Kk, d|_{|\p^\al\Kk|})$ 
are metric, tame Kuranishi atlases on $\p^\al Y$ for $\al=0,1$.
\end{itemlist}

\NI
For easy reference we list some consequences of this setting and notation conventions.

\begin{itemlist}
\item 
The associated intermediate Kuranishi atlas $\uKk$ is a tame topological Kuranishi atlas (resp.\ cobordism) by Lemma~\ref{le:Ku3} (resp.\ Remark~\ref{rmk:restrict}~(ii)), which has the same realization $|\uKk|=|\Kk|$, equipped with the quotient topology. 

\item
$d$ is a bounded metric on the set $|\uKk|$ such that for each $I\in \Ii_\Kk$ the pullback metric $\und d_I:=(\pi_\uKk|_{\uU_I})^*d$ on $\uU_I$ induces the quotient topology on the intermediate domain $\uU_I=\qu{U_I}{\Ga_I}$. By construction, these also induce $\Ga_I$-invariant pseudometrics $d_I:=(\pi_\uKk|_{U_I})^*d = \pi_I^* \und d_I$ on the Kuranishi domains $U_I$ of $\Kk$. 
Moreover, \cite[Lemma~3.1.8]{MW1} shows that these (pseudo)metrics are compatible with coordinate changes.
We denote the $\de$-balls around subsets $Q\subset |\uKk|$ resp.\ $R\subset \uU_I$ resp.\ $S\subset U_I$ for $\de>0$ by
\begin{align*}
B_\de(Q) &\,:=\; \bigl\{w\in |\Kk|\ | \ \exists q\in Q : d(w,q)<\de \bigr\}, \\
B^I_\de(R) &\,:=\; \bigl\{x\in \uU_I\ | \ \exists r\in R : \und d_I(x,r)<\de \bigr\}, \\
\Hat B^I_\de(S) &\,:=\; \bigl\{y\in U_I\ | \ \exists s \in S : d_I(y,s)<\de \bigr\},
\end{align*}
and note that balls in the pseudometric are $\Ga_I$-invariant preimages of balls in $\uU_I$,
\begin{equation}\label{eq:hat}
\Hat B^I_\de(S) = \pi_I^{-1}\bigl(  B^I_\de(\und S) \bigr)  
\qquad\qquad\text{and}\qquad\qquad
B_\de(\pi_\Kk(S))=B_\de(\pi_{\uKk}(\und S)).
\end{equation}

\item
While the metric topology on $|\uKk|$ is generally not compatible with the quotient topology, we know from \cite[Lemma~3.1.8]{MW1} that the identity map $|\Kk|\to(|\Kk|,d)$ is continuous, and thus $|\Kk|$ is a Hausdorff topology in which the metric $\de$-balls are open, and thus neighbourhoods.
\end{itemlist}

\NI
Given this setting, our goal is to construct admissible, precompact, transverse (cobordism) perturbations of the section $\s_\Kk|_\Vv^{\less\Ga}$ over a pruned domain category $\bB_\Kk|_\Vv^{\less\Ga}$; see Definition~\ref{def:sect} and Lemma~\ref{le:prune}. For that purpose we will also need to fix nested (cobordism) reductions $\Cc \sqsubset \Vv$ of $\Kk$.
These induce the following crucial data, on which the iterative construction of perturbations depends.
The claims here are all consequences of \cite[Theorem~5.1.6~(iii)]{MW1} and \cite[Lemma~7.3.4, Proposition~7.3.10]{MW2} applied to $\uKk$ together with \eqref{eq:hat} and properness of the projections $\pi_I:U_I\to\uU_I$ established in Lemma~\ref{le:vep}~(i).

\begin{itemlist}
\item
Given a reduction $\Vv$ of $\Kk$, there exists $\de_\Vv\in (0,\frac 14]$ such that for any $\de< \de_\Vv$ we have
\begin{align*}
\Hat B^I_{2\de}(V_I)\sqsubset U_I\qquad &\forall I\in\Ii_\Kk , \\
B_{2\de}(\pi_\Kk({V_I}))\cap B_{2\de}(\pi_\Kk({V_J}))
 \neq \emptyset &\qquad \Longrightarrow \qquad I\subset J \;\text{or} \; J\subset I .
\end{align*}
This gives rise to a continuum of nested reductions
$V_I \;\sqsubset\; \ldots V^{k''}_I \;\sqsubset\; V^{k'}_I \ldots \;\sqsubset\; V^0_I$
for $k''>k'> 0$, which for $k \geq 0$ are given by
$$
V_I^k \,:=\; \Hat B^I_{2^{-k}\de}(V_I) \;=\; \pi_I^{-1}\bigl( \uV_I^k \bigr)
\;\sqsubset\; U_I  \qquad\text{with}\qquad \uV_I^k:= B^I_{2^{-k}\de}(\uV_I) .
$$
\item
For suitable $k\geq 0$, the iteration will construct $\nu_J$ by extension of the pullbacks $\rho_{IJ}^*\nu_I$, which are defined for $I\subsetneq J$ on $N^k_{JI} := V^k_J \cap \pi_\Kk^{-1}(\pi_\Kk(V^k_I))$, also given as 
$$
N^k_{JI} \;=\; V^k_J \cap \rho_{IJ}^{-1}(V^k_I)  
\;=\, \pi_J^{-1}\bigl( \uN^k_{JI} \bigr)
\qquad\text{with}\qquad
\uN^k_{JI}  := \uV^k_J \cap \uphi_{IJ}(\uV^k_I \cap \uU_{IJ}).
$$

\item
We need to make a choice of {\bf equivariant norms on the obstruction spaces} as follows: 
For each basic chart $i\in\{1,\ldots,N\}$ we choose a $\Ga_i$-invariant norm $\|\cdot\|$ on $E_i$. Then the $\Ga_J$-invariant norm on $E_J$ for each $J\in\Ii_\Kk$ is given by
$$
\| e \| \,:=\;
\bigl\| {\textstyle \sum_{i\in J}} \Hat\phi_{iJ} (e_i) \bigr\| \,:=\; \max_{i\in J} \| e_i\|
\qquad
\forall e=  {\textstyle \sum_{i\in J}} \Hat\phi_{iJ} (e_i) \in E_J. 
$$

\item
While the sections $s_I: U_I\to E_I$ only induce continuous maps $\und s_I:\uU_I \to \qu{E_I}{\Ga_I}$ to the quotient of obstruction spaces, equivariance of the norms guarantees that the norm of sections descends to a continuous function $\|\und s_I\|: \uU_I \to [0,\infty)$ given by $x \mapsto \|s_I(y)\|$ for any $y\in\pi_I^{-1}(x)$. 
These functions provide (rather nontransverse) topological Kuranishi charts over the intermediate domain with the same footprint: $\upsi_I$ maps $\|\und s_I\|^{-1}(0)=\qu{s_I^{-1}(0)}{\Ga_I}$ homeomorphically to $F_I$.

\item
Given equivariant norms $\|\cdot\|$, nested reductions $\Cc\sqsubset\Vv$, and $0<\de<\de_\Vv$ we have
\begin{align*}
\si(\Vv,\Cc,\|\cdot\|,\de) &\,:=\; \min_{J\in\Ii_\Kk}
\inf \Bigl\{
\; \bigl\| s_J(x) \bigr\| \;\Big| \;
x\in \ov{V^{|J|}_J} \;\less\; \Bigl( \Ti C_J \cup {\textstyle \bigcup_{I\subsetneq J}} \Hat B^J_{\eta_{|J|-\frac 12}}\bigl(N^{|J|-\frac14}_{JI}\bigr) \Bigr) \Bigr\}  \\
&\,=\; \min_{J\in\Ii_\Kk}
\inf \Bigl\{
\; \bigl\| \und s_J \bigr\| (y) \;\Big| \;
y\in \ov{\uV^{|J|}_J} \;\less\; \Bigl(\und{\Ti C}_J  \cup {\textstyle \bigcup_{I\subsetneq J}} B^J_{\eta_{|J|-\frac 12}}\bigl(\uN^{|J|-\frac14}_{JI}\bigr) \Bigr) \Bigr\}  
\; >0 ,
\end{align*}
where $\eta_{k-\frac 12} := 2^{-k+\frac 12} (1 -  2^{-\frac 14} ) \de$ and a set containing $s_J^{-1}(0)=\pi_J^{-1}(\|\und s_J\|^{-1}(0))$ is 
$$
\Ti C_J \,:=\; {\textstyle \bigcup_{K\supset J} } \, \rho_{JK}(C_K)  
\;=\; \pi_J^{-1}\bigl(\und{\Ti C}_J  \bigr)
\qquad\text{with}\qquad
\und{\Ti C}_J \,:=\;  {\textstyle\bigcup_{K\supset J}} \, \uphi_{JK}^{-1}(\uC_K)  .
$$
\item
In the case of a metric tame Kuranishi cobordism $(\Kk,d)$ with equivariant norms $\|\cdot\|$ and nested cobordism reductions $\Cc\sqsubset\Vv$, let $\eps>0$ be the smallest of the collar widths of $\Kk,d,\Cc$, and $\Vv$.
Then for $0<\de<\min\{\eps,\de_\Vv\}$ we obtain positive numbers
\begin{align*}
 \si' (\Vv,\Cc,\|\cdot\|,\de) &\,:=\; \min_{J\in\Ii_\Kk} \inf \Bigl\{ \; \bigl\| s_J(x) \bigr\| \;\Big| \;
x\in \ov{V^{|J|+1}_J} \;\less\; \Bigl( \Ti C_J \cup {\textstyle \bigcup_{I\subsetneq J}} \Hat B^J_{\eta_{|J|+\frac 12}}\bigl(N^{|J|+\frac34}_{JI}\bigr) \Bigr) \Bigr\}  , \\
\si_{\rm rel}(\Vv,\Cc,\|\cdot\|,\de) &\,:=\; \min\bigl\{ \si(\p^0\Vv,\p^0\Cc,\p^0\|\cdot\|,\de), \,\si(\p^1\Vv,\p^1\Cc,\p^1\|\cdot\|,\de), \,\si'(\Vv,\Cc,\|\cdot\|,\de) \bigr\} .
\end{align*}
Here $\p^\al\|\cdot\|$ denotes the collection of equivariant norms on $E_I$ for $I\in\Ii_{\p^\al\Kk}\subset\Ii_\Kk$. 
\end{itemlist}

\MS
\NI
The constants $\de_\Vv$ and $\si(\Vv,\Cc,\|\cdot\|,\de)$ defined here will control the permitted support and norm of the perturbation $\nu$ for a Kuranishi atlas.
In particular, $\de_\Vv$ measures the separation between the components $V_I\neq V_J$ of the reduction $\Vv$, while $\si(\Vv,\Cc,\|\cdot\|,\de)$ measures the minimal norm of $\s_\Kk|_\Vv^{\less\Ga}$ on the complement of an open neighbourhood of the set $\pi_\Kk^{-1}(\pi_\Kk(\Cc))$, in which all perturbed zero sets will need to be contained.
We will construct perturbations $\nu = (\nu_I : V_I \to E_I )_{I\in\Ii_\Kk}$ by an iteration which constructs and controls each $\nu_I$ over the larger set ${V_I^{|I|}}$. Here the domains are determined by a choice of $0<\de<\de_\Vv$ and we ensure that the perturbed zero sets are contained in $\pi_\Kk^{-1}(\pi_\Kk(\Cc))$ by bounding the perturbations $\|\nu_I\|<\si$ by some $0<\si<\si(\Vv,\Cc,\|\cdot\|,\de)$.
In order to prove uniqueness of the VMC, we moreover have to interpolate between any two such perturbations. 
This requires the adjusted bound $\si_{\rm rel}(\Vv,\Cc,\|\cdot\|,\de)$ on the norm of cobordism perturbations for the following reason:
The construction of a cobordism perturbation with prescribed boundary values is achieved by an iteration on the domains $V^{|I|+1}_I$ instead of $V^{|I|}_I$, which guarantees that the boundary values -- which got constructed in iterations over $\p^\al V^{|I|}_I$ -- are given on sufficiently large boundary collars.
In view of this, it is also necessary to keep track of the refined properties arising from the iterative construction of a perturbation by the following notion of $(\Vv,\Cc,\|\cdot\|,\de,\si)$-adapted, as well as a stronger notion which guarantees extensions to Kuranishi concordances.

\begin{defn}  \label{a-e}
Given nested reductions $\Cc\sqsubset\Vv$ of a metric tame Kuranishi atlas $(\Kk,d)$, 
a choice of equivariant norms $\|\cdot\|$ on the obstruction spaces, and constants $0<\de<\de_\Vv$ and $0<\si\le\si(\Vv,\Cc,\|\cdot\|,\de)$, we say that a perturbation $\nu$ of $\s_\Kk|_\Vv^{\less\Ga}$ is {\bf $(\Vv,\Cc,\|\cdot\|,\de,\si)$-adapted} if the sections $\nu_I:V_I\to E_I$ extend to sections over ${V^{|I|}_I}$ (also denoted $\nu_I$) so that the following conditions hold for every $k=1,\ldots, M_\Kk:= \max_{I\in\Ii_\Kk} |I|$ 
\begin{itemize}
\item[a)]
The perturbations are compatible in the sense that for $ H\subsetneq I , |I|\leq k$ we have
$$
\nu_I |_{\rho_{HI}^{-1}(V^k_H)\cap V^k_I} \;=\; \Hat\phi_{HI} \circ \nu_H \circ \rho_{HI}
|_{\rho_{HI}^{-1}(V^k_H)\cap V^k_I} .
$$
\item[b)]
The perturbed sections are transverse, that is $(s_I|_{{V^k_I}} + \nu_I) \pitchfork 0$ for each $|I|\leq k$.
\item[c)]
The perturbations are strongly admissible, that is for all $H\subsetneq I$ and $|I|\le k$ we have $\;\nu_I( \Hat B^I_{\eta_k}(N^{k}_{IH})\bigr) \subset \Hat\phi_{HI}(E_H)$.
\item[d)]  
The perturbed zero sets are controlled by
$\; \pi_\Kk\bigl( (s_I |_{{V^k_I}}+ \nu_I)^{-1}(0) \bigr) \subset \pi_\Kk(\Cc)$
for $|I|\leq k$.
\item[e)]
The perturbations are small, that is $\;\sup_{x\in {V^k_I}} \| \nu_I (x) \| < \si$
for $|I|\leq k$. 
\end{itemize}

\NI  
Also, we say that a perturbation $\nu$ is {\bf strongly $(\Vv,\Cc)$-adapted} 
if it is a ${(\Vv,\Cc,\|\cdot\|,\de,\si)}$-adapted perturbation $\nu$ of $\s_\Kk|_\Vv^{\less\Ga}$ for some choice of equivariant norms $\|\cdot\|$ and constants $0<\de<\de_\Vv$ and 
$$
0<\si\le\si_{\rm rel}([0,1]\times \Vv,[0,1]\times  \Cc,\|\cdot\|,\de) = \min\bigl\{
\si(\Vv,\Cc,\|\cdot\|,\de) , \si'([0,1]\times \Vv, [0,1]\times \Cc,\|\cdot\|,\de) \bigr\} ,
$$
where we use the product metric on $[0,1]\times |\Kk|$.
\end{defn}

\begin{remark}\rm
(i)
Adapted perturbations are automatically admissible, precompact, and transverse in the sense of Definition~\ref{def:sect}. Indeed, this is guaranteed by the inclusions $V^k_I\subset V_I$ and the fact that strong admissibility $\nu_I(x)\in \im\Hat\phi_{HI}$ for $x\in\Hat B^I_{\eta_k}(N^{k}_{IH})$ for $H\subsetneq I$
implies admissibility $\im \rd_y \nu_I \subset\im\Hat\phi_{HI}$ for 
$y\in \TV_{HI} = V_I \cap \rho_{HI}^{-1}(V_H) \subset V^k_I \cap \rho_{IJ}^{-1}(V^k_I) = N^{k}_{IH}$.

\MS
\NI
(ii)
The admissibility condition is crucial for the transfer of transversality as follows:
Let $\nu$ be an admissible perturbation and let $z\in V_I$ and $w\in V_J$ so that $\pi_\Kk(z)=\pi_\Kk(w)\in|\Kk|$. Then $z$ is a transverse zero of $s_I|_{V_I}+\nu_I$ if and only if $w$ is a transverse zero of $s_J|_{V_J}+\nu_J$.

Indeed, by the reduction property we can assume w.l.o.g.\ $I\subset J$ and thus $z=\rho_{IJ}(w)$.
Since $\rho_{IJ}$ is a regular covering, we can pick a local inverse $\phi_IJ$ so that $w=\phi_{IJ}(z)$. Then the proof of \cite[Lemma~7.2.4]{MW2} directly applies -- using the index condition in terms of $\phi_{IJ}$.

\MS\NI
(iii)
Any $(\Vv,\Cc,\|\cdot\|,\de,\si)$-adapted perturbation for fixed $\Vv,\Cc,\|\cdot\|,\de$ and sufficiently small $\si>0$ is in fact strongly adapted.
Indeed, due to the product structure of all sets and maps involved in the definition of $\si'$, we may rewrite the condition on $\si>0$ in the definition of strong adaptivity as
$\si< \| s_J(x) \|$ for all $x\in \ov{V^{k}_J} \;\less\; \bigl( \Ti C_J \cup {\textstyle \bigcup_{I\subsetneq J}} \Hat B^J_{\eta_{k-\frac 12}}\bigl(N^{k-\frac14}_{JI}\bigr) \bigr)$, $J\in\Ii_\Kk$, and $k\in\{|J|,|J|+1\}$.
$\hfill\er$
\end{remark}

By the above remark, the following in particular proves the existence of admissible, precompact, transverse perturbations as well as strongly adapted perturbations.

\begin{prop}\label{prop:ext}
\begin{enumerate}
\item
Let $(\Kk,d)$ be a metric tame Kuranishi atlas with nested reductions $\Cc \sqsubset \Vv$ and equivariant norms $\|\cdot\|$ on the obstruction spaces.
Then for any $0<\de<\de_\Vv$ and $0<\si\le\si(\Vv,\Cc,\|\cdot\|,\de)$ there exists a $(\Vv,\Cc,\|\cdot\|,\de,\si)$-adapted perturbation $\nu$ of $\s_\Kk|_{\Vv}^{\less\Ga}$. 
\item
Let $(\Kk,d)$ be a metric tame Kuranishi cobordism with nested cobordism reductions $\Cc \sqsubset \Vv$, equivariant norms $\|\cdot\|$ on the obstruction spaces, and minimal collar width $\eps>0$ of $(\Kk,d)$
and the reductions $\Cc,\Vv$. 
Then, given $0<\de<\min\{\eps,\de_\Vv\}$, $0<\si\le \si_{\rm rel}(\de,\Vv,\Cc)$, and
perturbations $\nu^\al$ of $\s_{\p^\al\Kk}|_{\p^\al\Vv}^{\less\Ga}$ for $\al=0,1$ that are $(\p^\al\Vv,\p^\al\Cc,\de,\si)$-adapted, 
there exists an admissible, precompact, transverse cobordism perturbation $\nu$ of $\s_\Kk|_\Vv^{\less\Ga}$ with $\pi_\Kk\bigl((\s_\Kk|_\Vv^{\less\Ga}+\nu)^{-1}(0)\bigr)\subset \pi_\Kk(\Cc)$ and $\nu|_{\p^\al\Vv}=\nu^\al$ for $\al=0,1$.
\item
In the case of a product cobordism $[0,1]\times \Kk$ 
with product metric and nested product reductions 
$[0,1]\times \Cc\sqsubset  [0,1]\times \Vv$,
  (ii) holds for $0<\de<\de_{[0,1]\times \Vv}$ without restriction from the collar width.
\end{enumerate}
\end{prop}

\begin{proof}
As explained in \cite[Remark~7.3.2]{MW2}, the iterative constructions in \cite[Propositions~7.3.7, 7.3.10]{MW2} generalize directly to our setup based on the pruned domain category $\Bb_\Kk|_{\Vv}^{\less\Ga}$. We indicated the necessary adjustments in a series of footnotes in the proofs of \cite{MW2}.
Beyond the above setting and notations, this requires the following two systematic changes.

Firstly, all relationships between (or definitions/constructions of) subsets of $\Obj_{\Bb_\Kk}=\bigcup_{I\in\Ii_\Kk} U_I$ in \cite{MW2} should be replaced by two statements -- one for subsets of $\Obj_{\Bb_\uKk}=\bigcup_{I\in\Ii_\Kk} \uU_I$ in the intermediate atlas $\uKk$, and one for subsets in the pruned domain category $\Bb_\Kk|_{\Vv}^{\less\Ga}$ with $B_\de$ replaced by $\Hat B_\de$. These two statements will always be equivalent via the projection $\pi_I$. Statements can then be checked by working in the intermediate category, but they will be applied on the level of the pruned domain category. 
Here it is crucial to know that the projections $\pi_I:U_I\to \uU_I$ are continuous (by definition of the quotient topology) and proper by Lemma~\ref{le:vep}~(i).

Secondly, our goal -- constructing a precompact, transverse, admissible (cobordism) perturbation 
 $\nu:\bB_\Kk|_\Vv\to\bE_\Kk|_\Vv$ -- is essentially the same as that of \cite[Definitions~7.2.1,7.2.5,7.2.6]{MW2}.
Writing it in terms of the maps $\nu=(\nu_I:V_I\to E_I)_{I\in\Ii_\Kk}$, the only difference is that the compatibility conditions in \cite[(7.2.1)]{MW2},
$$
\nu_J\big|_{N_{JI}} \; =\; 
 \Hat\phi_{IJ}\circ\nu_I\circ \phi_{IJ}^{-1}\big|_{N_{JI}}
\quad \text{on}\quad
N_{JI}:=V_J \cap \phi_{IJ}(V_I \cap U_{IJ})
$$
for all $I \subsetneq J$ are replaced by
$$
\nu_J\big|_{\TV_{IJ}}\  =\ 
 \Hat\phi_{IJ}\circ\nu_I\circ \rho_{IJ}\big|_{\TV_{IJ}}
\quad \text{on}\quad
\TV_{IJ} := V_J\cap \rho_{IJ}^{-1}(V_I) ,
$$
and the precompactness conditions in \cite[(7.2.5)]{MW2},
$$
(s_J|_{V_J} + \nu_J)^{-1}(0)
\;\subset\;  {\textstyle \bigcup_{H\supset J} } \,\phi_{JH}^{-1}(C_H) \; \cup \; {\textstyle \bigcup_{H\subsetneq J} }\,\phi_{HJ}(C_H) 
$$
for all $J\in \Ii_\Kk$ are replaced by \eqref{eq:zeroVCC} above
$$
(s_J|_{V_J} + \nu_J)^{-1}(0)
\;\subset\;  {\textstyle \bigcup_{H\supset J} } \,\rho_{JH}(C_H) \; \cup \; {\textstyle \bigcup_{H\subsetneq J} }\,\rho_{HJ}^{-1}(C_H) .
$$
Here our setup guarantees that $\rho_{IJ}:\TV_{IJ}\to V_I \cap \rho_{IJ}(V_J)\subset U_{IJ}$ is a regular covering (i.e.\ local diffeomorphism with fibers given by the free action of a finite group $\Ga_{J\less I}\cong \qu{\Ga_J}{\Ga_I}$) analogous to $\phi_{IJ}^{-1}:N_{IJ}\to V_I \cap \phi_{IJ}^{-1}(V_J)\subset U_{IJ}$ in \cite{MW2}, which is a regular covering with trivial fibers.
Thus to adapt the proofs of \cite{MW2} one should replace $\phi_{IJ}$ with $\rho_{IJ}^{-1}$ and identify $N_{IJ}=\TV_{IJ}$.
\end{proof}

Finally, we make the additional choice of an orientation of the Kuranishi atlases/cobordisms in the sense of Definition~\ref{def:orient} to prove Theorem A from the introduction.

\begin{definition}\label{def:VMCF}
Let $(\Kk,\si)$ be an oriented weak Kuranishi atlas of dimension $D$ on a compact, metrizable space $X$. 
Then its {\bf virtual moduli cycle} $\Zz^\Kk:=\bigl[ \bigl( |\bZ^\nu|_\Hh , \La^\nu\bigr) \bigr]$ is the cobordism class of weighted branched manifolds (without boundary) of dimension $D$ given by the choices of a preshrunk tame shrinking $\Kk_{\rm sh}$ of $\Kk$, an admissible metric on $|\Kk_{\rm sh}|$, nested reductions $\Cc\sqsubset\Vv$ of $\Kk_{\rm sh}$, and a strongly $(\Vv,\Cc)$-adapted perturbation $\nu$.

Moreover, the {\bf virtual fundamental class}
$$
[X]^{\rm vir}_\Kk \,:=\; |\psi_{\Kk_{\rm sh}}|_* \bigl( \, \underset{\leftarrow}\lim\, [\, \io^{\nu_k} \,] \, \bigr) \;\in\; \check{H}_D(X;\Q) 
$$
is constructed as follows: 
\begin{itemize}
\item 
Choose a preshrunk tame shrinking $\Kk_{\rm sh}$ of $\Kk$, an admissible metric on $|\Kk_{\rm sh}|$, and a nested sequence of open sets $\Ww_{k+1}\subset \Ww_k\subset \bigl(|\Kk_{\rm sh}|, d\bigr)$ with $\bigcap_{k\in\N}\Ww_k = |\s_{\Kk_{\rm sh}}^{-1}(0)|$.
(These exist by Theorem~\ref{thm:K}, and e.g.\ taking $\Ww_k=B_{\frac 1k}(|\s_{\Kk_{\rm sh}}^{-1}(0)|$.)
Then equip $\Kk_{\rm sh}$ with the orientation induced from $\Kk$ by Lemma~\ref{le:cK}. 
\item
For each $k\in\N$ choose a $(\Vv_k,\Cc_k)$-adapted perturbation $\nu_k$ of $\s_{\Kk_{\rm sh}}|_{\Vv_k}^{\less\Ga}$ for some nested reductions $\Cc_k\sqsubset\Vv_k$ with $\pi_{\Kk_{\rm sh}}(\Cc_k)\subset\Ww_k$. (These exist by Remark~\ref{rmk:vicin} and Proposition~\ref{prop:ext}.)
\item
Denote by $[ |\io^{\nu_k}|_\Hh ] \in \check{H}_D(\Ww_k;\Q)$ the \v{C}ech homology classes induced by the maps
$$
|\io^{\nu_k}|_\Hh \,:\; \bigl( |\bZ^{\nu_k}|_\Hh, \La^{\nu_k}\bigr) \;\hookrightarrow \;  \Ww_k \;\subset\; \bigl(|\Kk_{\rm sh}|, d\bigr) ,
$$
take their inverse limit under pushforward with the inclusions $\Ww_{k+1}\hookrightarrow\Ww_k$, and finally take the pushforward under the homeomorphism $|\psi_{\Kk_{\rm sh}}| = \io_{\Kk_{\rm sh}}^{-1}: |\s_{\Kk_{\rm sh}}^{-1}(0)| \to X$ from Lemma~\ref{le:realization}~(iv).
\end{itemize}
\end{definition}
 
Note here that Lemma~\ref{le:realization}~(iii) identifies the quotient topology on $|\s_{\Kk_{\rm sh}}^{-1}(0)|$ with the relative topology induced by the embedding $|\s_{\Kk_{\rm sh}}^{-1}(0)|\hookrightarrow |\Kk_{\rm sh}|$. 
The latter is also identified with the metric topology given by restriction of $d$, due to the 
nesting uniqueness of Hausdorff topologies and the fact that the identity map $|\Kk|\to(|\Kk|,d)$ is continuous; see \cite[Lemma~3.1.8, Remark~3.1.15]{MW1}.
Hence there is no ambiguity of topologies in the isomorphism explained in \cite[Remark~8.2.4]{MW2} and used in the definition of $[X]^{\rm vir}_\Kk$, 
$$
\check{H}_D(|\s_{\Kk}^{-1}(0)|;\Q) \;\overset{\cong}{\longrightarrow}\; \underset{\leftarrow }\lim\, \check{H}_D(\Ww_k;\Q)  .
$$
Finally, we can prove our main theorem: The VMC/VFC are well defined and are
invariants of the oriented weak Kuranishi cobordism class.
The proof uses the same line of argument as \cite[Theorems~8.2.2, 8.2.5]{MW2}, just replacing manifolds with weighted branched manifolds. We summarize and unify these arguments here for ease of reference.
 
\begin{thm}\label{thm:VMCF}
\begin{enumerate}
\item
The virtual moduli cycle $\Zz^\Kk$ and virtual fundamental class $[X]^{\rm vir}_\Kk$ are well defined and independent of the cobordism class of oriented weak Kuranishi atlases on a fixed compact, metrizable space $X$.
\item
Let $\Kk$ be an oriented weak Kuranishi cobordism, and choose strongly adapted perturbations $\nu^\al$ in the definition of $\Zz^{\p^\al\Kk}=\bigl[  \bigl( |\bZ^{\nu^\al}|_\Hh , \La^{\nu^\al}\bigr) \bigr]$ for $\al=0,1$. Then the perturbed zero sets  $ \bigl( |\bZ^{\nu^0}|_\Hh , \La^{\nu^0}\bigr)  \sim  \bigl( |\bZ^{\nu^1}|_\Hh , \La^{\nu^1}\bigr) $ are cobordant as weighted branched manifolds, and thus $\Zz^{\p^0\Kk}=\Zz^{\p^1\Kk}$.
\item
Let $\Kk$ be an oriented weak Kuranishi cobordism of dimension $D+1$ on a compact, metrizable collared cobordism $(Y,\io_Y^0,\io_Y^1)$.
Then the virtual fundamental classes $[\p^\al Y]^{\rm vir}_{\p^\al\Kk}\simeq [\p^1 Y]^{\rm vir}_{\p^1\Kk}$ of the boundary restrictions are homologous in $Y$, 
$$
(\io_Y^0)_*\bigl([\p^0 Y]^{\rm vir}_{\p^0\Kk}\bigr)
\;=\;
(\io_Y^1)_*\bigl([\p^1 Y]^{\rm vir}_{\p^1\Kk}\bigr)
\quad \in \check{H}_D(Y;\Q) .
$$
\end{enumerate}
\end{thm}

\begin{proof}  
First note that all the necessary choices of data exist, as noted in Definition~\ref{def:VMCF}. Given such choices, Step 1 below constructs a representative of the virtual moduli cycle, and Step 5 constructs the virtual fundamental class.
To prove independence of those choices in (i), we use transitivity of the cobordism relation for compact weighted branched manifolds to prove increasing independence of choices in Steps 1--5.
Parts (ii), (iii) are then proven in Step 6.
In the following, all Kuranishi atlases will be of dimension $D$, and all cobordisms of dimension $D+1$.

\MS\NI
{\bf Step 1:} 
{\it
Fix an oriented, metric, tame Kuranishi atlas $(\Kk,d)$, nested reductions $\Cc\sqsubset\Vv$, equivariant norms $\|\cdot\|$, and constants $0<\de<\de_\Vv$, $0<\si\leq \si_{\rm rel}([0,1]\times\Vv,[0,1]\times\Cc,\|\cdot\|,\de)$. 
Then each $(\Vv,\Cc,\|\cdot\|,\de,\si)$-adapted perturbation $\nu$ induces a
$D$-dimensional weighted branched manifold $\Zz^\nu:=\bigl( |\bZ^{\nu}|_\Hh , \La^{\nu}\bigr)$ 
and a cycle $|\io^{\nu}|_\Hh: \Zz^\nu \to |\Cc|$ whose cobordism class resp.\  \v{C}ech homology class $\bigl[|\io^{\nu}|_\Hh\bigr] \in \check{H}_D(|\Cc|;\Q)$ is independent of the choice of $\nu$.
}

\MS
The regularity of the perturbed zero sets is proven in  Theorem~\ref{thm:zero}.
To prove independence of the choice of $\nu$ we consider two $(\Vv,\Cc,\|\cdot\|,\de,\si)$-adapted perturbations $\nu^0,\nu^1$. Then Proposition~\ref{prop:ext}~(iii) provides an admissible, precompact, transverse cobordism perturbation $\nu^{01}$ of
 $\s_{[0,1]\times\Kk|^{\less \Ga}_{[0,1]\times\Vv}}$ 
 with boundary restrictions $\nu^{01}|_{\{\al\}\times \Vv}=\nu^\al$ for $\al=0,1$.
Moreover, by Lemma~\ref{le:cK}~(iii) the orientation of $\Kk$ induces an orientation of $[0,1]\times\Kk$, whose restriction to the boundaries $\p^\al([0,1]\times \Kk) =\Kk$ equals the given orientation on $\Kk$. 
Now Theorem~\ref{thm:zero} implies that $\Zz:=\bigl(|\bZ^{\nu^{01}}|,\La^{\nu^{01}}\bigr)$ is a cobordism from $\p^0\Zz=\bigl(|\bZ^{\nu^0}|,\La^{\nu^{0}}\bigr)$ to $\p^1\Zz=\bigr(|\bZ^{\nu^1}|,\La^{\nu^{1}}\bigr)$ and induces a cycle $|\io^{\nu^{01}}|_\Hh: \Zz\to [0,1]\times |\Cc|$.
Finally, the boundary restrictions of this cycle prove the equality
$\bigl[|\io^{\nu^0}|_\Hh\bigr]=\bigl[|\io^{\nu^1}|_\Hh\bigr]$ in $\check{H}_D(|\Cc|;\Q)$; 
see \cite[(8.2.6)]{MW2} for the detailed homological argument.

\MS\NI
{\bf Step 2:} {\it 
Fix an oriented, metric, tame Kuranishi atlas $(\Kk,d)$ and nested reductions $\Cc\sqsubset\Vv$. Then the cobordism class of $\Zz^\nu$ as well as $\bigl[|\io^{\nu}|_\Hh\bigr] \in \check{H}_D(|\Cc|;\Q)$ are independent of the choice of strongly $(\Vv,\Cc)$-adapted perturbation~$\nu$. }

\MS
To prove this we consider two strongly $(\Vv,\Cc)$-adapted perturbations~$\nu^\al$ for $\al = 0,1$.
Thus $\nu^\al$ is $(\Vv,\Cc,\|\cdot\|^\al,\de^\al,\si^\al)$-adapted for some choices of equivariant norms $\|\cdot\|^\al$ and constants $0<\de^\al<\de_\Vv$ and  $0<\si^\al\leq \si_{\rm rel}{([0,1]\times\Vv,[0,1]\times\Cc,\|\cdot\|^\al,\de^\al)}$. 
We note that $\de:=\max(\de^0,\de^1)<\de_\Vv=\de_{[0,1]\times\Vv}$, pick equivariant norms $\|\cdot\|$ on $\Kk$ such that $\|\cdot\|^\al\leq\|\cdot\|$ for $\al=0,1$, and choose $\si \leq \min\{ \si^0,\si^1, \si_{\rm rel}([0,1]\times\Vv,[0,1]\times\Cc,\|\cdot\|,\de) \}$. 
Then Proposition~\ref{prop:ext}~(iii) provides an admissible, precompact, transverse cobordism perturbation $\nu^{01}$ of $\s_{[0,1]\times\Kk|^{\less \Ga}_{[0,1]\times\Vv}}$, 
whose restrictions $\Ti\nu^\al:=\nu^{01}|_{\{\al\}\times \Vv}$ for $\al=0,1$ are $(\Vv,\Cc,\|\cdot\|,\de,\si)$-adapted perturbations of $\s_\Kk|^{\less \Ga}_\Vv$.
Since $\de^\al\leq \de$, $\|\nu^{01}|_{\{\al\}\times \Vv}\|^\al \leq \|\nu^{01}|_{\{\al\}\times \Vv}\|<\si$ and $\si \leq \si^\al \leq \si_{\rm rel}([0,1]\times\Vv,[0,1]\times\Cc,\|\cdot\|,\de^\al)$, they are also $(\Vv,\Cc,\|\cdot\|,$ $\de^\al,\si^\al)$-adapted.
Then, as in Step~1, the perturbed zero set of $\nu^{01}$ is a cobordism from $\Zz^{\Ti\nu^{0}}$ to $\Zz^{\Ti\nu^{1}}$, and the induced cycle in $[0,1]\times|\Cc|$ shows
$\bigl[|\io^{\Ti\nu^{0}}|_\Hh\bigr]=\bigl[|\io^{\Ti\nu^{1}}|_\Hh\bigr]$ in $\check{H}_D(|\Cc|;\Q)$.

Moreover, for fixed $\al\in\{0,1\}$ both the restriction $\Ti\nu^\al=\nu^{01}|_{\{\al\}\times\Vv}$ and the given perturbation $\nu^\al$ are $(\Vv,\Cc,\|\cdot\|,\de^\al,\si^\al)$-adapted, so that Step 1 provides cobordisms $\Zz^{\nu^\al}  \sim \Zz^{\Ti\nu^\al}$ and identities 
$\bigl[|\io^{\nu^\al}|_\Hh\bigr]=\bigl[|\io^{\Ti\nu^\al}|_\Hh\bigr]$ in $\check{H}_D(|\Cc|;\Q)$. By transitivity of the cobordism relation this proves $\Zz^{\nu^0} \sim \Zz^{\nu^1}$ as claimed, and also $\bigl[|\io^{\nu^{0}}|_\Hh\bigr]=\bigl[|\io^{\nu^{1}}|_\Hh\bigr] \in \check{H}_D(|\Cc|;\Q)$.

\MS\NI
{\bf Step 3:} {\it
For a fixed oriented, metric, tame Kuranishi atlas $(\Kk,d)$, the oriented cobordism class $\Aa^{(\Kk,d)}$ of weighted branched manifolds $\Zz^\nu$ is independent of the choice of strongly adapted perturbation~$\nu$.
Moreover, given any open neighbourhood $\Ww\subset (|\Kk|,d)$ of $|\s_\Kk^{-1}(0)|$, the 
class
$A^{(\Kk,d)}_{\Ww} := \bigl[ |\io^\nu|_\Hh :\Zz^\nu \to \Ww \bigr] \in \check{H}_D(\Ww;\Q)$
is independent of the choice of strongly $(\Vv,\Cc)$-adapted perturbation $\nu$ for nested reductions $\Cc\sqsubset\Vv$ with $\pi_\Kk(\Cc)\subset\Ww$.
}

\MS
To prove this we consider two strongly $(\Vv^\al,\Cc^\al)$-adapted perturbations $\nu^\al$ with respect to nested reductions $\Cc^\al\sqsubset\Vv^\al$ with $\pi_\Kk(\Cc)\subset\Ww$, equivariant norms $\|\cdot\|^\al$, and admissible metrics $d^\al$ for $\al=0,1$.
Remark~\ref{rmk:vicin} provides a nested cobordism reduction $\Cc\sqsubset\Vv$ of $[0,1]\times \Kk$ with $\p^\al\Cc=\Cc^\al$, $\p^\al\Vv=\Vv^\al$, and $\pi_{[0,1]\times\Cc}\subset [0,1]\times\Ww$.
Now pick equivariant norms $\|\cdot\|$ on $\Kk$ such that $\|\cdot\|^\al\leq\|\cdot\|$ for $\al=0,1$, and choose $0<\de<\de_\Vv$ smaller than the collar width of $d$, $\Vv$, and $\Cc$.
Then, for any $0<\si\leq\si_{\rm rel}(\Vv,\Cc,\|\cdot\|,\de)$, Proposition~\ref{prop:ext}~(ii) provides an admissible, precompact, transverse cobordism perturbation $\nu^{01}$ of 
$\s_{[0,1]\times  \Kk|^{\less \Ga}_\Vv}$ 
whose boundary restrictions $\Ti\nu^\al:=\nu^{01}|_{\p^\al\Vv}$ for $\al=0,1$ are $(\Vv^\al,\Cc^\al,\|\cdot\|,\de,\si)$-adapted perturbations of $\s_{\Kk}|^{\less \Ga}_{\Vv^\al}$.
As before, $\Zz^{\nu^{01}}$ is an oriented cobordism from $\Zz^{\Ti\nu^{0}}$ to $\Zz^{\Ti\nu^{1}}$ and induces a cycle in $[0,1]\times\Ww$ that shows
$\bigl[|\io^{\Ti\nu^{0}}|_\Hh\bigr]=\bigl[|\io^{\Ti\nu^{1}}|_\Hh\bigr]$ in $\check{H}_D(\Ww;\Q)$.
Moreover, we can pick $\si\leq\si_{\rm rel}([0,1]\times  \Vv^\al,[0,1]\times  \Cc^\al,\|\cdot\|^\al,\de)$ 
for $\al=0,1$, so that each $\nu^{01}|_{\p^\al\Vv}$ is also strongly $(\Vv^\al,\Cc^\al)$-adapted. 
Then the claim follows by transitivity as in Step~2.

\MS\NI
{\bf Step 4:} {\it
Let $(\Kk,d)$ be an oriented, metric, tame Kuranishi atlas, and let $\Ww_k\subset (|\Kk|,d)$ be a nested sequence of open sets with $\bigcap_{k\in\N}\Ww_k = |\s_{\Kk}^{-1}(0)|$ as in Definition~\ref{def:VMCF}.
Then the \v{C}ech homology class
$A^{(\Kk,d)} := \underset{\leftarrow}\lim\, A^{(\Kk,d)}_{\Ww_k} \in \check{H}_D\bigl(|\s_{\Kk}^{-1}(0)|;\Q \bigr)$
is well defined and independent of the choice of nested sequence $(\Ww_k)_{k\in\N}$.
}

\MS
The pushforward $\check{H}_D(\Ww_{k+1};\Q)\to \check{H}_D(\Ww_k;\Q)$ by the inclusion $\Ii_{k+1}:\Ww_{k+1}\to\Ww_k$ maps $A^{(\Kk,d)}_{\Ww_{k+1}}=[|\io^{\nu_{k+1}}|_\Hh]$ to $A^{(\Kk,d)}_{\Ww_k}$ since any strongly adapted perturbation $\nu_{k+1}$ with respect to nested reductions $\Cc_{k+1}\sqsubset\Vv_{k+1}$ with $\pi_\Kk(\Cc_{k+1})\subset\Ww_{k+1}$ can also be used as strongly adapted perturbation for $A^{(\Kk,d)}_{\Ww_k}$.
This shows that the homology classes $A^{(\Kk,d)}_{\Ww_k}$ form an inverse system and thus have a well defined inverse limit.
To see that this limit is independent of the choice of nested sequence, note that the intersection $\Ww_k:=\Ww^0_k\cap \Ww^1_k$ of any two such sequences $(\Ww_k^\al)_{k\in\N}$ is another nested sequence of open sets with $\bigcap_{k\in\N}\Ww_k = |\s_{\Kk}^{-1}(0)|$.
Now choose a sequence of strongly adapted perturbations $\nu_k$ w.r.t.\ nested reductions $\Cc_k\sqsubset\Vv_k$ with $\pi_\Kk(\Cc_k)\subset\Ww_k$, then these also fit the requirements for the larger open sets $\Ww_k^\al$ and hence the inclusions $\Ww_k\hookrightarrow\Ww^\al_k$ push $[|\io^{\nu_k}|_\Hh]\in H_D(\Ww_k;\Q)$ forward to $[|\io^{\nu_k}|_\Hh]\in H_D(\Ww^\al_k;\Q)$. Hence, by the definition of the inverse limit, we have equality
$$
\underset{\leftarrow}\lim\, A^{(\Kk,d)}_{\Ww^0_k} 
\;=\; 
\underset{\leftarrow}\lim\, A^{(\Kk,d)}_{\Ww_k}
\;=\;  
\underset{\leftarrow}\lim\, A^{(\Kk,d)}_{\Ww^0_k} 
 \;\in\; \check{H}_D\bigl(|\s_{\Kk}^{-1}(0)|;\Q \bigr) .
$$

\MS\NI
{\bf Step 5:} {\it
Given an oriented weak Kuranishi atlas $\Kk$, the cobordism class $\Zz^\Kk:=\Aa^{(\Kk_{\rm sh},d)}$ of weighted branched manifolds in Step 3 and the pullback $[X]_\Kk^{\rm vir}:=|\psi_{\Kk_{\rm sh}}|_* A^{(\Kk,d)} \in \check{H}_D\bigl(X;\Q \bigr)$ of the \v{C}ech homology classes in Step 4
are independent of the choice of tame shrinking $\Kk_{\rm sh}$ of $\Kk$ and admissible metric $d$ on $|\Kk_{\rm sh}|$.
}
\MS

Here the pushforward under $|\psi_{\Kk_{\rm sh}}|$ is well defined since this is a homeomorphism by Lemma~\ref{le:realization}~(iv). Given different choices $(\Kk_{\rm sh}^\al,d^\al)$ of metric tame shrinkings of $\Kk$ and strongly adapted perturbations $\nu^\al$ resp.\ $(\nu^\al_k)_{k\in\N}$ that define $\Aa^{(\Kk_{\rm sh}^\al,d^\al)}\sim \Zz^{\nu^\al}$ resp.\ 
$A^{(\Kk_{\rm sh}^\al,d^\al)} =\underset{\leftarrow}\lim\, \bigl[ |\io^{\nu^\al_k}|_\Hh \bigr]  \in \check H_D(|\s_{\Kk_{\rm sh}^\al}^{-1}(0)|;\Q)$, we can apply Step 6 below to the cobordism $[0,1]\times \Kk$ to obtain a weighted branched cobordism from $\Zz^{\nu^0}$ to $\Zz^{\nu^1}$ and the identity 
$$
I^0_* \, \bigl([X]^{\rm vir}_{\Kk^0_{\rm sh}}\bigr)
\;=\; I^1_* \, \bigl([X]^{\rm vir}_{\Kk^1_{\rm sh}} \bigr)
\; \in\; \check{H}_D([0,1]\times  X ;\Q ) .
$$
with the natural boundary embeddings $I^\al :X\to \{\al\}\times X\subset [0,1]\times X=Y$.
Further, $I^0_* = I^1_* : \check{H}_D( X ;\Q) \to \check{H}_D( [0,1]\times  X ;\Q)$ are the same isomorphisms, because the two maps $I^0, I^1$ are both homotopy equivalences and homotopic to each other.
Hence we obtain the identity 
$[X]^{\rm vir}_{\Kk^0_{\rm sh}}  = [X]^{\rm vir}_{\Kk^1_{\rm sh}}$ in $\check{H}_D( X ; \Q)$, which
proves Step~5.

\MS\NI
{\bf Step 6:} 
{\it
Let $\Kk$ be an oriented weak Kuranishi cobordism over a compact collared cobordism $Y$. For $\al=0,1$ fix choices of preshrunk tame shrinkings $\Kk_{\rm sh}^\al$ of $\p^\al\Kk$, admissible metrics $d^\al$ on $|\p^\al\Kk|$.
Then, for any choice of strongly adapted perturbations $\nu^\al$ on $\Kk_{\rm sh}^\al$, 
there is a weighted branched cobordism $\Zz^{\nu^{01}}$ from $\Zz^{\nu^0}$ to $\Zz^{\nu^1}$. Moreover, the VFC's of the boundary components pushforward by the embeddings $\io^\al_Y:\{\al\}\times\p^\al Y \to Y$ to the same \v{C}ech homology class in $Y$, 
$(\io^0_{Y})_*\bigl([\p^0 Y]^{\rm vir}_{\p^0\Kk}\bigr)=(\io^1_{Y})_*\bigl([\p^1 Y]^{\rm vir}_{\p^1\Kk}\bigr) \in \check{H}_D(Y;\Q)$.}

\MS
First we use Theorem~\ref{thm:K} to find a preshrunk tame shrinking $\Kk_{\rm sh}$ of $\Kk$ with $\p^\al\Kk_{\rm sh}=\Kk^\al_{\rm sh}$, and an admissible metric $d$ on $|\Kk_{\rm sh}|$ with boundary restrictions $d|_{|\p^\al\Kk_{\rm sh}|}=d^\al$. If we equip $\Kk_{\rm sh}$ with the orientation induced by $\Kk$, then by Lemma~\ref{le:cK} the induced boundary orientation on $\p^\al\Kk_{\rm sh}=\Kk^\al_{\rm sh}$ agrees with that induced by shrinking from $\p^\al\Kk$.
Next, Remark~\ref{rmk:vicin}  provides nested cobordism reductions $\Cc\sqsubset \Vv$ of $\Kk_{\rm sh}$ and we may choose equivariant norms $\|\cdot\|$ on $\Kk_{\rm sh}$. Then Proposition~\ref{prop:ext} with 
$$
\si\;=\; \min\bigl\{  \si_{\rm rel}(\Vv,\Cc,\|\cdot\|,\de) ,\;  \min_{ \al=0,1} \si_{\rm rel}([0,1]\times  \p^\al\Vv, [0,1]\times  \p^\al\Cc,\p^\al\|\cdot\|,\de) \bigr\}
$$
yields an admissible, precompact, transverse cobordism perturbation $\nu^{01}$ of $\s_{\Kk_{\rm sh}}|^{\less \Ga}_\Vv$, whose restrictions $\Ti\nu^\al:=\nu^{01}|_{\p^\al\Vv}$ for $\al=0,1$ are $(\p^\al\Vv,\p^\al\Cc,\p^\al\|\cdot\|,\de,\si)$-adapted perturbations of $\s_{\Kk^\al_{\rm sh}}|^{\less \Ga}_{\p^\al\Vv}$. In particular, these are strongly adapted by the choice of $\si$, and $\Zz^{\nu^{01}}$ is a cobordism from $\Zz_{\Ti\nu^0}$ to $\Zz_{\Ti\nu^1}$. Invariance of the VMC under oriented weak Kuranishi cobordism then follows from Step 3 by transitivity of weighted branched cobordism.

To prove the identity between VFCs, we first construct a sequence of nested cobordism reductions
$\Cc_k \sqsubset \Vv$ of $\Kk_{\rm sh}$ by
$$
\Cc_k \,:=\; \Cc \cap \pi_{\Kk_{\rm sh}}^{-1}\bigl(\Ww_k) \;\sqsubset\; \Vv\qquad\text{with}\quad \Ww_k: =
B_{\frac 1k}(\io_{\Kk_{\rm sh}}(Y) \bigr) \;\subset\;|\Kk_{\rm sh}|,
$$
in addition discarding components $C_k\cap V_I$ that have empty intersection with $s_I^{-1}(0)$.
With that, Proposition~\ref{prop:ext} provides admissible, precompact, transverse cobordism perturbations $\nu_k$ with $|(\s_{\Kk_{sh}}|_\Vv^{\less\Ga} + \nu_k)^{-1}(0)| \subset \Ww_k$, and with boundary restrictions $\nu^\al_k:= \nu_k|_{\p^\al\Vv}$ that are strongly adapted perturbations of $(\Kk^\al_{\rm sh},d^\al)$ for $\al=0,1$.
Since these boundary restrictions satisfy the requirements of Step 4, they define the \v{C}ech homology classes
$A^{(\Kk^\al_{\rm sh},d^\al)} = \underset{\leftarrow}\lim\, \bigl[ |\io^{\nu^\al_k}|_\Hh \bigr] \;\in\; \check{H}_D\bigl(|\s_{\Kk^\al_{\rm sh}}^{-1}(0)|;\Q \bigr)$.

On the other hand, pushforward with the topological embeddings $J^\al: ( |\Kk^\al_{\rm sh}| , d^\al ) \to ( |\Kk_{\rm sh}|, d)$ also yields \v{C}ech homology classes $J^\al_* \bigl[ |\io^{\nu^\al_k}|_\Hh\bigr]$ that form two inverse systems in $H_D(|\Kk_{\rm sh}|;\Q)$.
Now the cycles $\io^{\nu_k}: |\Zz^{\nu_k}| \to \Ww_k$ given by Theorem~\ref{thm:zero} yield
identities 
$J^0 _* \bigl[ |\io^{\nu^0_k}|_\Hh \bigr] = J^1_* \bigl[ |\io^{\nu^1_k}|_\Hh \bigr]$ in $\check{H}_D(\Ww_k; \Q)$, 
and taking the inverse limit -- which commutes with pushforward -- we obtain
$J^0_*\,\bigl(  \underset{\leftarrow}\lim\, \bigl[ |\io^{\nu^0_k}|_\Hh \bigr] \bigr) = J^1_* \,\bigl(\underset{\leftarrow}\lim\, \bigl[ |\io^{\nu^1_k}|_\Hh  \bigr] \bigr)$
in $\check H_D( |\s_{\Kk_{\rm sh}}^{-1}(0)| ;\Q )$.
Further pushforward with $|\psi_{\Kk_{\rm sh}}|$ turns this into an equality in $\check{H}_D(Y ; \Q)$.
Finally, we use the identities
$|\psi_{\Kk_{\rm sh}}| \circ J^\al \big|_{\io_{\Kk^\al_{\rm sh}}(\p^\al Y)}
=\io^\al_Y \circ |\psi_{\Kk^\al_{\rm sh}}|$ to obtain in $\check{H}_D(Y ; \Q)$
$$
\bigl(|\psi_{\Kk_{\rm sh}}| \circ J^\al\bigr)_*\,\Bigl(  \underset{\leftarrow}\lim\, \bigl[ i^{\nu^\al_k} \bigr] \Bigr)
\;=\; 
(\io^\al_Y)_* \Bigl(  |\psi_{\Kk^\al_{\rm sh}}|_* \Bigl(  \underset{\leftarrow}\lim\, \bigl[ i^{\nu^\al_k} \bigr]\Bigr) \Bigr)
\;=\; 
(\io^\al_Y)_* \bigl[ \p^\al Y\bigr]^{\rm vir}_{\Kk^\al_{\rm sh}}.
$$
This proves Step 6 since the left hand side was shown to be independent of $\al=0,1$.
\end{proof}

\appendix
\numberwithin{theorem}{section}
\section{Groupoids and weighted branched manifolds}\label{ss:br}

The purpose of this appendix is to review the definition and properties of weighted branched manifolds from \cite{Mcbr}, and slightly generalize these notions to a cobordism theory. This will be based on the following language of groupoids.

An {\bf  \'etale groupoid} $\bG$ is a small category whose sets of objects $\Obj_\bG$ and morphisms $\Mor_\bG$ are equipped with the structure of a smooth manifold of a fixed finite dimension 
such that
\begin{itemlist}
\item all morphisms are invertible;
\item  all structural maps\footnote{
The structure maps of a category are source and target $s,t:\Mor_\bG\to\Obj_\bG$, identity $\id :\Obj_\bG\to\Mor_\bG$, and composition ${\rm comp}:\Mor_\bG \leftsub{t}{\times}_s \Mor_\bG \to \Mor_\bG$. If source and target are local diffeomorphisms, then the fiber product in the domain of composition is transverse and hence inherits a smooth structure. 
A groupoid has the additional structure map ${\rm inv}:\Mor_\bG \to \Mor_\bG$ given by the unique inverses.
} 
are local diffeomorphisms.
\end{itemlist}

\NI
All groupoids considered in this appendix are \'etale.  
Moreover, a groupoid is called
\begin{itemlist}
\item {\bf proper} if the source and target map $s\times t: \Mor_\bG\to \Obj_\bG\times \Obj_\bG$
is proper (i.e.\ preimages of compact sets are compact);
\item  {\bf nonsingular} if there is at most one morphism between any two of its objects;
\item  {\bf oriented} if its spaces of objects and morphisms 
 are oriented manifolds and if all structural maps preserve these orientations;
\item {\bf d-dimensional} if $\Obj_\bG$ and $\Mor_\bG$ are d-dimensional manifolds;
\item
{\bf compact} if its realization $|\bG|$ is compact.
\end{itemlist}

\NI
\'Etale proper groupoids are often called {\bf ep groupoids}.
It is well known that in the current finite dimensional context the properness assumption
is equivalent to  the condition that the realization $|\bG|$ is Hausdorff.\footnote{To see that proper groupoids have Hausdorff realization one can argue that the equivalence relation has closed graph and then use \cite[Ch~I,\S10,~Ex.~19]{Bourb} or \cite [Lemma~3.2.4]{MW1}.
}
Here the realization $|\bG|$ of $\bG$  is the quotient of the space of objects by the equivalence relation 
given by the morphisms, i.e.\ $x\sim y \ \Longleftrightarrow\ \Mor_\bG(x,y)\ne \emptyset$.  
It is equipped with the quotient topology, and the natural projection is denoted  $\pi_\bG: \Obj_\bG\to |\bG|$.
In general, the realization $|\bG|$ of an ep groupoid is an orbifold. It is a manifold if the groupoid is nonsingular, and an orientation of the groupoid induces an orientation of $|\bG|$.
 
Two kinds of groupoids appear in this paper:
Theorem~\ref{thm:zero} shows that the zero set of a transverse section defines a wnb groupoid (which is
 \'etale but generally not proper, and equipped with an additional weighting function, see Definition~\ref{def:brorb}).
On the other hand, each Kuranishi chart $\bK_I$ comprises two ep groupoids $\bG_{(U_I,\Ga_I)}$ and  $\bG_{(U_I\times E_I,\Ga_I)}$, which arise from group quotients as follows.

\begin{example}\label{ex:grqot}\rm 
(i)  
A group quotient $(U,\Ga)$ in the sense of Definition~\ref{def:gq} defines an ep groupoid $\bG_{(U,\Ga)}$ with $\Obj_\bG= U$, $\Mor_\bG= U\times \Ga$, $(s\times t)(u,\ga)=(u,\ga u)$, 
${\rm id}(u)=(u,\id)$, ${\rm comp}((u,\ga),(\ga u,\de)) = (u, \de\ga )$, ${\rm inv}(u,\ga)=(u,\ga^{-1})$, 
and realization $|\bG|=\qu{U}{\Ga}=\uU$.
In particular, properness is proven in Lemma~\ref{le:vep}~(i). 
This groupoid is nonsingular iff the action of $\Ga$ is free.
It is oriented if $U$ is oriented and the action of each $\ga\in \Ga$ preserves the orientation.  
 \MS
  
\NI (ii)  
The category $\bB_\Kk$ defined by a Kuranishi atlas with trivial obstruction spaces on a compact space $X$
is not a groupoid because when $I\subsetneq J$ the morphisms from $U_I$ to $U_J$ are not invertible.  
However, it is shown in \cite{Mcn} that $\bB_\Kk$ may be completed to an ep groupoid with the same realization (namely, $X$ itself) by adding appropriate inverses and composites to its set of morphisms.
 $\hfill\er$
  \end{example}

When we take restrictions of Kuranishi charts in the sense of Definition~\ref{def:restr}, this is reflected in the associated groupoids by an analogous notion:

\begin{itemlist}
\item  
If $\bG$ is an \'etale groupoid and $\uV\subset |\bG|$ is open, we define the {\bf restriction} $\bG|_{\uV}$ to be the full subcategory of $\bG$ with objects $\pi_\bG^{-1}(\uV)$.
\end{itemlist}

\NI
To discuss the theory of Kuranishi cobordisms in terms of groupoids, we need the following notions.
Here we use the notation $A^0_\eps:=[0,\eps)$ and $A^1_\eps:=(1-\eps,1]$ for 
neighbourhoods of $0,1\in [0,1]$ of size $\eps>0$ as in \cite{MW1}.

\begin{itemlist}
\item  
If $\bG$ is a groupoid and $A\subset \R$ is an interval we define the {\bf  product groupoid}
$A\times \bG$ to be the groupoid with objects $A\times \Obj_\bG$ and morphisms $A\times \Mor_\bG$, and with all structural maps given by products with $\id_A$.  

\item 
A {\bf cobordism groupoid} is a triple $(\bG, \io_\bG^0, \io_\bG^1)$ consisting of a compact proper groupoid $\bG$ and collaring functors $ \io_\bG^\al: A^\al_\eps \times \p^\al \bG \to \bG$ for $\al=0,1$.
Here $\bG$ is required to be ``\'etale with boundary'' in the sense that its object and morphism spaces are manifolds with boundary. Moreover, these boundaries form a strictly full\footnote{A subcategory is strictly full if it contains all morphisms that have source or target in its objects.} subcategory  $\p\bG$ of $\bG$ that splits, 
$\p (\Obj_\bG)= \Obj_{\p^0\bG}\sqcup\Obj_{\p^1\bG}$,  $\p(\Mor_\bG)= \Mor_{\p^0\bG}\sqcup\Mor_{\p^1\bG}$, into the disjoint union of two ep groupoids $\p^0\bG, \p^1\bG$.
Finally, the functors $\io^\al_\bG: A^\al_\eps\times \p^\al\bG\to \bG$ are defined for some $\eps>0$ and required to be tubular neighbourhood diffeomorphisms on both the sets of objects and morphisms. In particular, $\io^\al_\bG(\al, \cdot)$ is the identification between $\p^\al \bG$ and the full subcategories formed by the boundary components of $\bG$.

\item
 An {\bf oriented  cobordism groupoid} is a cobordism groupoid $(\bG, \io_\bG^0, \io_\bG^1)$ such that both $\bG$ and its boundary groupoids $\p^0\bG,\p^1\bG$ are  oriented.  
Moreover the collaring functors are required to consist of orientation preserving maps $\io_\bG^\al: A^\al_\eps\times \Obj_{\p^\al \bG}   \to \Obj_\bG$ and 
$\io_\bG^\al: A^\al_\eps\times \Mor_{\p^\al \bG}   \to \Mor_\bG$ for $\al = 0,1$, 
where products are oriented as in Remark~\ref{rmk:orientb}.
\end{itemlist}

\begin{lemma} \label{le:Hquotient}
Any topological space $Y$ has a unique {\bf maximal Hausdorff quotient} $Y_\Hh$, that is a quotient of $Y$ which is Hausdorff and satisfies the universal property: Any continuous map from $Y$ to a Hausdorff space factors through the quotient map  $\pi_\Hh: Y\to Y_\Hh$. 
\end{lemma}

\begin{proof}
To construct the maximal Hausdorff quotient  let $A$ be the set of all equivalence relations $\sim$ on $Y$ for which the quotient topology on $\qu{Y}{\sim}$ is Hausdorff. This is a set since every relation $\sim$ on $Y$ is represented by a subset of $Y\times Y$. 
Then $Y_A := \prod_{\sim\in A} \, \qu{Y}{\sim}$
is a product of Hausdorff spaces, hence Hausdorff.
The map $\pi: Y\to Y_A, y\mapsto \prod_{\sim\in A} \, [y]_\sim$ is continuous by the definition of quotient topologies.
Now the image $Y_\Hh :=\pi(Y) \subset Y_A$ with the relative topology 
is Hausdorff, and $\pi$ induces a continuous surjection ${\pi_\Hh: Y\to Y_\Hh}$. 

To check that $\pi_\Hh:Y\to Y_\Hh$ satisfies the universal property, consider a continuous map ${f:Y\to Z}$ to a Hausdorff space $Z$. This induces an equivalence relation $\sim_f$ on $Y$ given by ${x\sim_f y}$ 
$:\Leftrightarrow f(x)=f(y)$ whose quotient space $\qu{Y}{\sim_f}$ we equip with the quotient topology.
Then $f:Y\to Z$ factors as $Y\stackrel{\pi_f}\to  \qu{Y}{\sim_f}\stackrel{\io_f}\to Z$, where 
$\io_f: [y]\mapsto f(y)$ is continuous by definition of the quotient topology.
Since $\io_f$ is also injective, this implies that $\qu{Y}{\sim_f}$ is Hausdorff.  
Therefore, $\qu{Y}{\sim_f}$ is one of the factors of $Y_A$, so that
 $f:Y\to Z$ factors as the following sequence of continuous maps
 $$
 Y\stackrel{\pi_\Hh}\longrightarrow \; Y_\Hh 
 \; \stackrel{\pr_f}\longrightarrow  \;\qu{Y}{\sim_f} \;\stackrel{\io_f} \longrightarrow Z,
 $$  
where $\pr_f$ denotes the restriction to $Y_\Hh$ of the projection from $Y_A$ to its factor $\qu{Y}{\sim_f}$.
  
To  see that $Y_\Hh$ is in fact a quotient of $Y$, we will identify $Y_\Hh=\pi(Y)$ with the quotient $\qu{Y}{\sim_\pi}$ that is induced by the surjection $\pi_\Hh:Y\to Y_\Hh$. 
In this case the injection $\io_\pi : \qu{Y}{\sim_\pi}\to Y_\Hh$ is in fact a continuous bijection by continuity and surjectivity  of $\pi_\Hh$.
In particular, this implies that $\qu{Y}{\sim_\pi}$ is Hausdorff, so that 
we have a continuous map $\pr_\pi: Y_\Hh \to \qu{Y}{\sim_\pi}$ by restriction of the projection $Y_A\to \qu{Y}{\sim_\pi}$ as above. 
It is inverse to $\io_\pi$ because for $[y]\in \qu{Y}{\sim_\pi}$ we have 
$\pr_\pi \bigl( \io_\pi ([y]) \bigr) = \pr_\pi ( \pi_\Hh(y) ) 
= \pr_\pi ( \ldots \times [y] \times \ldots) = [y]$.
This identifies $Y_\Hh\cong \qu{Y}{\sim_\pi}$ as topological spaces and thus finishes the proof
that a topological space $Y_\Hh$ with the above properties exists.

To prove uniqueness, consider another Hausdorff quotient $\pr : Y \to \qu{Y}{\sim}$
that satisfies the universal property.  Then $\pr$  factors  
$Y\stackrel{\pi_\Hh}\longrightarrow \; Y_\Hh
 \; \stackrel{a}\longrightarrow  \;\ \qu{Y}{\sim} $
and by the universal property, $\pi_\Hh: Y \to Y_\Hh$  factors  $Y\stackrel{\pr}\longrightarrow \; \qu{Y}{\sim} \; \stackrel{b}\longrightarrow  \;\ Y_\Hh $.  Then $a$ is surjective since $\pr$ is.
 Moreover $a$ is injective, because otherwise there are two points $y_1, y_2\in Y$ with 
 $\pi_\Hh(y_1)\ne \pi_\Hh(y_2)$ 
 but $\pr (y_1) = a(\pi(y_1)) = a(\pi(y_2)) = \pr (y_2)$ so that $\pi(y_1) = b(\pr(y_1)) = b(\pr(y_2)) = \pi(y_2)$, a contradiction.    A similar argument shows that $b$ is bijective.  Moreover the composite
 $b^{-1} a:Y_\Hh\to Y_\Hh$ has the property that 
 $b^{-1} a\circ \pi_\Hh= \pi_\Hh$.   Since $\pi_\Hh$
  is surjective this implies that $b^{-1}a = \id$, and similarly $a^{-1}b = \id$.
 Finally, note that because both $Y_\Hh$ and $\qu{Y}{\sim}$ have the quotient topology  $a,b$ are continuous, and hence homeomorphisms.
\end{proof}

In the following we write $|\bG|$ for the realization $\Obj_{\bG}/\!\!\sim$ of an \'etale  groupoid $\bG$, and $|\bG|_\Hh$ for its maximal Hausdorff quotient.
We denote the natural maps by
$$
\pi_\bG:\Obj_{\bG}\to |\bG|,\qquad   \pi_{|\bG|}^\Hh: |\bG|\longrightarrow |\bG|_\Hh, \qquad 
\pi^\Hh_\bG:= \pi_{|\bG|}^\Hh\circ \pi_\bG : \Obj_{\bG}\to |\bG|_\Hh .
 $$
Moreover, for $U\subset \Obj_\bG$ we write $|U| := \pi_\bG(U) \subset |\bG|$ and $|U|_\Hh := \pi_\Hh(U)\subset |\bG|_\Hh$.

\begin{lemma} \label{le:Hquot2}  Let $\bG$ be an \'etale groupoid.  
 \begin{itemlist}\item[(i)]
 Any smooth functor $F:\bG\to \bG'$ induces a continuous map 
 $|F|_\Hh: |\bG|_\Hh\to |\bG'|_\Hh.$
\item[(ii)] If $A\subset \R$ is any interval, we may identify 
 $|A\times \bG|$ with  $A\times |\bG|$ and
 $|A\times \bG|_\Hh$ with  $A\times |\bG|_\Hh$.
More precisely,  
 there are commutative diagrams
\[
\xymatrix
{
\Obj_{A\times \bG} \ar[d]_{\;\;\;\;\;\; \pi_{A\times \bG}}  \ar[r]^{\;\pr_A\times \pr_{\bG}}&\;\;\;
A\times \Obj_{\bG} \ar[d]^{\id_A\times \pi_{\bG} } \\
|A\times \bG| 
\ar[r]^{\;\;\;|\pr_A|\times |\pr_{\bG}|}& \quad A\times |\bG| ,
}
\qquad 
\xymatrix
{
|A\times \bG|\quad \ar[d]_{\;\;\;\;\;\; \pi_{|A\times \bG|}^\Hh}  \ar[r]^{\;\;\;|\pr_A|\times |\pr_{\bG}|}&\quad\quad
A\times |\bG| \ar[d]^{\id_A\times \pi_{|\bG|}^\Hh } \\
|A\times \bG|_\Hh \quad
\ar[r]^{\;\;\;|\pr_A|_\Hh\times |\pr_{\bG}|_\Hh}& \quad\quad\quad A\times |\bG|_\Hh ,
}
\]
where the horizontal maps are homeomorphisms.  Here
$\pr_A: A\times \bG\to \bA$ and $\pr_\bG: A\times \bG\to \bG$ are the two projection functors from the product groupoid to its factors and $\bA$ is the groupoid with objects $A$ and only identity morphisms so that $A = |\bA| = |\bA|_\Hh$.
\end{itemlist}
\end{lemma}

\begin{proof} 
Any smooth functor $F:\bG\to \bG'$ induces a continuous map $|\bG|\stackrel{|F|}\longrightarrow |\bG'|$. Then by Lemma~\ref{le:Hquotient} applied to $|\bG|$, the composite 
 $$
|\bG|\stackrel{|F|}\longrightarrow |\bG'|  \stackrel{\pi_{|\bG'|}^\Hh} \longrightarrow |\bG'|_\Hh
 $$
factors uniquely through the quotient map $|\bG|\stackrel{\pi_{|\bG|}^\Hh}\longrightarrow |\bG|_\Hh$.
The resulting continuous map $|F|_\Hh: |\bG|_\Hh\to  |\bG'|_\Hh$ is uniquely determined by
$\pi_{|\bG'|}^\Hh\circ |F| =  |F|_\Hh\circ \pi_{|\bG|}^\Hh$.
This proves (i).

To prove (ii), first consider the diagram on the left.  The
 bottom horizontal map is bijective because $\Mor_{A\times \bG} = A\times \Mor_{\bG}$, and continuous by definition of the product topology.  Finally, it is a homeomorphism because $A$ is locally compact; c.f.\ \cite[Ex~29.11]{Mun}. 
In the diagram on the right we define the bottom horizontal arrow using the product of the maps induced as in (i) by the two functors 
$\pr_A,  \pr_{\bG}$.  Hence it is continuous.  Since the diagram commutes and we have already seen that the top horizontal map is homeomorphism, it remains to check this for the bottom  map.  But this holds because  the uniqueness property of the maximal Hausdorff quotient implies that, for any homeomorphism $\phi:Y\to Y'$, the unique continuous map $\phi_\Hh: Y_\Hh\to Y'_\Hh$ such that
$Y\stackrel{\phi}\to Y'\stackrel{\pi_{Y'}}\to Y'_\Hh$ equals $Y\stackrel{\pi_\Hh}\to Y_\Hh\stackrel{\phi_\Hh}\to Y'_\Hh$ must be a homeomorphism.
\end{proof}

The smooth structure on a weighted branched manifold  will be given by a homeomorphism to the realization of an \'etale groupoid with the following weighting structure.

 \begin{defn}[\cite{Mcbr},Def.~3.2]\label{def:brorb}
A   {\bf weighted nonsingular branched groupoid} (or {\bf wnb groupoid} for short) of dimension $d$ is a pair $(\bG,\La)$ consisting of an oriented, nonsingular, \'etale groupoid $\bG$ of dimension $d$, together with a rational weighting function $\La:|\bG|_{\Hh}\to \Q^+: = \Q\cap (0,\infty)$ that satisfies the following compatibility conditions.
For each $p\in |\bG|_{\Hh}$ there is an open neighbourhood $N\subset|\bG|_{\Hh}$ of $p$, a collection $U_1,\dots,U_\ell$ of disjoint open subsets of $(\pi_{\bG}^{\Hh})^{-1}(N)\subset \Obj_{\bG}$ (called {\bf local branches}), and a set of positive rational weights $m_1,\dots,m_\ell$ such that the following holds: 
\SSS

\NI
{\bf (Covering) } $( \pi_{|\bG|}^{\Hh})^{-1}(N) = |U_1|\cup\dots \cup |U_\ell| \subset |\bG|$;
\SSS

\NI
{\bf (Local Regularity)}  
for each $i=1,\dots,\ell$ the projection 
$\pi_{\bG}^{\Hh}|_{U_i}: U_i\to |\bG|_{\Hh}$ is a homeomorphism onto a relatively closed subset of $N$;
\SSS

\NI
{\bf (Weighting)}  
for all $q\in N$, 
the number 
$\La(q)$ is the sum of the weights of the local
branches whose image contains $q$:
$$
\La(q) = 
\underset{i:q\in |U_i|_{\Hh}}\sum m_i.
$$

A {\bf wnb cobordism groupoid} is a tuple  $(\bG, \io_\bG^0, \io_\bG^1, \La)$ in which 
$(\bG, \io_\bG^0, \io_\bG^1)$ is an oriented, nonsingular, \'etale cobordism groupoid of dimension $d$,
and $\La:|\bG|_{\Hh}\to \Q^+$ is a weighting function as above with the additional property that $\La$ and the local branches $U_1,\dots,U_\ell$ have product form in the collars.

In particular, this means that each boundary groupoid $\p^\al \bG$ is equipped with a weighting function $\La^\al$ as above such that the following diagram commutes
\[
\xymatrix
{
A^\al_\eps\times |\p^\al \bG|_\Hh\ar[d]_{\;\;\;\;\;\id_{A^\al_\eps}\times \La^\al }  \ar[r]^{\;\;\;\;\;|\io^\al_\bG|_\Hh}&
|\bG|_\Hh \ar[d]_{ \La } \\
\Q^+ 
\ar[r]^{\id}& \Q^+
}
\]
where  ${|\io_{\bG}^\al |}_\Hh$ is induced by the 
collaring functor $\io_{\bG}^\al: A^\al_\eps\times \p^\al \bG\to  \bG$ and
we identify $|A^\al_\eps\times \p^\al \bG|_\Hh$ with $ A^\al_\eps\times |\p^\al \bG|_\Hh$ as in Lemma~\ref{le:Hquot2} with orientation  as specified in Definition~\ref{def:orient}.
\end{defn}

Now we can formulate the notion of weighted branched manifold resp.\ cobordism.

\begin{defn}\label{def:brman} 
A {\bf weighted branched manifold/cobordism} of dimension $d$ is a pair $(Z, \La_Z)$ consisting of a topological space $Z$ together with a function $\La_Z:Z\to \Q^+$ and an equivalence class\footnote{
The precise notion of equivalence is given in \cite[Definition~3.12]{Mcbr}. In particular it ensures that the induced function $\La_Z: = \La_\bG\circ f^{-1}$ and the dimension of $\Obj_{\bG}$ is the same for equivalent structures $(\bG,\La_\bG, f)$. 
Moreover, if $(\bG, \io_\bG^0, \io_\bG^1)$ is a cobordism groupoid, then the images $f(|\p^\al \bG|_\Hh) : = \p^\al Z \subset Z$ of the two boundary components are well defined.
} 
of wnb (cobordism) $d$-dimensional groupoids $(\bG,\La_\bG)$ and homeomorphisms $f:|\bG|_\Hh\to Z$ that induce the function $\La_Z = \La_\bG\circ f^{-1}$.

For a weighted branched cobordism $(Z, \La_Z, [\bG,\io_\bG^0, \io_\bG^1,\La_\bG, f])$, the induced {\bf boundary components} $\p^\al Z: = f\bigl( |\io^\al_\bG|_\Hh ( |\p^\al \bG|_\Hh) \bigr)\subset Z$ 
for $\al=0,1$ are equipped with the weighted branched manifold structures $[(\p^\al\bG,\La^\al_\bG), f|_{|\p^\al \bG|_\Hh}]$.
\end{defn}

The underlying space $Z$ of a weighted branched 
manifold or cobordism is always Hausdorff due to the homeomorphism $Z\cong|\bG|_\Hh$ to a Hausdorff quotient.
Moreover, since cobordism groupoids are compact by definition, the underlying space $Z$ of a weighted branched cobordism is always compact.

It is shown in \cite[Proposition~3.5]{Mcbr} that the weighting function $\La: |\bG|_\Hh\to (0,\infty)$ is locally constant on the complement of 
the {\bf branch locus} $\Br(\bG)\subset |\bG|_\Hh$. (This 
 is defined to be the set of points in $|\bG|_\Hh$ over which $|\pi|^\Hh_{|\bG|}:|\bG|\to |\bG|_\Hh$ is not injective, and is closed and nowhere dense.) 
 Further, every point in $|\bG|_\Hh\less \Br(\bG)$  has a neighbourhood that is homeomorphic via $\pi^\Hh_{|\bG|}$ to an open subset in a local branch and so has the structure of a smooth oriented manifold.

 \begin{example}\label{ex:wnb} \rm 
(i) 
Any compact oriented smooth manifold/cobordism may be considered as  a weighted branched 
manifold/cobordism with weighting function $\La_Z\equiv 1$ and empty branch locus.  
\MS

\NI
(ii)
A compact weighted branched manifold of dimension $0$ also necessarily has empty branch locus and consists of a finite set of points $\{p_1,\ldots,p_k\}$, each with a positive rational weight $m(p_i)\in \Q^+$  and orientation $\mathfrak o(p_i)\in \{\pm\}$. 
Any representing groupoid $\bG$ has as object space $\Obj_{\bG}$  a set with the discrete topology, that is equipped with an orientation function $\orr: \Obj_{\bG}\to \{\pm\}$.  The morphism space $\Mor_{\bG}$ is also a discrete set and, because we assume that $\bG$ is oriented, 
defines an equivalence relation on $\Obj_{\bG}$ such that  $x\sim y\Longrightarrow \orr(x) = \orr(y)$.  
Moreover, because  $|\bG|$ is  Hausdorff, we can identify $ |\bG|= |\bG|_\Hh$ and hence conclude that $\Obj_{\bG}$ consists of precisely $k$ classes of points that are equivalent under $\Mor_{\bG}$ and project to $p_1,\ldots,p_k$ in $Z\cong|\bG|_\Hh$.
\MS

\NI
(iii)
For the prototypical example of a $1$-dimensional weighted branched cobordism $(|\bG|_\Hh,\La)$, take $\Obj(\bG)=I\sqcup I'$ equal to two copies of the interval $I=I'=[0,1]$ with nonidentity morphisms from $x\in I$ to $x\in I'$ for $x\in [0,\frac 12)$ and their inverses, where we suppose that $I$ is oriented in the standard way. Then the realization and its Hausdorff quotient are
$$
\begin{array}{cl}
 |\bG| &= \;\qq{I\sqcup I'}{\bigl\{(I,x)\sim (I',x) \; \text{iff}\;  x\in [0,\tfrac 12)\bigr\}},\\
 |\bG|_\Hh& = \;\qq{I\sqcup I'}{\bigl\{(I,x)\sim (I',x) \; \text{iff}\;  x\in [0,\tfrac 12]\bigr\}},
\end{array}
$$
and the branch locus is a single point $\Br(\bG)=
\bigl\{[I,\frac 12]=[I',\frac 12]\bigr\}
\subset |\bG|_\Hh$.
The choice of weights $m, m'>0$ on the two local branches $I$ and $I'$ determines the weighting function $\La: |\bG|_\Hh \to (0,\infty)$ as
$$
\La([I,x])
 = \left\{
 \begin{array}{ll} 
 m+m'  & \mbox{ if }\;\; x\in [0,\frac 12],  \vspace{.1in}\\
m &\mbox{ if }\;\;  x\in (\frac 12,1],
\end{array}\right. 
\qquad
\La([I', x]) 
 = \left\{
 \begin{array}{ll} 
m+m'  & \mbox{ if }\;\; x\in [0,\frac 12],  \vspace{.1in}\\
m'  &\mbox{ if }\;\;  x\in (\frac 12,1].   
\end{array}\right.
$$
For example, giving each branch $I, I'$ the weight $m=m'=\frac 12$, together with an appropriate choice of collar functors $\io^\al_\bG$, yields a weighted branched cobordism $(|\bG|_\Hh, \io_\bG^0, \io_\bG^1, \La)$ with
$|\p^0\bG|_\Hh=\bigl\{[I,0]=[I',0]\bigr\}$,
which is a single point with weight $1$, and $|\p^1\bG|_\Hh= \bigl\{[I,1],[I',1]\bigr\}$, 
which consists of two points with weight $\tfrac 12$, all with positive orientation 
because as explained in Remark~\ref{rmk:orientb} the induced  orientation on the boundary $\p^\al \bG$  of a cobordism
is completed to an orientation of the collar by adding as the first component the positive unit vector along $A^\al_\eps$.

Another choice of collar functors for the same weighted groupoid $(\bG,\La)$ 
might give rise to a different partition of the boundary into incoming $\p^0\bG$ and outgoing $\p^1\bG$, for example yielding a weighted branched cobordism with $|\p^0\bG|_\Hh=\bigl\{[I,0]=[I',0] , [I,1]\bigr\}$, consisting of two points with weights and orientations $(1,+)$ and $(\frac 12, -)$ 
and $|\p^1\bG|_\Hh=\bigl\{[I',1]\bigr\}$, consisting of one point with weight $(\tfrac 12,+)$.

\MS
\NI
(iv)
In the situation of Theorem~\ref{thm:zero}, the nonsingular \'etale groupoid $\Hat\bZ^\nu$ with $\Obj_{\Hat\bZ^\nu}=(\s_\Kk|_{\Vv} + \nu)^{-1}(0)$ has a maximal Hausdorff quotient $|\Hat\bZ^\nu|_\Hh=|\Hat\bZ^\nu_\Hh|$ that,
as we show in Lemma~\ref{le:zero2}, is given by the realization of the groupoid $\Hat\bZ^\nu_\Hh$ obtained as in (iii) above by closing the set of morphisms $\Mor_{\Hat\bZ^\nu}\subset \Obj_{\Hat\bZ^\nu}\times \Obj_{\Hat\bZ^\nu}$.
Therefore, in this case we can give a completely explicit description of $|\bZ|_\Hh$ and its weighting function $\La_\bZ$; see the proof of Theorem~\ref{thm:zero}.
$\hfill\er$
\end{example}

The following is a version of some parts of  \cite[Proposition~3.25]{Mcbr}, 
which more generally defines a notion of integration over weighted branched manifolds and cobordisms.

\begin{prop}\label{prop:fclass}  
Any compact $d$-dimensional weighted branched manifold $(Y, \La_Y)$ induces a {\bf fundamental class} $[Y]\in H_d(Y;\Q)$, and any 
$d$-dimensional weighted branched cobordism $(Z, \La_Z)$ with boundary
$\p Z:= \p^0 Z \cup \p^1 Z$ induces a {\bf fundamental class} $[Z]\in H_d(Z,\p Z;\Q)$, whose image under the boundary map 
$$
\partial: H_d(Z,\p Z;\Q)\to H_{d-1}(\p Z;\Q)\cong H_{d-1}(\p^0 Z;\Q)+ H_{d-1}(\p^1 Z;\Q)
$$
 is $\p [Z] = [\p^1 Z] - [\p^0 Z]$.
\end{prop}
\begin{proof}  
If $(Y, \La_Y)$ has a weighted branched manifold structure $(\bG,\La_\bG)$ with well behaved (e.g.\ piecewise smooth) branch locus,
then one can triangulate $|\bG|_\Hh\cong Y$ so that  the branch locus lies in the codimension-$1$ skeleton. 
We may then define a singular cycle on $Y$ by using  the local weights $m_i$ to assign a rational weight to each top dimensional simplex. 
As explained in Remark~\ref{rmk:orientb}, in the case of a cobordism $Z$ the induced orientation on the  boundary component $\p^\al Z$ is completed to the orientation of the collar by adding the unit positive vector along the collar
as the first component.  In the case of $\p^0 Z$ this yields an  orientation of $\p^0 Z$  that is the opposite of the standard way of orienting a boundary component by adding the outward pointing normal, a fact that is reflected in the minus sign in the formula 
$\p [Z] = [\p^1 Z] - [\p^0 Z]$.  
For more details and the general case, see \cite{Mcbr}. 
\end{proof}

\bibliographystyle{alpha}

\end{document}